\documentclass[11pt]{amsart}
\usepackage{amssymb,amsmath,amsthm,mathtools}
\usepackage{a4wide}
\usepackage{graphicx}
\usepackage[utf8]{inputenc} 
\usepackage{microtype}
\usepackage{hyperref}
\usepackage{amsfonts} 
\usepackage{latexsym}
\usepackage[font=small,format=hang,labelfont={sf,bf}]{caption}
\usepackage{epsfig}
\usepackage{subfig}
\usepackage{url}
\usepackage{varioref}
\usepackage{bm}
\mathtoolsset{showonlyrefs} 
\usepackage{tikz}
\usepgflibrary{patterns} 
\usepgflibrary[patterns] 
\usetikzlibrary{patterns} 
\usetikzlibrary[patterns] 
\usepackage{amsfonts} 
\usepackage{geometry}
\usepackage{color}
\usepackage{xcolor}
\usepackage{mathrsfs}

\usepackage{latexsym}
\usepackage{bbm}
\usepackage{dsfont}
\numberwithin{equation}{section}
\theoremstyle{plain}
\begingroup
\newtheorem{theorem}{Theorem}[section]

\newtheorem{proposition}[theorem]{Proposition}

\endgroup

\theoremstyle{definition}
\begingroup
\newtheorem{definition}[theorem]{Definition}
\newtheorem{remark}[theorem]{Remark}

\endgroup

\definecolor{ddorange}{rgb}{1,0.5,0}
\definecolor{ddcyan}{rgb}{0,0.2,1.0}

\newcommand{\e}{\varepsilon}

\newcommand{\mrestr}{%
	\,\raisebox{-.127ex}{\reflectbox{\rotatebox[origin=br]{-90}{$\lnot$}}}\,%
}
\newcommand{\numberset}{\mathbb}
\newcommand{\R}{\numberset{R}}

\newcommand{\N}{\numberset{N}}
\newcommand{\Z}{\numberset{Z}}

\title[Discrete and Continuum Area-Preserving Mean-Curvature Flow of Rectangles]{Discrete and Continuum Area-Preserving Mean-Curvature Flow of Rectangles}

\author[M. Cicalese]
{M. Cicalese}
\address[Marco Cicalese]{
	Zentrum Mathematik - M7, Technische Universitat M\"unchen, Boltzmannstrasse 3, 85748 Garching, Germany	
}
\email[M. Cicalese]{cicalese@ma.tum.de}

\author[A. Kubin]
{A. kubin}
\address[Andrea Kubin]{
	Zentrum Mathematik - M7, Technische Universitat M\"unchen, Boltzmannstrasse 3, 85748 Garching, Germany	
}
\email[A. Kubin]{andrea.kubin@tum.de}

\begin{document}

	\begin{abstract}	
We investigate the area-preserving mean-curvature-type motion of a two-dimensional lattice crystal obtained by coupling constrained minimizing movements scheme introduced by Almgren, Taylor and Wang in \cite{ATW} with a discrete-to-continuous analysis. We first examine the continuum counterpart of the model and establish the existence and uniqueness of the flat flow, originating from a rectangle. Additionally, we characterize the governing system of ordinary differential equations. Subsequently, in the atomistic setting, we identify geometric properties of the discrete-in-time flow and describe the governing system of finite-difference inclusions. Finally, in the limit where both spatial and time scales vanish at the same rate, we prove that a discrete-to-continuum evolution is expressed through a system of differential inclusions which does never reduce to a system of ODEs.
	\end{abstract}

\maketitle


	\tableofcontents

	\section*{Introduction}
	
Comprehending the interplay between the energetic and kinematic aspects of crystal deformation over time is a formidable task, presenting significant challenges in the quest for simplified models that can capture the essential features of this intricate problem. The complexity is compounded by the emergence of geometrically defected microstructures, microscopic atomic arrangements whose flow deviate from the simple geometric description of the motion observed at mesoscopic or continuum scales. Establishing meaningful connections between microscopic and macroscopic observables in this context proves to be a big challenge. \\
In this paper we provide a simple geometric answer to the problem starting from a well-known microscopic lattice model and considering at several spatial and temporal scales the motion of a two-dimensional crystal which preserves its area along the flow. We first investigate the problem at a continuum level; i.e., when the crystal is modelled as a continuum. In this setting the evolution of the crystal is obtained implementing a time-discrete iterative minimization scheme proposed by Almgren, Taylor and Wang in \cite{ATW} and by Luckhaus and Sturzenhecker in \cite{LS}, here modified to include the conservation of area constraint. In this case, passing to the limit as the time step vanishes, we find a system of forced mean-curvature equations governing the evolution of the crystal. 
Drawing from our experience in addressing the continuum problem in the second part of the paper we turn our attention to the atomistic setting. In this case the crystal is modelled as a finite union of point masses occupying the position of a subset of the square lattice and the minimizing movement scheme above needs to be adapted to this framework (see \cite{BGN} for the unconstrained problem and \cite{BS} for an introduction to the geometric motion of lattice systems). In this framework we first prove several geometric properties of the solution of the iterative scheme when the number of particles is kept fixed in the time-discrete iteration. Then, we provide existence of the flow and characterize it in terms of a system of finite difference inclusions when the lattice spacing and the time step vanish at the same rate. We finally conjecture the discrete-to-continuum evolution of the crystal and propose a regularized version of the flow in which case the convergence to an explicit system of ordinary differential inclusions can be proved. We emphasize that while most of our arguments could be extended to higher dimensions or to more general lattices, albeit at the cost of significantly increased computational complexity, for the sake of readers' convenience, we limit our consideration to the two-dimensional case.\\

The simplest geometric model of a two-dimensional crystal assumes its atoms to occupy some of the nodes of a simple periodic lattice that we fix to be $\Z^2$. The optimal shape of the crystal is then obtained by minimizing an energy depending on the position of the atoms in the lattice. Energies with geometric flavours have been proposed by Cahn and coauthors in several papers (see for instance \cite{Taylor1, Taylor2}). The simplest energy functional one can think about is a perimeter-like energy obtained by referring to its ground states a classical nearest-neighbors ferromagnetic Ising energy (see \cite{ABC, CDL} and refer to \cite{ABCS} for a more comprehensive variational treatment of spin systems). The latter can be written as
\begin{equation}\label{intro-en}
P_\e(u)=\frac14\sum_{n.n.}\e(u^i-u^j)^2,
\end{equation}
where $u:i\in\e\Z^2\mapsto u^i\in\{-1,+1\}$ is an Ising variable and the sum is extended to all nearest neighboring $(n.n.)$ points of the lattice $\e\Z^2$, that is to those $i,j\in\e\Z^2$ such that $|i-j|=\e$. In this model the atoms of the crystal are those lattice points $i\in\e\Z^2$ where $u^i=1$ and the region occupied by the crystal can be thought of as the union of those elementary cells of the dual lattice of $\e\Z^2$ centred at those points, namely $E:=\bigcup_{\{i\in{\e\Z^2}:\,u^i=1\}}(i+\e[-1/2, 1/2)^2)$. In the rest of this introduction we generically refer to $E$ as to a $\e\Z^2$ crystal. With this identification between functions and sets the energy in \eqref{intro-en} (also known as {\it edge perimeter} energy) can be regarded as defined on sets of finite perimeter. In this sense the energy of the $\e\Z^2$ crystal $E$ is defined by $P_\e(E):=P_\e(u)$ and one can carry out the variational coarse-graining of such energies as $\e\to 0$ in terms of $\Gamma$-convergence with respect to the $L^1$ distance between sets. In \cite{ABC} (see also \cite{ACR, BC, BP, BSpenrose} for generalizations covering also the case of non-periodic lattices) it has been proved that, as $\e\to 0$, the  $\Gamma$-limit of the functionals $P_\e$ is the anisotropic perimeter functional defined on sets of finite perimeter as
\begin{equation*}
\mathrm{P}_{|\cdot|_1}(E):= \int_{\partial^* E} |\nu_E|_1 d \mathcal{H}^{1},
\end{equation*}
where for $\nu=(\nu_1,\nu_2)$, $|\nu|_1:=|\nu_1|+|\nu_2|$ is the $1$-norm of $\nu$. Note that in the functional above the symmetry properties of the anisotropic surface tension $|\cdot|_1$ reflects the effect of geometric symmetries of the lattice $\Z^2$ on the continuum limit. \\


\noindent {\bf The continuum setting}\\
According to the preceding discussion, within the continuum framework, the area-preserving evolution by mean curvature of a crystal will be modelled as the solution of the following iterative scheme proposed in \cite{ATW, LS}, here modified to enforce the area constraint at each minimization step (see \cite{BCCN} for a similar result using the MBO scheme in the case of convex sets).  Let us assume that the initial set is a rectangle, say $R$, of unitary area $|R|=1$ and let us fix a time step $\tau>0$. We introduce the functional $ \mathcal{D}(E,\,F):= \int_{E \Delta F} d_{\infty}(z,\, \partial F) \, dz$, where $d_\infty(z,\,\partial F):= \inf \left\{\vert z-\hat{z} \vert_{\infty} \colon \,\hat{z} \in \partial F\right\}$ and $|z|_{\infty}=\max\{|z_1|,|z_2|\}$ for $z=(z_1,z_2)$. The functional $ \mathcal{D}(E,\,F)$ is interpreted as the energy cost dissipated for the crystal to evolve from the shape $E$ to the shape $F$. According to Definition \ref{12092023def1}, an approximate ${\mathrm P}_{|\cdot|_1}$ area-preserving mean-curvature flow consists of a family of sets $\{E_{t}^{(\tau)}\}_{t\geq 0}$, where
$ E_{t}^{(\tau)}:= E_{k \tau}^{(\tau)} \, \text{for any } t \in [k \tau,\, (k+1)\tau)$ and where $\{ E_{k\tau}^{(\tau)}\}_{k \in \mathbb{N}}$ are obtained by solving the following iterative scheme 
\begin{equation}\label{introscheme}
E_0^{(\tau)}=R, \qquad
E_{k \tau}^{(\tau)}\in\underset{E \subset \mathbb{R}^2,\, \vert E \vert=1}{\arg \min} \left\{\mathrm{P}_{\vert \cdot \vert_1}(E)+\frac{1}{\tau} \mathcal{D}(E,\,E_{(k-1)\tau}^{(\tau)}) \right\},\quad  k\geq 1. 
\end{equation}
The flat $P_{|\cdot|_1}$ area-preserving mean-curvature flow is then obtained passing to the limit as $\tau\to 0$ (see Definition \ref{defflat}).
We emphasize that our iterative scheme requires the area to remain constant at each iteration step. Hence, the discrete evolution is area preserving, a requirement that translates the idea that the time step $\tau$ holds a physical meaning. Specifically, it  is larger or equal than  the timescale required for energy relaxation to occur. If instead the time-discrete evolution were merely regarded as an approximation of the time-continuum one, we could relax the area constraint and make it active only in the limit as $\tau\to 0$ as it is the case in \cite{MSS} (see also \cite{JMPS, MPS}). As it will become evident later, keeping the area constraint at each time step significantly complicates our analysis in the continuum setting and even more in the lattice setting. This complexity is also the primary reason why we confine our investigation to initial sets with rectangular shapes, leaving the case of more general initial shapes for future investigations. In Theorem \ref{thm12092023} we prove that for every $R \subset \R^2$ rectangle with $\vert R \vert=1$ there exists a unique area-preserving $P_{|\cdot|_1}$ flat flow which is given by a family of rectangles $R(t)=[-\frac{a(t)}{2},\frac{a(t)}{2}]\times[-\frac{b(t)}{2},\frac{b(t)}{2}]$ such that $R(0)=R$. Moreover their side lengths $a$ and $b$ solve the following system of ODEs:
\begin{equation}\begin{cases}\displaystyle
 \frac{d}{d t} a(t)= -\frac{4}{b(t)}+\frac{8}{a(t)+b(t)},&\\ \\\displaystyle\frac{d}{d t} b(t)= -\frac{4}{ a(t)}+\frac{8}{a(t)+b(t)}. &
 \end{cases}
\end{equation}
As a byproduct of that, in Remark \ref{remexp} we show that $R(t)$ converge exponentially fast to the square $[-\frac{1}{2},\frac{1}{2}]\times[-\frac{1}{2},\frac{1}{2}]$. The proof of Theorem \ref{thm12092023} is obtained combining several results. In Proposition \ref{steinersim} and in Proposition \ref{decresdisp} we prove that both perimeter and dissipation decrease under Steiner symmetrization in the coordinate directions. Combining this information with the specific geometry of symmetrized sets, in Theorem \ref{thm-equiv} we prove that in \ref{introscheme} one can reduce the class of competitors to rectangles with fixed baricenter. As a consequence, we eventually find the solution of the approximate flat flow. Making use of uniform Lipschitz estimates obtained in the step-by-step minimization on the side lengths of the evolving rectangles, we eventually conclude the proof of Theorem \ref{thm12092023} taking the limit as the time step vanishes.\\

\noindent {\bf The discrete setting}\\
In the second part of the paper, we examine the evolution of $\varepsilon \mathbb{Z}^2$ crystals. In the case when the area is not preserved, Braides, Gelli and Novaga in \cite{BGN} were the first to combine a variational discrete iterative scheme as the one in \eqref{introscheme} with a discrete-to-continuum procedure to obtain a corse-grained version of the flow (see also \cite{BCY, BMN, BST, BS, CDGM, MN} for similar results in this context). In the present case we are interested to the area-preserving evolution of $\varepsilon \mathbb{Z}^2$ crystals. We denote the class of such crystals with unitary area by $\mathcal{AD}_{\varepsilon}$ and determine their evolution by solving an iterative scheme similar to the one presented in \eqref{introscheme}. Interestingly, due to potential incompatibilities between the lattice spacing and the area constraint, the shape of the solution of the minimization problem may not be a rectangle in successive iterations. Instead, it typically transforms into a quasi-rectangle (see Definition \eqref{quasirettangolo}). Therefore, the problem we address in the discrete setting is formulated as follows: given a quasi-rectangle $QR_\varepsilon$, solve the iterative minimization problem:
\begin{equation}\label{introscheme-eps}
E_0^{(\e, \tau)}=QR_\e, \qquad
E_{k \tau}^{(\e,\tau)}\in{\arg \min} \left\{\mathrm{P}_{\e}(E)+\frac{1}{\tau} D_\e(E,\,E_{(k-1)\tau}^{(\e,\tau)}), E\in{\mathcal AD}_\e \right\},\quad  k\geq 1,.
\end{equation}
Here, the discrete dissipation $D_\varepsilon$ is introduced in \eqref{dissipazioneD_epsilon} and it is a slight variation of the one introduced in the previous paragraph, more suited for the computation needed for $\varepsilon \mathbb{Z}^2$ crystals. Following an approach similar to the one used in the continuum setting, in Section \ref{Section:Steiner-eps}, we introduce a discrete Steiner-like symmetrization technique and use it to prove that if $QR_\varepsilon$ satisfies an additional symmetry condition, referred to as pseudo-axial symmetry (see Definition \eqref{paquasirettangolo}), then the solution of the minimization problem in \eqref{introscheme-eps} also shares the same symmetry property. As a consequence, without loss of generality, one can restrict the class of competitors in \eqref{introscheme-eps} to that of symmetric quasi rectangles. However, within this class of shapes, the process of sending $\varepsilon$ and $\tau$ to zero and characterizing a coarse-grained continuum flow proves to be an exceedingly challenging endeavor, primarily due to the too many degrees of freedom involved. To have an idea of the difficulties arising in this situation, it is worth pointing out that at each step of the iterative scheme one needs to optimize a fourth order polynomial equation (in the unconstrained case considered in \cite{BGN} the polynomial is of second order) whose coefficients depend on the shape of the crystal at the previous step. The optimization necessarily goes through a very long and tedious case-by-case study that we have decided to simplify by reducing the class of competitors as much as possible while still keeping the main features of the general flow. Within this class we consider the following coarse-grained procedure. We take a family $QR_{\varepsilon}$ of quasirectangles with pseudo-axial symmetry and assume that, as $ \varepsilon \rightarrow 0$ $QR_\e$ converge in the Hausdorff distance to the rectangle $[-\frac{a}{2}, \frac{a}{2}]\times [-\frac{b}{2}, \frac{b}{2}]$. We want to prove existence of the flow and obtain some geometric information, and possibly even the governing equations of the flow $E_t^{\varepsilon,\tau}$ with initial datum $QR_\e$ as both $\e$ and $\tau$ vanish. As explained in \cite{BS}, in this kind of problems the interaction between time and space discretization parameters plays a crucial role, the limit motion depending strongly on their relative rate of convergence. If $ \varepsilon \ll \tau$ we expect that, as already observed in \cite{BGN}, the limit flow is identified by first letting $ \varepsilon \rightarrow 0$ and then taking the limit as $\tau\to 0$. At a heuristic level for fixed $\tau$, $P_\varepsilon$ would be substituted by the anisotropic perimeter $\mathrm{P}_{\vert \cdot \vert_1}$ and $D_{\varepsilon}$ by $ \mathcal{D}$. The flat flow would then correspond to the one studied in the continuum setting.  If $\varepsilon\gg \tau$ we expect no motion, namely $E_k^{\varepsilon,\tau}= QR_{\varepsilon}$ at every step $k$. Indeed, a heuristic argument shows that for $E\neq QR_{\varepsilon}$ and for $\tau$ sufficiently small we have
\begin{equation*}\frac{1}{\tau}D_{\varepsilon}(E, QR_{\varepsilon})= \frac{1}{\tau}\int_{E \Delta QR_{\varepsilon}} d^{\varepsilon}_{\infty}((x,y), \, \partial QR_{\varepsilon})\,dxdy \geq c \frac{\varepsilon }{\tau} \geq \mathrm{P}(QR_{\varepsilon}),
\end{equation*}
which implies that the limit motion is the constant set $[-\frac{a}{2}, \frac{a}{2}]\times [-\frac{b}{2}, \frac{b}{2}]$. The previous two arguments suggest the most interesting regime to be $ \tau=\alpha\e$, for some $\alpha>0$. This is the case we focus on in the last part of the paper. More precisely, we compute the minimizer of the incremental problem in \eqref{introscheme-eps} for $E$ belongin to the special class $\mathcal{SQR}_{\varepsilon}(QR)$ defined in \eqref{defSQR(QR)} and describe the asymptotic behavior, as $ \varepsilon \rightarrow 0$, of the approximate flat solution of the area-preserving mean-curvature flow in the lattice $\varepsilon \mathbb{Z}^2$ within this subclass. In Theorem \ref{theorem:approxflat-eps} we prove the existence of the time-discrete flat flow and characterize it in terms of a system of finite difference inclusions. At this stage of our investigation, even if one can prove the existence of of a continuous-in-time flat flow (again given by a family of rectangles), the characterization of the limit equations is not clear. We conjecture the limit equations in Remark \ref{conjectured-flow}. We conclude the paper by introducing a regularized version of the flow which we call rectangular flow and that at each step of the evolution is $O(\e)$ close to the previous one. For such a flow, which however preserves the volume only in the continuum limit, we can write down explicitly the  equation of motion in the form of a system of differential inclusions. It is worth observing (see Remark \ref{ultimopomeriggio}) that, unlike the unconstrained case considered in \cite{BGN}, here the system of differential inclusions that characterizes the regularized flow never reduces to a single system of ODEs. As a result, the pinning phenomenon, that is when the flow is given by the family of sets identically equal to the initial datum, can only be one of the possible motions. In Remark \ref{ultimopomeriggio} we explicitly write the condition on $\alpha,\, a$ and $b$ under which the phenomenon occurs. 
\section{Notation}
We use the convention that the set $ \N$ contains the zero.
	Let $\mathbb{R}^2$ be the Euclidean plane and let $\{e_1,e_2\}$ denote the canonical basis of $\mathbb{R}^2$. Given $(x,y)\in\R^2$, by $\vert (x,y) \vert_1= \vert x \vert + \vert y \vert$ and $\vert (x,y) \vert_{\infty}=\max \{ \vert x \vert,\, \vert y \vert\}$ we denote its $1-$ and the $\infty-$ norms, respectively. The set $  \mathcal{S}^1$ stands for the unit circle in $\R^2$. Given $\eta \in \mathcal{S}^1$, we denote by $ \eta^{\bot} \in \mathcal{S}^1$ the vector orthogonal to $\eta$ such that $ \{\eta,\, \eta^{\bot}\}$ has the same orientation as $ \{ e_1,e_2\}$. For every $\eta \in \mathcal{S}^1$ we denote by $\mathbf{p}_{\eta}\colon \mathbb{R}^2 \rightarrow \mathbb{R}^2$ the projection on $\eta$. In the case $\eta=e_1,e_2$ we also use the notation $\Pi_{i}$ in place of $\mathbf{p}_{e_i}$.
We denote the n-dimensional Lebesgue measure of a measurable set $E\subset\R^n$ by $\mathcal{L}^n(E)$. In the case $n=2$ we also use the notation $|E|:=\mathcal{L}^2(E)$.
	 Given a measurable set $E \subset \mathbb{R}^2$ we denote its barycenter as $\mathrm{Bar}(E):= \int_{E} x \,d x.$
	The symmetric difference of $E$ and $F$ is denoted $ E \triangle F$, their Hausdorff distance by $ d_{\mathcal{H}}(E,\,F)$.
	If $E$ is a set of finite perimeter we denote by $\partial^* E$ its reduced boundary. The perimeter of $E$ is denoted by $\mathrm{P}(E)= \mathcal{H}^1(\partial^* E)$ where $\mathcal{H}^1$ stands for the one dimensional Hausdorff measure. For all $x \in \partial^* E $ we denote by $\nu_E(x)$ the measure theoretic outer normal vector field to $E$ at the point $x$. We denote by $ \left[\mathcal{M}(\R^2)\right]^m$ the space of the $\R^m$-valued Radon measure on $\R^2$. If $ \mu \in \left[\mathcal{M}(\R^2)\right]^m$ we set $\vert \mu \vert$ the total variation measure associated to $\mu$. We denote by $BV(\R^2)$ the space of functions with bounded variation in $\R^2$. If $F \subset \mathbb{R}^2$ we denote its complement $F^c= \mathbb{R}^2 \setminus F$. If $n,\, m \in \mathbb{Z}$ we write $n \equiv_{2} m$ if $n, \, m$ have the same remainder with respect to the Euclidean division by $2$. Given $x \rightarrow g(x)$ a real valued function, we denote by $O(g)$ the class of all functions $ x \rightarrow f(x)$ such that for all $x$: $\vert f(x) \vert \leq C \vert g(x) \vert $ for some constant $C>0$.
	Finally we denote by $C(\star,\cdots,\star)$ a constant that depends on $\star,\cdots,\star$; this constant can change from line to line. 

	\section{The continuum setting}\label{Sec1}
	In this section we explain the area-preserving crystalline mean-curvature flow of a rectangle in the Euclidean space $\mathbb{R}^2$. We introduce a notion of global
	flat solution to the area-preserving crystalline mean-curvature flow which is based on the definition introduced independently by Almgren, Taylor \& Wang in \cite{ATW} and by Luckhaus \& Sturzenhecker in \cite{LS}. 
	
	Given $E, F \subset \R^2$ of finite perimeter we define
	\begin{equation}\label{energiaincrementale}
		\mathcal{F}_{\tau}(E,\,F):= \mathrm{P}_{\vert \cdot \vert_1}(E)+ \frac{1}{\tau} \int_{E \Delta F} d_{\infty}(z,\, \partial F) \,dz.
	\end{equation}
	In the previous formula the functional $\mathrm{P}_{\vert \cdot \vert_1}$ is the anisotropic perimeter defined for any set $E \subset \mathbb{R}^2$ of finite perimeter as
	$$  \mathrm{P}_{\vert \cdot \vert_1}(E):= \int_{\partial^* E} \vert \nu_E(z) \vert_1 d \mathcal{H}^1(z),$$ 
	or equivalently as
	\begin{equation}\label{per1divformula}
		\mathrm{P}_{\vert\cdot \vert_1}(E)= \sup \left\{\int_{E} \mathrm{div}(T)(z)dz\colon \, T \in C^1_c(\R^2,\R^2), \, \vert T \vert_{\infty}(z) \leq 1, \, \forall z \in \R^2\right\}.
	\end{equation}
In what follows we will make use of the following notation: given an open set $\Omega \subset \R^2$, for all set of finite perimeter $E \subset \R^2$  we denote the relative perimeter of $E$ in $\Omega$ by
$$\mathrm{P}_{\vert \cdot \vert_1}(E, \Omega)= \int_{\partial^* E \cap \Omega } \vert \nu_E(z) \vert_1 d \mathcal{H}^1(z).$$
In \eqref{energiaincrementale} $ \mathcal{D}(E,\,F)$ denotes the dissipation of  the
	 two sets $E\,, F\subset \mathbb{R}^2$, that we define as
	$$ \mathcal{D}(E,\,F):= \int_{E \Delta F} d_{\infty}(z,\, \partial F) \, dz$$
	where 
	$$d_\infty(z,\,\partial F)= \inf \left\{\vert z-\hat{z} \vert_{\infty}	\colon \,\hat{z} \in \partial F\right\}.$$
	\subsection{Approximate flat $\mathrm{P}_{\vert \cdot \vert_1}$ area-preserving mean-curvature flow}
	In this subsection we compute the approximate flat solution of the area-preserving crystalline mean-curvature flow with initial datum $R_0$ a rectangle of unitary area.
	\begin{definition}[Approximate flat $\mathrm{P}_{\vert \cdot \vert_1}$ area-preserving mean-curvature flow] \label{12092023def1}
		Let $R_0 \subset \mathbb{R}^2$ be a rectangle with $\vert R_0 \vert =1$ and $\mathrm{Bar}(R_0)=(0,0)$, and $\tau>0$. Let $\{ E_{k\tau}^{(\tau)}\}_{k \in \mathbb{N}}$ be a family of sets defined iteratively as
		$$ E_0^{(\tau)}=R_0 \quad \text{and} \quad E_{k \tau}^{(\tau)} \in \underset{E \subset \mathbb{R}^2,\, \vert E \vert=1}{\arg \min} \left\{\mathcal{F}_{\tau}(E,\, E_{(k-1)\tau}^{(\tau)})\right\} \quad k\geq 1. $$ 
		We define 
		$$ E_{t}^{(\tau)}:= E_{k \tau}^{(\tau)} \quad \text{for any } t \in [k \tau,\, (k+1)\tau).$$
		We call $\{E_t^{(\tau)}\}_{t \geq 0}$ an approximate flat $\mathrm{P}_{\vert \cdot \vert_1}$ solution of the area-preserving mean-curvature flow with initial datum $R_0$ and time step $\tau$. 
	\end{definition}
In the next theorems we prove that for every $R \subset \R^2$ rectangle with $\vert R \vert=1$ and $\mathrm{Bar}(R)=0$ there is a unique minimizer of the problem
\begin{equation*}\label{minimoprimadirettangolo}
\min \left\{ \mathrm{P}_{\vert \cdot \vert_{1}}(E)+ \frac{1}{\tau}\mathcal{D}(E,R)\colon \, E \subset \R^2, \,\vert E \vert=1 \right\} \tag{$\mathcal{P}_1$}
\end{equation*}
and it is a rectangle. We start by proving that the problem \eqref{minimoprimadirettangolo} is equivalent to
\begin{equation*}\label{aaaaminimoprimadirettangolo}
	\min \left\{ \mathrm{P}_{\vert \cdot \vert_{1}}(\overline{R})+ \frac{1}{\tau}\mathcal{D}(\overline{R},R)\colon \, \overline{R} \subset \R^2 \text{ be a rectangle and} \,\vert \overline{R} \vert=1, \, \mathrm{Bar}(\overline{R})=(0,0) \right\}. \tag{$\mathcal{P}_2$} 
\end{equation*}
	\begin{definition}
	Let $E\subset \mathbb{R}^2$ be a measurable set and let $\eta \in \mathcal{S}^1$. For every $z \in \R^2$ we define 
	the section $E_{z}^{\eta}\subset \mathbb{R}$ of $E$ as  $$E_{z}^{\eta}:=\left\{t \in \mathbb{R} \colon \mathbf{p}_{\eta^{\bot}}z+ t \eta \in E\right\}.$$
	The Steiner symmetrization $E^\eta$ of $E$ in direction $\eta$ is then defined as
	$$ E^{\eta}:= \left\{z \in \mathbb{R}^2 \colon \, \vert \mathbf{p}_{\eta} z \vert \leq  \frac{\mathcal{L}^1(E_{z}^{\eta})}{2}\right\}.$$
\end{definition}
We recall the following classical proposition (for a proof see for instance \cite{Maggibook}, \cite{Fusco2015}).
\begin{proposition}\label{proprietàsimsteiner}
	Let $E \subset \R^2$ be a measurable set. The following statements hold true.
	\begin{itemize}
		\item[i.] For all $  \eta \in \mathcal{S}^1$ and $\mathcal{H}^1$-a.e. $z \in \mathrm{ker}(\mathbf{p}_{\eta})$ the set $E_{z}^{\eta}$ is $\mathcal{H}^1$-measurable (the statement is true for all $z \in \mathrm{ker}(\mathbf{p}_{\eta})$ if $E$ is a Borel set).
		\item[ii.] The set $E^{\eta}\subset \R^2$ is $\mathcal{L}^2$-measurable.
		\item[iii.]  For all $\eta \in \mathcal{S}^1$ the function $z \in \mathrm{ker}(\mathbf{p}_{\eta}) \rightarrow \mathcal{L}^1(E_z^{\eta})$ is $\mathcal{H}^1$-measurable. 
		\item[iv.] We set $\pi_{1}(E):= \{z \in \R \colon \mathcal{L}^1(E_{(x,y)}^{e_2})>0\}$ and $\pi_{2}(E):= \{y \in \R\colon \mathcal{L}^1(E_{(x,y)}^{e_1})>0\}$ are $\mathcal{L}^1$-measurable and $ \vert E \setminus \pi_{1}(E) \times \pi_{2}(E) \vert=0$. 
	\end{itemize}
\end{proposition}
In the next theorem we state the classical Steiner inequality. For readers' convenience we provide a proof of the inequality regarding the anisotropic perimeter $P_{|\cdot|_1}$. 
	\begin{proposition}\label{steinersim}
		Let $E \subset \mathbb{R}^2$  be a set of finite perimeter. Then for every $ \eta \in \mathcal{S}^1$ the set $ E^{\eta}$ is a set of finite perimeter with $ \vert E \vert = \vert E^{\eta}\vert$ and
		$$ \mathrm{P}(E) \geq \mathrm{P}(E^{\eta}).$$
		Moreover if $ \eta \in \{ \pm e_1,\pm\, e_2\}$ it holds that
		$$ \mathrm{P}_{\vert \cdot \vert_1}(E )\geq \mathrm{P}_{\vert \cdot \vert_1}(E^{\eta} ).$$
	\end{proposition}
	\begin{proof}
		For a proof of the first statement see \cite{Maggibook}[chapter 14]. We present only the proof of the second statement when $\eta= e_2$, the other case being analogous. Given $u \colon \R \rightarrow \R$ and $H \subset \R$ we denote the graph of $u $ over $H$ by
		$$\Gamma(u,H):= \left\{(x,y) \in \R^2\colon x \in H, \, y=u(x)\right\}.$$
		We divide the proof in two steps. \\
		\textit{Step 1)}\\
		Let $E\subset \R^2$ be an open bounded set with polyhedral boundary.
 By this assumption there exists a finite partition of the set $G:= \{x \in \R \colon  \mathcal{L}^1(E_{(x,0)}^{e_2})>0\}$ into intervals $ \{(a_h,b_h)\}_{h=1}^{M}$ and into points $\{c_l\}_{l=1}^{L}$ in $\R$  i.e., $$ G= \bigcup_{h=1}^{M}(a_h,b_h) \cup \bigcup_{l=1}^{L} \{c_l\},$$
		and there exist affine functions $v_{h}^{k},u_h^{k}\colon (a_h,b_h) \rightarrow \mathbb{R}$, $1 \leq h \leq M$, $ 1 \leq k \leq N(h) $ with $ N(h) \in \mathbb{N}$ and there exist numbers $t_l^j,w_l^j \in \R$, $1 \leq l \leq L$, $ 1 \leq j \leq S(l)$ with $S(l) \in \N$ such that
		\begin{equation*}
			\partial^* E= \bigcup_{h=1}^{M} \bigcup_{k=1}^{N(h)} \Gamma(u_h^k,(a_h,b_h))\cup \Gamma(v_h^k,(a_h,b_h)) \cup \bigcup_{l=1}^{L} \bigcup_{j=1}^{S(l)} \{c_l\} \times (t_l^j,w_l^j)
		\end{equation*}
		and 
		\begin{equation}\label{asdad!!}
			E= \bigcup_{h=1}^{M} \left\{(x,y) \in (a_h,b_h) \times \R \colon \, y \in \bigcup_{k=1}^{N(h)}\left(v_h^k(x),u_h^k(x)\right)\right\}.	
		\end{equation} 
	By formula \eqref{asdad!!} we have that $ \forall h\in\left\{1,\dots,M\right\} \text{ and } \forall x \in (a_h,\,b_h)$
	$$  \mathcal{L}^1(E_{(x,0)}^{e_2})= \sum_{k=1}^{N(h)}u_h^k(x)-v_h^k(x)$$
	and $\forall h \in \left\{1,\dots,L\right\} \text{ and }  x =c_l$ 
	$$ \mathcal{L}^1(E_{(x,0)}^{e_2})= \sum_{j=1}^{S(l)}\vert w_l^j-t_l^j \vert.$$
	We set $m(x):= \mathcal{L}^1(E_{(x,0)}^{e_2})$ for every $x \in \R$ and we observe that such an $m$ is an affine function on every $ (a_h,b_h)$.  Since $ E^{e_2}= \left\{(x,y) \in G \times \R\colon \vert y \vert \leq \frac{m(x)}{2}\right\}$ is a
	bounded set with polyhedral boundary, its perimeter is finite. We denote by $\mu_E, \, \mu_{E^{e_2}}\in \left[\mathcal{M}(\R^2)\right]^2$ the distributional derivatives of the functions $\chi_E$ and $ \chi_{E^{e_2}}$, respectively. By the De Giorgi structure theorem we have that $\vert \mu_E \vert= \mathcal{H}^1\mrestr \partial^{*} E $ and $\vert \mu_{E^{e_2}} \vert= \mathcal{H}^1\mrestr \partial^{*} E^{e_2} $. For every $h =1,\dots,M$, for every $k =1,\dots, N(h)$ and for all $x \in (a_h,b_h)$ we have that the exterior normal vector field to $E$ is 
	\begin{equation*}
		\nu_{E}((x,u_h^k(x)))= \frac{\left(-\frac{d u_h^k}{dx}(x), 1\right)}{\sqrt{1+ \left| \frac{d u_h^k}{dx}(x) \right|^2 }},\quad  \nu_{E}((x,v_h^k(x)))= \frac{\left(\frac{d v_h^k}{dx}(x), -1\right)}{\sqrt{1+ \left| \frac{d v_h^k}{dx}(x) \right|^2 }}
	\end{equation*} 
and similarly
\begin{equation*}
	\nu_{E^{e_2}}((x,\frac{m(x)}{2}))= \frac{\left(-\frac{1}{2}\frac{d m}{dx}(x), 1\right)}{\sqrt{1+ \left| \frac{1}{2}\frac{d m}{dx}(x) \right|^2 }},\quad \nu_{E^{e_2}}((x,-\frac{m(x)}{2}))= \frac{\left(\frac{1}{2}\frac{d m}{dx}(x), -1\right)}{\sqrt{1+ \left| \frac{1}{2} \frac{d m}{dx}(x) \right|^2 }}.
\end{equation*} 
Moreover for every $l=1,\dots,L$ and for every $j=1,\dots,S(l)$ and for all $(x,y) \in \{c_l\}\times(t_l^j,w_l^j)$ we have that
\begin{equation*}
	\nu_{E}((x,y)) \in \{-e_1,e_1\},
\end{equation*}
and for $(x,y) \in \{c_l\} \times (- \frac{m(c_l)}{2},\frac{m(c_l)}{2})$ we have that
\begin{equation*}
	\nu_{E^{e_2}}((x,y)\in \{-e_1,e_1\}.
 \end{equation*}
    Using the
	formula for the area of a graph, the formulas above and by the very definition of the function $m$ it holds that
	\begin{equation}\label{sabform1}
		\begin{split}
		\mathrm{P}_{\vert \cdot \vert_1}(E)=&  \int_{\partial^* E} \vert \nu_{E}(z)  \vert_{1} d \mathcal{H}^1(z)\\
		=& \sum_{h=1}^M \mathrm{P}_{\vert \cdot \vert_1}(E, (a_h,b_h) \times \R) + \sum_{l=1}^{L} \sum_{j=1}^{S(l)} \int_{\{c_l\}\times (t_l^j,w_l^j)} \vert \nu_{E}(z)\vert_1 d \mathcal{H}^1(z)\\
		= & \sum_{h=1}^{M} \int_{a_h}^{b_h} \sum_{k=1}^{N(h)} \vert \nu_{E}(x,u^k_h(x)) \vert_1 \sqrt{1+ \vert \frac{d u_h^k}{d x}(x)\vert^2}+\vert \nu_{E}(x,v^k_h(x)) \vert_1 \sqrt{1+ \vert \frac{d v^k_h}{d x}(x)\vert^2}dx \\
		&+ \sum_{l=1}^{L} \sum_{j=1}^{S(l)} \vert t_l^j-w_l^j \vert\\
		=& \sum_{h=1}^{M} \int_{a_h}^{b_h} \sum_{k=1}^{N(h)} 2+ \left| \frac{du_h^k}{d x} (x)\right|+ \left| \frac{dv_h^k}{d x} (x)\right| +\sum_{l=1}^{L}m(c_l)
	\end{split}
	\end{equation}
and we have
\begin{equation}\label{sabform2}
	\begin{split}
		 \mathrm{P}_{\vert \cdot \vert_1}(E^{e_2})= & \int_{\partial^* E^{e_2}} \vert \nu_{E^{e_2}}(z) \vert_1 d \mathcal{H}^1(z) \\
		= &\mathrm{P}_{\vert \cdot \vert_1}(E^{e_2},G \times \R)+ \sum_{l=1}^L  \int_{\{c_l\} \times (- \frac{m(c_l)}{2},\frac{m(c_l)}{2})}\vert \nu_{E^{e_2}}(z) \vert_1 d \mathcal{H}^1(z) \\
		= &\int_{G} \sqrt{1+\left| \frac{1}{2} \frac{d m}{d x}(x)\right|^2} \left|\nu_{E^{e_2}}\left(x,\frac{m(x)}{2}\right) \right|_1+ \sqrt{1+\left| \frac{1}{2} \frac{d m}{d x}(x)\right|^2} \left|\nu_{E^{e_2}}\left(x,-\frac{m(x)}{2}\right) \right|_1dx \\
		&+ \sum_{l=1}^{L}m(c_l) 
		= \int_{G} 2 + \left| \frac{d m}{d x}(x) \right| + \sum_{l=1}^{L}m(c_l) \\
		=& \sum_{h=1}^M \int_{a_h}^{b_h} 2+ \left| \sum_{k=1}^{N(h)} \frac{du_h^k}{d x} (x)- \frac{dv_h^k}{d x} (x)   \right|+ \sum_{l=1}^{L}m(c_l).
	\end{split}
\end{equation}
	Therefore by the formulas \eqref{sabform1} and \eqref{sabform2} and from the triangular inequality it follows that 
	\begin{equation}\label{dissteiner}
	\mathrm{P}_{\vert \cdot \vert_1}(E) \geq \mathrm{P}_{\vert \cdot \vert_1}(E^{e_2}).
    \end{equation} 
\textit{Step 2)}\\
Let $E \subset \R^2$ be a set of finite perimeter and finite measure. There exists a sequence $\{E_h\}_{h \in \N} $ of open bounded sets with polyhedral
boundary such that 
\begin{equation}\label{limdensitàsab11}
 \lim_{h \rightarrow + \infty} \vert E_h \Delta E \vert=0 ,\quad \lim_{h \rightarrow + \infty} \mathrm{P}(E_h)= \mathrm{P}(E).
\end{equation}
Hence by the formula \eqref{limdensitàsab11} and the Reshetnyak continuity theorem we have 
\begin{equation}\label{contperansab11}
 \lim_{h \rightarrow + \infty} \mathrm{P}_{\vert \cdot \vert_1}(E_h)= \mathrm{P}_{\vert \cdot \vert_1}(E).
\end{equation}
Let us set $m_h(x):= \mathcal{L}^1((E_h)_{(x,0)}^{e_2})$ and  $G_h:= \{x \in \R\colon m_h(x)>0\}$,
by Fubini’s theorem, we have
\begin{equation*}\label{difsimmsab11}
	\vert E_h \Delta E \vert = \int_{\R} \mathcal{L}^1((E_h)_{(x,0)}^{e_2}\Delta E_x)dx \geq \int_{\R} \vert m_h(x)-m(x) \vert dx= \vert E_h^{e_2} \Delta E^{e_2} \vert.
\end{equation*}
Hence by the above formula and \eqref{limdensitàsab11} we have
\begin{equation}\label{asdasd99}
	\lim_{h \rightarrow +\infty} \vert E_h^{e_2} \Delta E^{e_2} \vert=0.
\end{equation} 
Therefore by \eqref{dissteiner}, \eqref{contperansab11}, \eqref{asdasd99} and by the lower semicontinuity of the function $ E \rightarrow \mathrm{P}_{\vert \cdot \vert_1}(E)$ with respect to the $ 	\mathrm{L^1(\R^2)}$ topology we obtain
\begin{equation*}
	\mathrm{P}_{\vert \cdot\vert_1}(E)= \lim_{h \rightarrow + \infty} \mathrm{P}_{\vert \cdot \vert_1}(E_h) \geq \liminf_{h \rightarrow + \infty} \mathrm{P}_{\vert \cdot \vert_1}((E_h)^{e_2}) \geq \mathrm{P}_{\vert \cdot \vert_1}(E^{e_2})
\end{equation*}
i.e. the thesis.
	\end{proof}
\begin{proposition}\label{decresdisp}
	Let $E \subset \R^2$ be a set of finite perimeter with $ \vert E \vert=1$ and let $R \subset \R^2$ be a rectangle with $\mathrm{Bar}(R)=(0,0)$ and $\vert R \vert=1$ with horizontal sidelength equal to $a>0$ and vertical sidelength equal to $b>0$. Then for every $ \eta \in \{ \pm e_1,\pm e_2\}$ 
	\begin{equation}
		\mathcal{D}(E,R) \geq \mathcal{D}(E^{\eta}, R).
	\end{equation}
\end{proposition}
\begin{proof}
	We only prove the case $\eta= e_2$. For $\mathcal{H}^1$-a.e. $(x,0) \in \R^2$ let as set $ E_x:=E_{(x,0)}^{e_2}$ and $ \tilde{E}_x:=  (-\frac{\mathcal{L}^1(E_x)}{2},\frac{\mathcal{L}^1(E_x)}{2})$. We denote by $ R_x$ the interval $ (-\frac{b}{2},\frac{b}{2})$. By the Fubini-Tonelli theorem we can write
	\begin{equation}\label{23032023pom1}
		\mathcal{D}(E,R)= \int_{E \Delta R} d_{\infty}(z,\partial R)dz= \int_{\R}dx \int_{E_x \Delta R_x} d_{\infty}((x,y),\partial R) dy
	\end{equation}
and 
\begin{equation}\label{23032023pom2}
	\mathcal{D}(E^{e_2},R)= \int_{E^{e_1} \Delta R} d_{\infty}(z,\partial R)dz= \int_{\R}dx \int_{\tilde{E}_x \Delta R_x} d_{\infty}((x,y),\partial R) dy.
\end{equation}
We claim that 
\begin{equation}\label{23032023pom3}
	 \int_{E_x \Delta R_x} d_{\infty}((x,y),\partial R) dy\geq \int_{\tilde{E}_x \Delta R_x} d_{\infty}((x,y),\partial R) dy. \end{equation}
 We first study the case $\mathcal{L}^1(E_x) <  \mathcal{L}^1(R_x)$. 
	If moreover $\mathcal{L}^1(E_x \cap R_x^c) \neq 0$ then we have $\mathcal{L}^1(R_x \setminus E_x)>\mathcal{L}^1(E_x \cap R^c)$. Hence there exists a $\mathcal{L}^1$-measurable set $M \subset R_x \setminus E_x$ such that $\mathcal{L}^1(M)= \mathcal{L}^1(E_x \cap R^c)$. 
	Indeed the function $f(t):= \mathcal{L}^1( (-t,t) \cap R_x \setminus E_x)$ is $C^0(\R)$ 
	  and such that $f(0)=0$ and  for $ f(\frac{b}{2})= \mathcal{L}^1(R_x \setminus E_x)$. Therefore there exist $s \in (-t,t)$ such that $ f(s)=\mathcal{L}^1( (-s,s) \cap R_x \setminus E_x)= \mathcal{L}^1(E_x \cap R^c)$ and hence we define $M= : (-s,s) \cap R_x \setminus E_x$. Therefore we have 
	  \begin{equation}\label{27032023mattina1}
	  	\begin{split}
	  	&\int_{E_{x} \Delta R_{x}} d_{\infty}((x,y), \partial R ) dy = \int_{E_x \setminus R_x } d_{\infty}((x,y), \partial R ) dy+ \int_{R_x \setminus E_x } d_{\infty}((x,y), \partial R ) dy\\
	  	&\geq \int_{R_x \setminus E_x } d_{\infty}((x,y), \partial R ) dy \geq \int_{R_x \setminus (E_x \cup M) } d_{\infty}((x,y), \partial R ) dy.
  	\end{split}
	  \end{equation}
  By the last formula it suffices to study the case $\mathcal{L}^1(E_x) < \mathcal{L}^1(R_x)$ and $ E_x \subset R_x$. We observe that 
  \begin{equation}\label{27032023mattina2}
  	R_x \setminus E_x =\tilde{E}_x \setminus E_x \cup R_x \setminus (\tilde{E}_x \cup E_x) \text{ and }  R_x \setminus \tilde{E}_x = E_x \setminus \tilde{E}_x \cup R_x \setminus (\tilde{E}_x \cup E_x),
  \end{equation}
  hence we obtain that
  \begin{equation}\label{27032023mattina4}
  \vert R_x \setminus E_x \vert=  \vert R_x \setminus \tilde{E}_x \vert \implies \vert \tilde{E}_x \setminus E_x \vert = \vert E_x \setminus \tilde{E}_x \vert. 	
\end{equation} 
  Moreover for all $ y' \in E_x \setminus \tilde{E}_x $ and  $ y \in \tilde{E}_x \setminus E_x$ it holds 
  \begin{equation}\label{27032023mattina5}
  d_{\infty}((x,y'),\partial R) \leq \frac{b- \mathcal{L}^1(E_x)}{2} \leq d_{\infty}((x,y),\partial R). 
  \end{equation} 
Therefore thanks to \eqref{27032023mattina2}, \eqref{27032023mattina4} and \eqref{27032023mattina5} we have
\begin{equation}\label{27032023mattina6}
	\begin{split}
	&	\int_{E_x \Delta R_x} d_{\infty}((x,y),\partial R) dy = \int_{R_x \setminus E_x }d_{\infty}((x,y),\partial R) dy\\
	& = \int_{\tilde{E}_x \setminus E_x }d_{\infty}((x,y),\partial R) dy +\int_{ R_x \setminus (\tilde{E}_x \cup E_x )}d_{\infty}((x,y),\partial R) dy  \\
	& \geq 	\int_{E_x \setminus \tilde{E}_x }d_{\infty}((x,y),\partial R) dy +\int_{ R_x \setminus ( \tilde{E}_x\cup  E_x )}d_{\infty}((x,y),\partial R) dy 	\\
	&= \int_{R_x \setminus \tilde{E}_x }d_{\infty}((x,y),\partial R) dy=\int_{\tilde{E}_x \Delta R_x} d_{\infty}((x,y),\partial R) dy.
	\end{split}
\end{equation}
If instead $ \mathcal{L}^1(E_x) \geq \mathcal{L}^1(R_x)$ and $\mathcal{L}^1( R_x \setminus E_x) \neq 0$ then we have that there exist $M \subset E_x \setminus R_x$ measurable set such that $\vert M \vert= \vert R_x \setminus E_x \vert $. Therefore we define 
$$E_x^1:= R_x \cup \left(E_x \setminus (M \cup R_x)\right) $$
and we obtain
$$ E^1_x \Delta R_x = E_x \setminus (M \cup R_x) \subset E_x \setminus R_x \subset E_x \setminus R_x \cup R_x \setminus E_x   =E_x \Delta R_x,$$
hence
\begin{equation}\label{27032023mattina7}
	\int_{E_x \Delta R_x} d_{\infty}((x,y),\partial R ) dy \geq \int_{E^1_x \Delta R_x } d_{\infty}((x,y),\partial R)dy .
\end{equation}
The last formula show that it suffices to analyze the case $\mathcal{L}^1(E_x) \geq \mathcal{L}^1(R_x)$ and $ R_x \subset E_x$.
We observe that
\begin{equation}\label{27032023mattina8}
	E_x \Delta R_x = E_x \setminus R_x = \left[ (E_x \setminus R_x) \cap \tilde{E}_x\right] \cup \left[ (E_x \setminus R_x)\setminus \tilde{E}_x\right]
\end{equation}
and 
\begin{equation}\label{27032023mattina9}
	\tilde{E}_x \Delta R_x = \tilde{E}_x \setminus R_x = \left[ (E_x \setminus R_x) \cap \tilde{E}_x\right] \cup \left[  \tilde{E}_x \setminus E_x \right].
\end{equation}
Moreover we have that for all $y \in (E_x \setminus R_x)\setminus \tilde{E}_x$ and for all $y' \in \tilde{E}_x \setminus E_x $
\begin{equation}\label{27032023mattina10}
	d_{\infty}((x,y),\partial R) \geq \frac{\mathcal{L}^1(E_x)-b}{2} \geq d_{\infty}((x,y'),\partial R).
\end{equation}
Thank to \eqref{27032023mattina8}, \eqref{27032023mattina9} and \eqref{27032023mattina10}, following the same arguments leading to  \eqref{27032023mattina6} we have
\begin{equation}\label{27032023mattina11}
	\int_{E_x \Delta R_x} d_{\infty}((x,y),\partial R)dy \geq \int_{\tilde{E}_x \Delta R_x} d_{\infty}((x,y),\partial R)dy. 
\end{equation} 
The formula above together with \eqref{27032023mattina6} gives \eqref{23032023pom3}. The thesis follow from \eqref{23032023pom3}, \eqref{23032023pom1}, \eqref{23032023pom2}.
\end{proof}

\begin{proposition}\label{Pminpi1pi2}
	For every $E\subset \R^2$ set of finite perimeter it holds that 
	\begin{equation}\label{140220231}
		\mathrm{P}_{\vert \cdot\vert_1}(E) \geq 2 \mathcal{H}^1(\pi_{1}(E))+  2 \mathcal{H}^1(\pi_{2}(E)).
	\end{equation}
\end{proposition}
\begin{proof}
	By \eqref{per1divformula} we have that
	\begin{equation}\label{formulaprepranzo}
	\begin{split}
		\mathrm{P}_{\vert \cdot \vert_1}(E)=& \sup \left\{ \int_{E} \mathrm{div}(T)(x,y)dxdy\colon T \in C^1_{c}(\R^2,\R^2), \, \vert T \vert_{\infty}(x,y) \leq 1, \, \forall x \in \R^2\right\}\\
		=&\sup \left\{ \int_{E} \frac{\partial T_1}{\partial x}+ \frac{\partial T_2}{\partial y}dxdy\colon T_i \in C^1_{c}(\R^2,\R), \, \vert T_i \vert(x,y) \leq 1, \, \forall x \in \R^2, \, \forall i=1,2\right\} \\
		=&  \sup \left\{ \int_{E} \frac{\partial T_1}{\partial x} dxdy\colon T_1 \in C^1_{c}(\R^2,\R), \, \vert T_1 \vert(x,y) \leq 1, \, \forall x \in \R^2\right\}\\
		&+ \sup \left\{ \int_{E} \frac{\partial T_2}{\partial y} dxdy\colon T_2 \in C^1_{c}(\R^2,\R), \, \vert T_2 \vert(x,y) \leq 1, \, \forall x \in \R^2\right\}.
	\end{split}
	\end{equation}
On he other hand we have 
\begin{equation}\label{03032023formulaprop1}
	\begin{split}
		 \sup \left\{ \int_{E} \frac{\partial T_1}{\partial x} dxdy\colon T_1 \in C^1_{c}(\R^2,\R), \, \vert T_1 (z)\vert \leq 1, \, \forall z \in \R^2\right\}\\
		&\hspace{-9cm}\geq  \sup \left\{ \int_{E} \frac{\partial T_1}{\partial x} dxdy\colon T_1(x,y)=f(x)g(y), \, f,g \in C^1_{c}(\R,\R), \, \vert f \vert, \, \vert g \vert \leq 1, \, \text{in } \R\right\}\\
		&\hspace{-9cm}= \sup \left\{ \int_{\pi_{2}(E)}g(y)dy \int_{ E^{e_1}_{(0,y)}} \frac{ d  f}{d x}(x)dx\colon  f,g \in C^1_{c}(\R,\R), \, \vert f \vert, \, \vert g \vert \leq 1, \, \text{in } \R\right\}\\ &\hspace{-9cm}\geq 2 \mathcal{H}^{1}(\pi_{2}(E))
	\end{split}
\end{equation}
where in the last inequality we have used that the slice $ E_{(0,y)}^{e_1}$ is a set of finite perimeter in $\R$  for $\mathcal{H}^1$-a.e. $ y \in \R$, and for every $y \in \pi_{2}(E)$ the zero dimensional perimeter of the set $E_{(0,y)}^{e_1}$ is greater that $2$, see \cite{Fusco2015}[Theorem 2.5]. With the same argument we obtain that  
\begin{equation}\label{03032023porpform123}
	\sup \left\{ \int_{E} \frac{\partial T_2}{\partial y} dxdy\colon T_2 \in C^1_{c}(\R^2,\R), \, \vert T_2 (z)\vert \leq 1, \, \forall z \in \R^2\right\}\geq 2 \mathcal{H}^{1}(\pi_{1}(E)),
\end{equation}
Hence the thesis.
\end{proof} 
We are now in a position to prove the equivalence of the problem \eqref{aaaaminimoprimadirettangolo} and \eqref{minimoprimadirettangolo}.
\begin{theorem}\label{thm-equiv}
	Let $R \subset \R^2$ be a rectangle such that $ \vert R \vert=1$.
	The problem \eqref{aaaaminimoprimadirettangolo} and \eqref{minimoprimadirettangolo} are equivalent.
\end{theorem}
\begin{proof}
	Let $E \subset \mathbb{R}^2$ be such that $ \vert E \vert=1$. We define the set $E^s$  as the set obtained through the Steiner symmetrization with respect to $e_1$ and $e_2$; i.e., $ E^s= (E^{e_1})^{e_2}$. By Proposition \ref{steinersim} and  Proposition \ref{decresdisp} we obtain 
	\begin{equation}
		\mathrm{P}_{\vert \cdot \vert_1}(E^s ) + \frac{1}{\tau} \mathcal{D}(E^s, R) \leq \mathrm{P}_{\vert \cdot \vert_1}(E ) + \frac{1}{\tau} \mathcal{D}(E, R). 
	\end{equation} 
We first observe that if $ \mathcal{H}^1(\pi_{1}(E^s)) >a $ and $\mathcal{H}^1(\pi_{2}(E^s)) >b  $ then by Proposition \ref{Pminpi1pi2}
\begin{equation}
	 \mathrm{P}_{\vert \cdot \vert_1}(E^s ) + \frac{1}{\tau} \mathcal{D}(E^s, R) \geq 2 \mathcal{H}^1(\pi_{1}(E^s))+  2 \mathcal{H}^1(\pi_{2}(E^s)) \geq \mathrm{P}_{\vert \cdot \vert_1}(R ) + \frac{1}{\tau} \mathcal{D}(R, R).  
\end{equation}
We complete the proof fo the result only in the case
 $ \mathcal{H}^1(\pi_{1}(E^s)) <a $, the case   $ \mathcal{H}^1(\pi_{2}(E^s)) <b $ being analogous.
We define the rectangle
$$\overline{R}:= \pi_{1}(E^s) \times \left[-\frac{L}{2},\frac{L}{2} \right] \quad \text{ with $L$ such that  } \quad L \mathcal{H}^1(\pi_{1}(E^s))=1.$$
Note that $ \mathrm{Bar}(\overline{R})= \mathrm{Bar}(E_s)=(0,0)$.
We claim that 
\begin{equation}\label{27032023pomclaim1}
	\mathrm{P}_{\vert \cdot \vert_1}(E^s) \geq 2 \mathcal{H}^1(\pi_{1}(E^s))+  2 L= \mathrm{P}_{\vert \cdot \vert_1}(\overline{R}).
\end{equation}
By the Proposition \ref{Pminpi1pi2} we have
\begin{equation}\label{27032023pom09}
	\mathrm{P}_{\vert \cdot \vert_1}(E^s) \geq 2 \mathcal{H}^1(\pi_{1}(E^s))+  2 \mathcal{H}^1(\pi_{2}(E^s)).
\end{equation}
 Up to a  null set we have that 
$$E^s \subset \pi_1(E^s) \times \pi_{2}(E^s)$$
 and then 
\begin{equation}\label{27032023p1608}
 \mathcal{H}^1(\pi_{1}(E^s)) \mathcal{H}^1(\pi_{2}(E^s)) \geq \vert E^s \vert =1= \mathcal{H}^{1}(\pi_{1}(E^s)) L \implies L \leq \mathcal{H}^1(\pi_{2}(E^s))	.
\end{equation}
Hence \eqref{27032023pomclaim1} follows from \eqref{27032023p1608} and \eqref{27032023pom09}. 
We now claim that 
\begin{equation}\label{tc270320231640}
	\mathcal{D}(E^s,R) \geq \mathcal{D}(\overline{R},R).
\end{equation}
We define the sets
\begin{equation}
	\mathcal{A}:= E^s \cap   \R \times \bigg(-\infty,-\frac{L}{2}\bigg) \cup \bigg(\frac{L}{2}, +\infty\bigg) \quad \mathcal{B}:= E^s \cap \R \times \bigg(-\frac{L}{2},-\frac{b}{2}\bigg) \cup \bigg(\frac{b}{2},\frac{L}{2}\bigg)
\end{equation}
and 
\begin{equation}
	\mathcal{C}:= \overline{R} \cap R \setminus E^s  \quad \mathcal{D} := R \setminus \overline{R}  \quad \mathcal{E}:= \pi_1(E^s) \times\bigg( -\frac{L}{2}, -\frac{b}{2}\bigg) \cup \bigg( \frac{b}{2}, \frac{L}{2} \bigg) \setminus E^s
\end{equation}
and we have
\begin{equation}\label{27032023sera1}
	R \Delta E^s = \mathcal{A} \cup \mathcal{B} \cup \mathcal{C} \cup \mathcal{D} \quad R \Delta \overline{R}= \mathcal{B} \cup \mathcal{D} \cup \mathcal{E}.
\end{equation}
We observe that $$ \vert E^s \vert= \vert \overline{R}\vert \implies \vert \mathcal{A} \vert= \vert \mathcal{E} \cup  \mathcal{C} \vert,$$
and then we define $ \mathcal{A}' $ and $\mathcal{A}''$ measurable set such that $ \mathcal{A} = \mathcal{A}' \cup \mathcal{A}''$ and $ \vert \mathcal{A}' \vert = \vert \mathcal{C} \vert $, $ \vert \mathcal{A}'' \vert = \vert \mathcal{E} \vert $. Note that the construction of $ \mathcal{A}'$ straightforward and can be obtained as follows. We define the continuous function $f(t):= \vert \mathcal{A} \cap \R \times (-\infty,-t) \cup (t,\infty) \vert $  and observe that $f(\frac{\mathcal{H}^1(\pi_{2}(E^s))}{2})= 0$, $f(\frac{b}{2})= \vert \mathcal{A} \vert $. Hence there exists $s\in \R$ such that $f(s)= \vert \mathcal{C} \vert $ and $\mathcal{A}':= \mathcal{A} \cap \R \times (-\infty,-s) \cup (s,\infty)$, $\mathcal{A}'':= \mathcal{A} \setminus \mathcal{A}'$. 
By the construction of $\mathcal{A}' $ and $ \mathcal{A}''$ it holds that
$$\text{ for all }(x',y') \in \mathcal{A}' \text{ and for all }   (x'',y'') \in\mathcal{A}'' \implies y' \geq y''.$$
We introduce a new set $ \hat{E}:= \mathcal{C} \cup E^s \setminus \mathcal{A}' $ and observe that 
\begin{equation}\label{27032023sera2}
	R \Delta \hat{E}= \mathcal{A}'' \cup \mathcal{B} \cup \mathcal{D} \subset  \mathcal{A} \cup \mathcal{B} \cup \mathcal{C} \cup \mathcal{D}= R \Delta E^s,
\end{equation} 
which in turn implies that  
\begin{equation}
\mathcal{D}(\hat{E},R)	=\int_{R \Delta \hat{E}} d_{\infty}((x,y),\partial R)dxdy \leq \int_{R \Delta E^s} d_{\infty}((x,y),\partial R) = \mathcal{D}(E^s,R).
\end{equation}
We are left to prove that 
\begin{equation}\label{27032023sera4}
	\mathcal{D}(\overline{R},R) \leq \mathcal{D}(\hat{E},R).
\end{equation} 
We observe that
\begin{equation}
	\text{for all } z'' \in \mathcal{A}'' \text{ and for all } z \in\mathcal{E} \implies d_{\infty}(z,\partial R) \leq \frac{L-b}{2} \leq d_{\infty}(z'',\partial R)
\end{equation}
and since $ \vert \mathcal{E} \vert= \vert \mathcal{A}''\vert$ we have
\begin{equation}\label{27032023sera3}
	\int_{\mathcal{E}} d_{\infty}(z,\partial R)dz \leq \int_{\mathcal{A}''} d_{\infty}(z,\partial R) dz.
\end{equation}
Therefore by \eqref{27032023sera2} and \eqref{27032023sera3} we obtain
\begin{equation}
	\begin{split}
		\mathcal{D}(\overline{R},R) =& \int_{\mathcal{D}}d_{\infty}(z,\partial R)dz + \int_{\mathcal{B}}d_{\infty}(z,\partial R)dz+ \int_{\mathcal{E}}d_{\infty}(z,\partial R)dz \\
		 \leq &\int_{\mathcal{D}}d_{\infty}(z,\partial R)dz+ \int_{\mathcal{B}}d_{\infty}(z,\partial R)dz+ \int_{\mathcal{A}''}d_{\infty}(z,\partial R)dz = \mathcal{D}(\hat{E}, R)
	\end{split}
\end{equation}
which prove the claim. Eventually  by  \eqref{27032023pomclaim1}, \eqref{tc270320231640} we obtain the thesis.
\end{proof}

In the next theorem we want to compute the minimizer of the problem \eqref{aaaaminimoprimadirettangolo}, namely \begin{equation*} \min \left\{ \mathrm{P}_{\vert \cdot \vert_{1}}(\overline{R})+ \frac{1}{\tau} \mathcal{D}(\overline{R},R)\colon \, \overline{R}\subset \R^2 \text{ be a rectangle and } \vert \overline{R} \vert=1, \, \mathrm{Bar}(\overline{R})=(0,0)\right\},
\end{equation*}
where we remind the reader that $R \subset \R^2$ is a rectangle with horizontal sidelength $a$, vertical sidelength $b$, $\mathrm{Bar}(R)=(0,0)$, $\vert R \vert=ab=1$. Without loss of generality we moreover assume that $b\leq a$. 
\begin{theorem}\label{THMesistenzadeiminimi}
	Let $a,b \in \R$ be such that $0<b\leq a$ and $ab=1$ and let us consider $R = \left[-\frac{a}{2}, \frac{a}{2}\right] \times \left[-\frac{b}{2}, \frac{b}{2}\right]$.
	 For all $\Lambda>a+b$ there exists $ \tau_0:= \tau_0(\Lambda)$ such that for every $ \tau< \tau_0$ there exists a unique minimizer $R'$ of 
	\begin{equation*} \min \left\{ \mathrm{P}_{\vert \cdot \vert_{1}}(\tilde{R})+ \frac{1}{\tau} \mathcal{D}(\tilde{R},R)\colon \, \tilde{R}\subset \R^2 \text{ be a rectangle and } \vert \tilde{R} \vert=1, \, \mathrm{Bar}(\tilde{R})=0\right\}.
	\end{equation*}
Moreover we have $R^{'}= \left[-\frac{a'}{2}, \frac{a'}{2}\right] \times \left[-\frac{b'}{2}, \frac{b'}{2}\right]$ with
\begin{equation}
\begin{split}
	& a'= a-2x(\tau) \text{ with } x(\tau)= x' \tau+ O(\tau^2) \text{ and } x'=\frac{2(a-b)}{b(a+b)},\\
    &  b'= b+2y (\tau) \text{ with } y(\tau)= \frac{bx(\tau) }{a-2x(\tau)} \\
    &  a'+b'\leq a+b ,\, \vert a'-a \vert \leq \tau C(\Lambda),\, \vert b'-b \vert \leq \tau C(\Lambda) \text{ and } \vert R \Delta R' \vert \leq \tau C(\Lambda).
\end{split}
\end{equation}
\end{theorem}
\begin{proof}
	We divide the proof in to several steps. \\
	\textit{Step 1)}\\
	In this step we show that if $R'=\left[-\frac{a'}{2}, \frac{a'}{2}\right] \times \left[-\frac{b'}{2}, \frac{b'}{2}\right]$ is a minimizer 
then
	$$ a'< a \text{ and } b'>b.$$
%
%
%

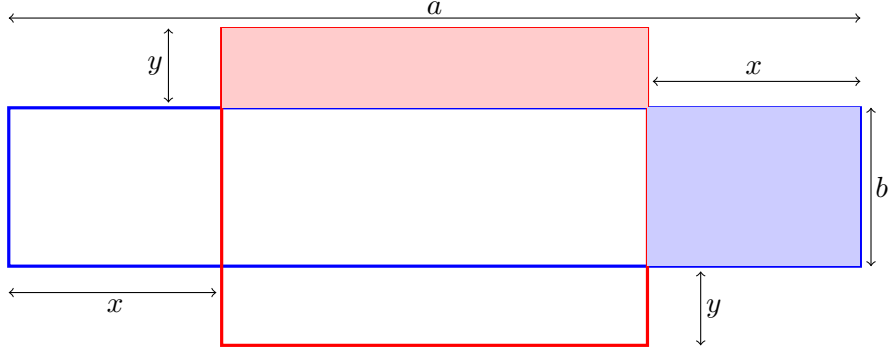
\begin{figure}[h!]
    \centering
    \begin{tikzpicture}[scale=0.7]
        \draw[blue,very thick] (-8,-1.5) rectangle (8,1.5);
        \draw[red,very thick] (-4,-3) rectangle (4,3);
        
        \draw [<->] (-8,-2) -- (-4.1,-2);
        \draw (-6,-2.25) node{$x$};
        
        \draw [<->] (8,2) -- (4.1,2);
        \draw (6,2.25) node{$x$};
        
        \draw [<->] (5,-1.6) -- (5,-3);
        \draw (5.25,-2.3) node{$y$};
        
        \draw [<->] (-5,1.6) -- (-5,3);
        \draw (-5.25,2.3) node{$y$};
        
        \draw [black,very thick] (0,2.1) node[fill=white]{$R_2$};
        \draw [black,very thick] (6,0) node[fill=white]{$R_1$};
        
        \draw [<->] (-8,3.2) -- (8,3.2);
        \draw[black, very thick] (0,3.4) node{$a$};
        
        \draw[<->] (8.2,-1.5) -- (8.2,1.5);
        \draw[black, very thick] (8.4,0) node{$b$};
        
        \fill[red!20,opacity=0.3] (-4,1.5) rectangle (4,3);
        \fill[blue!20,opacity=0.3] (4,-1.5) rectangle (8,1.5);
    \end{tikzpicture}
    \caption{The rectangle $R$, $\tilde{R}$, $R_1$, and $R_2$.}
    \label{fig1}
\end{figure}

	We argue by contradiction. Let us suppose $a'= a+x $ for some $x>0$ and let $b'$ be such that $ab=a'b'$, that is to say $ b'= \frac{ab}{a+x}$. Then, using the assumption $b \leq a$ we have that
	\begin{equation}\label{27032023pomeriggio1}
		  a+b \leq a+ x+ \frac{ab}{a+x}= a'+b'.
	 \end{equation}
Therefore by \eqref{27032023pomeriggio1} we get
	\begin{equation}
\mathrm{P}_{\vert \cdot \vert_1}(R)+ \frac{1}{\tau} \mathcal{D}(R,R) =2a +2b \leq 2a'+2b' < \mathrm{P}_{\vert \cdot \vert_1}(R')+ \frac{1}{\tau} \mathcal{D}(R',R),
	\end{equation}
which contradicts  the minimality of $R'$. \\
	\textit{Step 2)}\\
	In this step we show that our minimum problem can by studied by analyzing a function of two real variables.  From the previous step we can assume that a competitor of our minimum problem can be written as follows
	$$\tilde{R}= \left[-\frac{a}{2}+x, \frac{a}{2}-x\right] \times \left[-\frac{b}{2}-y,\frac{b}{2}+y\right]$$
	 for some $x \in [0, \frac{a}{2})$ and   
	$y= \frac{xb}{a-2x} $.
	The perimeter of the set $\tilde{R}$ is 
	\begin{equation}\label{perimeterinXY}
		\mathrm{P}_{\vert \cdot \vert_1}(\tilde{R})=2(a-2x+b+2y).
	\end{equation} 
	Using the symmetry of the set $ \tilde{R} \Delta R$, the dissipation is 
	\begin{equation}\label{fordisp03032023}
 \mathcal{D}(\tilde{R}, R)= \int_{\tilde{R}\Delta R} d_{\infty}(z,\partial R) dz= 2\int_{R_1} d_{\infty}(z,\partial R) dz+ 2 \int_{R_2} d_{\infty}(z,\partial R)dz
\end{equation}
	where 
	$$ R_1= \left[-\frac{a}{2}+x,\frac{a}{2}-x\right] \times \left[\frac{b}{2},\frac{b}{2}+y\right] \text{ and } R_2= \left[\frac{a}{2}-x,\frac{a}{2}\right] \times \left[-\frac{b}{2},\frac{b}{2}\right].$$
	We now compute explicitly the dissipation as a function of the variables $x $ and $y$.
	The first integral in \eqref{fordisp03032023} is
	\begin{equation}\label{for103032023}
		\int_{R_1} d_{\infty}((z_1,z_2),\partial R) dz_1dz_2= \int_{0}^{y} (a-2x)z_2dz_2= \frac{1}{2} y^2(a-2x).
	\end{equation}
The formula for the second integral in \eqref{fordisp03032023} depends on whether $ x < \frac{b}{2}$ or $x \geq \frac{b}{2}$. \\

\textit{Case $ x < \frac{b}{2}$ }. It holds that

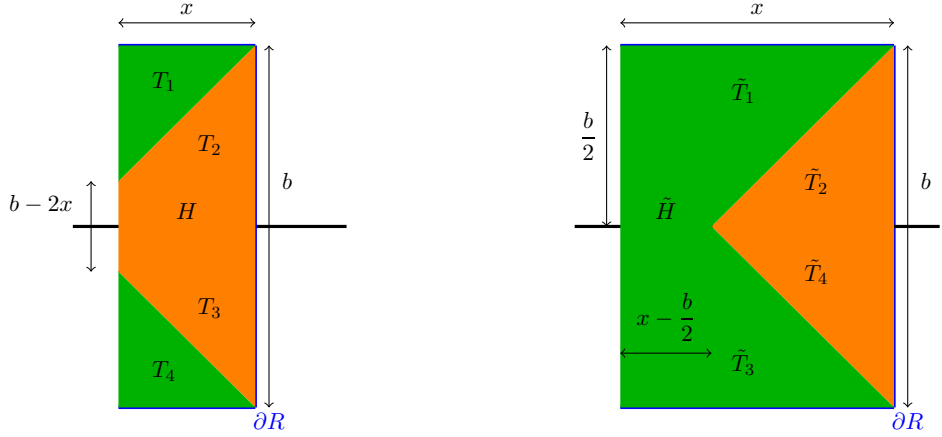
\begin{figure}[h!]
    \centering
    \begin{tikzpicture}[scale=0.6, every node/.style={scale=0.8}]
        \draw[black, very thick] (-12,0) -- (-6,0);
        \draw[blue,very thick] (-11,-4) -- (-8,-4);
        \draw[blue,very thick] (-11,4) -- (-8,4);
        \draw[blue,very thick] (-8,-4) -- (-8,4);
        \draw (-11,1) -- (-8,1);
        \draw (-11,-1) -- (-8,-1);
        
        \fill[orange, opacity=0.2]
            (-11,1) -- (-8,4) --(-8,-4)-- (-11,-1) --cycle;
        \fill[green!70!black, opacity=0.2]
            (-11,1) -- (-11,4) --(-8,4)--cycle;
        \fill[green!70!black, opacity=0.2]
            (-11,-1) -- (-11,-4) --(-8,-4)--cycle;
        
        \draw [black,very thick] (-10,-3.2) node{$T_4$};
        \draw [black,very thick] (-9,-1.8) node{$T_3$};
        \draw [black,very thick] (-10,3.2) node{$T_1$};
        \draw [black,very thick] (-9,1.8) node{$T_2$};
        \draw [black,very thick] (-9.5,0.35) node{$H$};
        \draw [blue,very thick] (-7.7,-4.3) node{$\partial R$};
        
        \draw [<->] (-7.7,-4) -- (-7.7,4);
        \draw (-7.3,1) node{$b$};
        \draw [<->] (-11.6,-1) -- (-11.6,1);
        \draw (-12.7,0.5) node{$b-2x$};
        \draw [<->] (-11,4.5) -- (-8,4.5);
        \draw (-9.5,4.8) node{$x$};
        
        \draw[blue,very thick] (0,-4) -- (6,-4);
        \draw[blue,very thick] (6,-4) -- (6,4);
        \draw[blue,very thick] (0,4) -- (6,4);
        \draw[black,very thick] (-1,0) -- (7,0);
        \draw (2,4)--(2,-4);
        
        \fill[orange, opacity=0.2]
            (6,-4) -- (6,4) --(2,0) --cycle;
        \fill[green!70!black, opacity=0.2]
            (2,0)--(6,4) -- (0,4)--(0,-4) -- (6,-4)--cycle;
        
        \draw [black,very thick] (1,0.4) node{$\tilde{H}$};
        \draw [black,very thick] (2.7,3) node{$\tilde{T}_1$};
        \draw [black,very thick] (2.7,-3) node{$\tilde{T}_3$};
        \draw [black,very thick] (4.3,1) node{$\tilde{T}_2$};
        \draw [black,very thick] (4.3,-1) node{$\tilde{T}_4$};
        
        \draw [<->] (0,4.5) -- (6,4.5);
        \draw (3,4.8) node{$x$};
        \draw [<->] (6.3,4) -- (6.3,-4);
        \draw (6.7,1) node{$b$};
        \draw [<->] (-0.3,4) -- (-0.3,0);
        \draw (-0.7,2) node{$\displaystyle\frac{b}{2}$};
        \draw [<->] (0,-2.8) -- (2,-2.8);
        \draw (1,-2) node{$x-\displaystyle\frac{b}{2}$};
        \draw [blue,very thick] (6.3,-4.3) node{$\partial R$};
    \end{tikzpicture}
    \caption{Nel lato sinistro il caso $x \leq \frac{b}{2}$ e nel lato destro il caso $x > \frac{b}{2}$.}
    \label{fig2}
\end{figure}

\begin{equation}\label{08032023asdasd}
	\begin{split}
	\int_{R_2} d_{\infty}((z_1,z_2),\partial R) dz_1 dz_2=&  \int_{T_1} d_{\infty}(z,\partial R)dz+  \int_{T_2} d_{\infty}(z,\partial R)dz \\
	&+ \int_{T_3} d_{\infty}(z,\partial R)dz+ \int_{T_4} d_{\infty}(z,\partial R)dz+ \int_{H} d_{\infty}(z,\partial R)dz.
\end{split}
\end{equation}
Here $T_1$ denotes the triangle with vertices 
$$\left\{\left(\frac{a}{2}-x, \frac{b}{2}-x\right), \, \left(\frac{a}{2}, \frac{b}{2}\right),\, \left(\frac{a}{2}-x, \frac{b}{2}\right)\right\}, $$
 $T_2 $ denotes the triangle with vertices
$$ \left\{\left(\frac{a}{2}-x, \frac{b}{2}-x\right), \, \left(\frac{a}{2}, \frac{b}{2}\right),\, \left(\frac{a}{2}, \frac{b}{2}-x\right)\right\}, $$
 while $T_3$ and $T_4$ are the triangles obtained reflecting $T_1$ and $T_2$ with respect to the horizontal axis, respectively.
  Finally, $H$ denotes the rectangle with vertices
$$ \left\{\left(\frac{a}{2}-x, \frac{b}{2}-x\right), \, \left(\frac{a}{2}, \frac{b}{2}-x\right),\, \left(\frac{a}{2}-x, -\frac{b}{2}+x\right), \, \left(\frac{a}{2}, -\frac{b}{2}+x\right)\right\}. $$
By a change of variables we observe that for every $i \in \{1,2,3,4\}$
$$  \int_{T_i} d_{\infty}(z,\partial R)dz=\int_{T} d_{\infty}(z,[0,x]\times \{0\}) dz= \int_{T} z_2 dz_1 dz_2$$
where $T$ is a triangle with vertices $(0,0), (x,0),(0,x)$.
 Hence, by \eqref{08032023asdasd} and the formula above, we obtain that
 \begin{equation}\label{08032023ffor1}
 	\begin{split}
 		\int_{R_2}d_{\infty}(z,\partial R)dz&= 4 \int_{T} d_{\infty}(z,\partial R)dz+ \int_{H} d_{\infty}(z,\partial R)dz  \\
 		&= 4\int_{0}^x dz_1 \int_{0}^{x-z_1} z_2 dz_2+ \int_{\frac{a}{2}-x}^{\frac{a}{2}}\int_{-\frac{b}{2}+x}^{\frac{b}{2}-x}(z_1- \frac{a}{2}+x)dz_1d z_2 \\
 		&= \frac{2x^3}{3}+ \frac{x^2(b-2x)}{2}.
 	\end{split}
 \end{equation} 
eventually by \eqref{fordisp03032023}, \eqref{for103032023}, \eqref{08032023ffor1} we have
\begin{equation}\label{dissipazionebuona}
	D(\tilde{R},R)= \int_{\tilde{R}\Delta R}d_{\infty}(z,\partial R)dz= y^2(a-2x)+\frac{4}{3}x^3+x^2(b-2x).
\end{equation}
\textit{Case $x \geq \frac{b}{2}$}.
It holds that
\begin{equation}\label{08032023dopopranzo1}
\begin{split}
	\int_{R_2} d_{\infty}((z_1,z_2),\partial R) dz_1 dz_2=&  \int_{\tilde T_1} d_{\infty}(z,\partial R)dz+  \int_{\tilde T_2} d_{\infty}(z,\partial R)dz \\
	&+ \int_{\tilde T_3} d_{\infty}(z,\partial R)dz+ \int_{\tilde T_4} d_{\infty}(z,\partial R)dz+ \int_{\tilde H} d_{\infty}(z,\partial R)dz.
\end{split}
\end{equation}
Here $\tilde T_1$ denotes the triangle with vertices
$$ \left\{\left(\frac{a}{2}-\frac{b}{2}, 0\right),\,\left(\frac{a}{2}-\frac{b}{2}, \frac{b}{2}\right),\, \left(\frac{a}{2}, \frac{b}{2}\right)  \right\},$$
 $\tilde T_2 $ denotes the triangle with vertices 
$$ \left\{\left(\frac{a}{2}-\frac{b}{2}, 0\right),\,\left(\frac{a}{2}, 0\right),\, \left(\frac{a}{2}, \frac{b}{2}\right)  \right\}$$
and $\tilde T_3, \tilde T_4$ are the triangles obtained reflecting $\tilde T_1$, $ \tilde T_2$ with respect to the horizontal axis, respectively.
 Finally $\tilde H$ is the rectangle with vertices
$$ \left\{\left(\frac{a}{2}-x, -\frac{b}{2}\right),\,\left(\frac{a}{2}-x, \frac{b}{2}\right),\, \left(\frac{a}{2}-\frac{b}{2}, \frac{b}{2}\right), \, \left(\frac{a}{2}-\frac{b}{2}, -\frac{b}{2}\right)   \right\}.$$
By a change of variables for every $ i \in \{1,2,3,4\}$
$$  \int_{\tilde T_i} d_{\infty}(z,\partial R)dz=  \int_{\tilde T} d_{\infty}(z,[0,\frac{b}{2}]\times \{0\}) dz= \int_{\tilde T} z_2 dz_1 dz_2$$
where $\tilde T$ is the triangle with vertices $(0,0), (\frac{b}{2},0),(0,\frac{b}{2})$.
A straightforward computation as in the previous case gives
$$ \int_{R_2}d_{\infty}(z,\partial R)dz=-\frac{b^3}{24}+\frac{x b^2}{4}.$$
Therefore by \eqref{fordisp03032023}, \eqref{for103032023} and the equality above  
\begin{equation}\label{dissipazionecattiva}
	D(\tilde{R},R)= \int_{\tilde{R}\Delta R}d_{\infty}(z,\partial R)dz= y^2(a-2x)-\frac{b^3}{12}+\frac{x b^2}{2}.
\end{equation} 
 Eventually, by \eqref{perimeterinXY}, \eqref{dissipazionebuona}, \eqref{dissipazionecattiva}, we can write the energy $$ \tilde{R} \rightarrow \mathrm{P}_{\vert \cdot \vert_{1}}(\tilde{R})+ \frac{1}{\tau} \mathcal{D}(\tilde{R},R)$$ as a function of the variables $x$ and $y $ as follows
 \begin{equation}\label{energiacontinXY}
 	E(x,y):=
 	\left\{
 	\begin{aligned}
 		& 2(a-2x+b+2y)+\frac{1}{\tau} \left[y^2(a-2x)+\frac{4}{3}x^3+x^2(b-2x)\right], & \text{ if } 0 \leq x < \frac{b}{2} \, , \\
 		& 2(a-2x+b+2y)+\frac{1}{\tau} \left[y^2(a-2x)-\frac{b^3}{12}+\frac{b^2}{2}x\right], & \text{ if }   \frac{b}{2} \leq x < \frac{a}{2} \, . 
 	\end{aligned}
 	\right.
 \end{equation}
\\
\textit{Step 3)}\\
In this step we want to compute the minimzer of the function $ E$ under the 
 area constraint $y=y(x)=\frac{xb}{a-2x}$. We distinguish two cases. \\

\textit{Case $x \in [\frac{b}{2},\frac{a}{2})$}. 
 In this case the energy can be written as
$$E(x,y(x))= \frac{1}{\tau} \left[2a\tau +2b\tau- \frac{b^3}{12}\right]+ \frac{1}{\tau} \left[\frac{x(a(b^2-8\tau)+8\tau b+ 16 \tau x)}{2a-4x}\right].$$
We observe that
\begin{equation}\label{17052023pom1}
	\frac{d}{dx}E(x,y(x))= \frac{1}{\tau} \left[ \frac{a^2b^2-8\tau a^2+8 \tau ab +32 \tau a x -32 \tau x^2}{2a^2-8ax+8x^2}\right].
\end{equation} 
 We claim the existence of $\tau_1:=\tau_1(\Lambda)>0$ such that for all $ \tau < \tau_1$,  $ \frac{d}{dx}E(x,y(x))>0$ for every $x $.
 
To this end, since the denominator is positive, it is enough to prove  the positivity of the numerator in \eqref{17052023pom1} for $\tau$ small enough. The latter follows by the chain of inequalities 
\begin{equation}
	a^2b^2-8\tau a^2+8 \tau ab +32 \tau a x -32 \tau x^2 > a^2 b^2 -(8\tau a^2+8 \tau ab +32 \tau a x +32 \tau x^2) >0
\end{equation}
which holds true 
for $ \tau <\tau_1:= \Lambda^{-k_1}$ with  $k_1 \in \N $ such that
\begin{equation}
	(8\tau a^2+8 \tau ab +32 \tau a x +32 \tau x^2) < Ca^2\Lambda^{-k_1}<\frac{1}{2}< a^2b^2=1,
\end{equation}
where  we have used that $x<a$ and $b< a$. 
By the same argument we have that there exists $ \tau_2(\Lambda)$  such that for all $ \tau < \tau_2(\Lambda) $
\begin{equation*}\label{limiEx>b/2}
	\lim_{x \rightarrow \frac{a}{2}^{-}} E(x,y(x))= +\infty.
\end{equation*}
 By the previous argument the only minimizer of $x \rightarrow E(x,y(x))$ for $ x \in \left[\frac{b}{2},\frac{a}{2}\right) $ is $ x= \frac{b}{2}$ where the energy takes the value 
$$ E\left(\frac{b}{2}, y\left(\frac{b}{2}\right)\right)=\frac{1}{\tau} \left[ \frac{24\tau a^2 +2 a b^3-24\tau ab +b^4+24 \tau b^3}{12a-12b}\right] .$$
\textit{Case $ 0\leq x\leq \frac{b}{2}$}. 
 In this case we have
\begin{equation*}
\begin{split}
	E(x,y(x))=& \frac{1}{3 \tau (a-2x)}[6a^2 \tau +6 ab \tau +3ab x^2 \\ &-24a \tau x-2ax^3+3b^2 x^2-6bx^3+24\tau x^2 +4 x^4]
\end{split} 
\end{equation*}
and
\begin{multline}\label{derivatabuona12}
	\frac{d}{d x}E(x, y(x))= \frac{1}{\tau (a-2x)^2} \big[ -8 x^4 +x^3(2a+8b) + x^2(-8ab -2b^2 -16\tau -2a^2)\\ +x(2a^2b+2ab^2+16a \tau)+4ab \tau -4a^2 \tau\big].
\end{multline}
We now look for stationary points of the energy $E$.
To this end we investigate the zeros of the numerator of \eqref{derivatabuona12} that we denote by 
\begin{equation}\label{11092023pom2}
	f(\tau,x):= -8 x^4 +x^3(2a+8b) + x^2(-8ab -2b^2 -16\tau -2a^2) +x(2a^2b+2ab^2+16a \tau)+4ab \tau -4a^2\tau.
\end{equation}
It holds that
\begin{equation}
	\frac{\partial f}{\partial x} (\tau,x)= -32 x^3 +x^2(6a +24 b)+x(-16ab -4b^2 -32\tau -4 a^2)+(2a^2b+ 2ab^2 +16 a \tau)
\end{equation}
and 
\begin{equation}\label{28032023pome1}
	\frac{\partial^2 f}{\partial^2 x} (\tau,x)= -96 x^2 +6x(2a+8b)+2(-8ab-2b^2-2a^2)-32 \tau.
\end{equation}
We observe that there exists $C_1(\Lambda)>0$ such that
\begin{equation}\label{22032023form1}
	\max_{x \in [0,\frac{b}{2}]} \vert f(\tau,x)-f(0,x) \vert  \leq \tau C_1(\Lambda).
\end{equation}
By \eqref{22032023form1} 
there exists $\tau_3:= \tau_3(\Lambda)$ such that for all $\tau <\tau_3$ the points $(\tau,x)$ such that $f(\tau,x)=0$ are in a neighbourhood of the zeros of
  $ f(0,\cdot)$. 
We study the roots of the polynomial
$$f(0,x)=-8x^4+x^3(2a+8b)+x^2(-8ab-2b^2-2a^2)+x(2a^2b + 2ab^2).$$
We observe that they are either
 $x=0$ or the roots of the polynomial
$$ p(x):=-8 x^3+x^2(2a+8b)+x(-8ab-2b^2-2a^2)+(2a^2b + 2ab^2).$$
We note that thanks to the assumption on $a$ and $b$
$$ \frac{d}{dx} p(x)= -24 x^2 +x(4a+16b)-2a^2-8ab-2b^2 \neq 0.$$
Since  $\frac{d}{d x}p(0)<0 $ we have that $\frac{d}{d x}p(x)<0 $ for all $ x \in \left[0, \frac{b}{2}\right]$, hence
 $p(x)$ is not increasing. We have that $ p(0)= 2a^2b+2ab^2>0$ and that $p(\frac{b}{2})= \frac{1}{2} ab (2a-3b)$. If $2a-3b=0$ then $\frac{b}{2}$ is the unique root of $p$. In order to find the other roots the following two cases need to be
  discussed: $2a-3b>0$ and $2a-3b<0$. In the case $2a-3b>0$ we have
  $p(\frac{b}{2})>0$, hence $p$ has no root in $[0, \frac{b}{2}]$.  
  If instead $ 2a -3b<0$,  $p(\frac{b}{2})<0$.  Using that $ p(0)>0$ and that $p$ is strictly decreasing  there exists a unique $ \bar x \in (0, \frac{b}{2})$ such that $p(\bar x)=0$ and it holds
\begin{equation}\label{11092023pom1}
	p(x)>0 \text{ if } x \in [0,\bar x) \text{ and } p(x)<0 \text{ if } x \in \big(\bar x, \frac{b}{2}\big] .
\end{equation} 
Thanks to the discussion above the roots $x_i$ of $f(0,x)$ are
\begin{itemize}
	\item[i)] if $ 2a-3b<0$ then $x_1=0, \,\,x_2= \bar x$.
	\item[ii)] if $ 2a- 3b=0$ then $x_1=0,\,\, x_2=\frac{b}{2}$,
	\item[iii)] if $ 2a-3b>0$ then $ x_1=0$.
\end{itemize}
We have that 
\begin{equation}
	f(0,0)=0 \quad \text{and} \quad \frac{\partial f}{\partial x}  (0,0)= 2a^2b+2ab^2>0 \quad \text{for all } a,b>0. 
\end{equation}
In what follows we prove the thesis only in the \textit{Case i)}. We first claim that for all $\Lambda>a+b$ there exists $\hat \delta:= \hat\delta(\Lambda)$ such that for all $ a,b>0$ and $a+b< \Lambda$
\begin{equation}\label{21032021derivpos}
	\frac{\partial f}{\partial x}(\tau,x)>0 \quad\text{for all }(\tau,x) \in (-\hat \delta,\hat \delta)\times(-\hat \delta,\hat \delta). 
\end{equation}
As a consequence of this claim and the implicit function Theorem, for every $a,b>0$ and $a+b \leq \Lambda$ there exists $ \bar x: (-\frac{\hat \delta}{2},\frac{\hat \delta}{2}) \rightarrow \R$ smooth function such that $ f(\tau, \bar x(\tau))=0$ and $\bar x(0)=0$.
We begin by observing that for $ \tau,\vert x \vert \leq \delta $ we have
\begin{equation}
	\begin{split}
\frac{\partial f}{\partial x}(\tau,x)=	&	-32 x^3 +x^2(6a +24 b)+x(-16ab -4b^2 -32\tau -4 a^2)+(2a^2b+ 2ab^2 +16 a \tau)   \\
		\geq &-32 x^3 -x^2(6a +24 b)+x(-16ab -4b^2 -32\tau -4 a^2)+(2a^2b+ 2ab^2 )  \\
		\geq &-32 \delta^3 -\delta^2(6a +24 b+32)-\delta(+16ab +4b^2 +4 a^2)+(2a^2b+ 2ab^2 )\\
		&= h(\delta) +2ab(a+b) ,
	\end{split}
\end{equation}
where
\begin{equation}
	h(\delta):= -32 \delta^3 -\delta^2(6a +24 b+32)-\delta(+16ab +4b^2  +4 a^2).
\end{equation}
We check that there exists $\hat \delta:=\hat \delta(\Lambda)>0$ such that for all $ \delta < \hat \delta$ it holds that $-h(\delta) \leq 3  $.
To this ends it is enough to observe that $-h(\delta) \leq C( \delta^3+\Lambda \delta^2 + \delta \Lambda^2)  $.

As a consequence, since $ab=1$ implies $a+b>2$ we have that
\begin{equation}\label{28032023mattina1}
	\frac{\partial f}{\partial x}(\tau,x) \geq 2ab(a+b)-3\geq 1 \quad \text{ for all } \tau \in (-\hat \delta,\hat \delta), \; \text{ and for all } x \in ( -\hat \delta, \hat\delta ).
\end{equation}
Let now $(0,x_2) $ be such that $f(0,x_2)=0$. Hence by the very definition of $p$ we have that for all $ a,b>0$
\begin{equation}
	\frac{\partial f}{\partial x}(0,x_2)=x_2 \frac{d}{dx} p(x_2)<0.
\end{equation} 
Next we claim that for all $\Lambda>a+b$ there exist $\tilde{\delta}:=\tilde{\delta}(\Lambda)$ such that for all $a,b>0$ and $a+b< \Lambda$
\begin{equation}\label{22032023form4}
	\frac{\partial f}{\partial x}(\tau,x)<0 \quad \text{for all }(\tau,x) \in (-\tilde{\delta},\tilde{\delta})\times (x_2-\tilde{\delta},x_2+\tilde{\delta}).
\end{equation}
As consequence of this claim thanks to the implicit function Theorem we have that for every $a,b>0$ and $a+b \leq \Lambda$ there exists a smooth function  $ \tilde x: (-\frac{\tilde{\delta}}{2},\frac{\tilde{\delta}}{2}) \rightarrow \R$ such that $ f(\tau, \tilde x(\tau))=0$ and $\tilde x(0)=x_2$.
We now prove the claim. 
We first observe that for all $a,b$ it holds that $ x_2 > \Lambda^{-k_0}$ for $k_0 \in \N$ big enough. This is a consequence of the following estimate
$$p(\Lambda^{-k_0})>2 (a+b) -(\Lambda^{-3k_0}+\Lambda^{-2k_0}+\Lambda^{-k_0}) C(\Lambda^2+\Lambda)>0$$
and of formula \eqref{11092023pom1}.
We observe that there exists $ C_2(\Lambda)$ 
\begin{equation}\label{18052023mat1}
	\max_{x \in[0,\frac{b}{2}]}  \left| \frac{\partial f}{\partial x}(\tau,x)- \frac{\partial f}{\partial x}(0,x) \right| \leq \tau C_2(\Lambda)
\end{equation}
and that
\begin{equation}\label{18052023mat2}
	\begin{split}
	\frac{\partial f}{\partial x}(0,x_2)=&x_2 \frac{d}{dx} p(x_2)< \Lambda^{-k_0} \max_{x \in[0,\frac{b}{2}]} \frac{d}{d x} p(x) \\&= \frac{\Lambda^{-k_0}}{6}(-11a^2-4ab-4b^2)< \frac{-4 \Lambda^{-k_0}}{6}<0.
\end{split} 
\end{equation}
The claim follows thanks to \eqref{18052023mat1}, \eqref{18052023mat2}. Defining $\tau_0:= \frac{\min\{ \tilde{\delta}, \hat \delta\}}{2}$ we have proved that in \textit{Case i)} for all $ \tau < \tau_0$
\begin{equation}
	f(\tau,x) <0 \text{ for all } x \in (0, \bar x(\tau)) \cup (\tilde x(\tau),\frac{b}{2}) \text{ and } f(\tau,x)>0 \text{ for all } x \in (\bar x(\tau),\tilde x(\tau))) .
\end{equation}

Repeating the same argument as above in \textit{Case ii)} we have that there exists $\tau_0:= \tau_0(\Lambda)$ such that if $ \tau < \tau_0$
 then either $ \tilde x(\tau)< \frac{b}{2}$ and
\begin{equation}
	f(\tau,x) <0 \text{ for all } x \in (0, \bar x(\tau)) \cup (\tilde x(\tau),\frac{b}{2}) \text{ and } f(\tau,x)>0 \text{ for all } x \in (\bar x(\tau),\tilde x(\tau))) ,
\end{equation}
or  $\tilde x(\tau)\geq \frac{b}{2}$ and
\begin{equation}
	f(\tau,x) <0 \text{ for all } x \in (0,\bar x(\tau)) \text{ and } f(\tau,x) >0 \text{ for all } x \in (\bar x(\tau), \frac{b}{2}) .
\end{equation} 
If instead \textit{Case iii)} holds, then for all 
$ \tau < \tau_0$
\begin{equation}
	f(\tau,x) <0 \text{ for all } x \in (0,\bar x(\tau)) \text{ and } f(\tau,x) >0 \text{ for all } x \in (\bar x(\tau), \frac{b}{2}) .
\end{equation}
As a consequence of the previous results we obtain the existence of $ \tau_0$ such that for all $ \tau < \tau_0$ the minimizer of $x \rightarrow E(x,y(x))$ is either $\bar x(\tau)$ or $\frac{b}{2}$.
 We now compute the first coefficient of the Taylor expansion of the function $\tau \rightarrow \bar x(\tau)$ in the interval $ (-\tau_0,\tau_0)$ at the point $\tau=0$. Substituting $\bar x(\tau)= \bar x'(0) \tau +o(\tau)$ in \eqref{derivatabuona12} we have that 
\begin{equation}
	\frac{d}{d x}E(\bar x(\tau),y(\bar x(\tau)))=0 \iff o(\tau)+ \bar x'(0)(2a^2b +2ab^2)\tau+ 4ab\tau-4a^2\tau=0,
\end{equation}
hence
\begin{equation}\label{formulafondamentale1}
\bar x'(0)= \lim_{\tau \rightarrow 0^+}\frac{2(-b+a)}{b(a+b)}+ \frac{o(\tau)}{\tau(2a^2 b+2ab^2)}=\frac{2(-b+a)}{b(a+b)}.
\end{equation}
We check that $\bar x(\tau)$ is the only minimizer of $ x \rightarrow E(x, y(x))$ in $ [0,\frac{a}{2}]$. To this end it is enough to check that
\begin{equation}
	E(\bar x(\tau),y(\bar x(\tau)))< E\left(\frac{b}{2},y\left(\frac{b}{2}\right)\right) 
\end{equation}
or equivalently that
\begin{equation}
	\frac{6a^2 \tau + 6ab \tau + o(\tau)}{3\tau a+ o(\tau)}<\frac{1}{\tau} \left[ \frac{24\tau a^2 +2 a b^3-24\tau ab +b^4+24 \tau b^3}{12a-12b}\right]  \text{ for all } \tau< \tau_0.
\end{equation}
The above inequality holds true provided $\tau_0$ is chosen sufficiently small.
  For the time being we have proved that for all $\Lambda>a+b$ there exists $ \tau_0:= \tau_0(\Lambda)$ such that for all $\tau< \tau_0$ the only minimizer of the energy $ x \rightarrow E(x,y(x))$ is $\bar x(\tau)$ and moreover by \eqref{formulafondamentale1}
\begin{equation}\label{11092023pom4}
	\bar x(\tau)=  \tau\frac{2(-b+a)}{b(a+b)} + o(\tau).
\end{equation}
By \eqref{28032023mattina1} and \eqref{11092023pom2}, applying the implicit function Theorem, we have that for all $\tau \in [0,\tau_0]$
\begin{equation*}
	\left| \frac{d}{d \tau} \bar x(\tau) \right| = \left| \frac{\frac{\partial f}{\partial \tau}(\tau,\bar x(\tau))}{\frac{\partial f}{\partial x}(\tau,\bar x(\tau))}\right| \leq \vert -16(\bar x(\tau))^2+16a\bar x(\tau) +4ab -4a^2 \vert \leq C\Lambda^2.
\end{equation*}
Moreover, using that $ \frac{\partial^2 f}{\partial \tau^2 }=0$, the estimate above together with \eqref{28032023mattina1}, \eqref{28032023pome1} and the implicit function Theorem gives
\begin{equation}\label{11092023pom3}
	\left| \frac{d^2}{d \tau^2} \bar x(\tau) \right| = \left| \frac{-\frac{d}{d \tau} \bar x(\tau) \frac{\partial^2 f}{\partial x^2} (\tau, \bar x(\tau))+\frac{\partial^2 f}{\partial \tau^2} (\tau, \bar x(\tau)) }{ \left(\frac{\partial f}{\partial x}(\tau, \bar x(\tau)) \right)^2} \right| \leq C\Lambda^2.
\end{equation}
Thanks to \eqref{11092023pom3} and \eqref{11092023pom4} we can write that $\bar x(\tau)= \bar x'(0) \tau+ O(\tau^2)$, with $\bar x'(0)$ as in \eqref{formulafondamentale1}.
Recalling the area constraint $y=\frac{bx}{a-2x}$, we now define the function $ y(\tau):= \frac{b\bar x(\tau) }{a-2\bar x(\tau)}$. We eventually have that 
the minimizer $R'= \left[-\frac{a'}{2}, \frac{a'}{2}\right] \times \left[-\frac{b'}{2}, \frac{b'}{2}\right]$ satisfies the following properties:
\begin{equation}
	\begin{split}
		& a'= a-2 \bar x(\tau) ,\\
		&  b'= b+2y (\tau) , \\
		&  \vert a'-a \vert = 2\vert \bar x(\tau) \vert \leq \tau C(\Lambda) \text{ and } \vert b'-b \vert = 2 \vert y(\tau) \vert \leq \frac{b\tau C(\Lambda) }{\vert a-2 \bar \bar x(\tau) \vert }\leq \tau C(\Lambda), \\
		&  \vert R \Delta R' \vert = 2b\bar x(\tau)+ 2(a-2\bar x(\tau))y(\tau)= 4b \bar x(\tau) \leq \tau C(\Lambda).
	\end{split}
\end{equation}
 By the minimality of $R'$ we have
\begin{equation}
	\begin{split}
	2a'+2b' = &\, \mathrm{P}_{\vert \cdot \vert_{1}}(R') \leq  \mathrm{P}_{\vert \cdot \vert_{1}}(R')+ \frac{1}{\tau} \mathcal{D}(R',R) \\
	\leq &\, \mathrm{P}_{\vert \cdot \vert_{1}}(R)+ \frac{1}{\tau} \mathcal{D}(R,R) = 2a+2b \leq 2\Lambda,
\end{split}
\end{equation}
hence $ a'+b'\leq a+b \leq\Lambda$.
\end{proof}
\subsection{Flat $\mathrm{P}_{\vert \cdot \vert_1}$ area-preserving mean-curvature flow}
\begin{definition}\label{defflat}
	Let $R\subset \R^2$ be a rectangle with unitary area.
	By flat $ \mathrm{P}_{\vert \cdot \vert_1}$ area-preserving mean-curvature flow with initial datum $R_0$, we mean any family of sets $ \{E_t\}_{t \geq 0}$ such that   $\chi_{E_t} \in BV(\R^2)$ for all $ t \geq 0$ and
	$$ \lim_{k \rightarrow + \infty} \| \chi_{E_t^{(\tau_k)}}-\chi_{E_t}\|_{L^1(\R^2)}=0$$
	locally uniformly in $t$. In the formula above $ \{ E_{t}^{(\tau_k)}\}_{t\geq 0}$ denotes an approximate flat $\mathrm{P}_{\vert \cdot \vert_1}$ area-preserving mean-curvature flow with initial datum $ R$ and $\{\tau_k\}_{k \in \N}$ is a vanishing  sequence.
\end{definition}
\begin{theorem}\label{thm12092023}
	Let $R$ be a rectangle with horizontal sidelength $a_0$ and vertical sidelength $b_0$ such that $a_0 b_0=1$, $b_0 < a_0$ and $a_0+b_0<\Lambda$, for some $ \Lambda>0$. For any $ 0<\tau< \tau_0$ (here $\tau_0$ is as in Theorem \ref{THMesistenzadeiminimi} for $a=a_0$ and $b=b_0$) let $ \{R_t^{\tau}\}_{t\geq 0}$ be an approximate flat $\mathrm{P}_{\vert \cdot \vert_1}$ area-preserving mean-curvature flow with initial datum $R$.
	Then, for all $k\geq 1$
	\begin{equation}
		R_t^{\tau}= \left[-\frac{a((k+1)\tau)}{2}, \frac{a((k+1) \tau)}{2}\right] \times \left[-\frac{b((k+1)\tau)}{2}, \frac{b((k+1) \tau)}{2}\right] \quad \forall t \in [(k+1)\tau,(k+2)\tau),
	\end{equation}
where $a(k\tau)$ and $ b(k\tau)$ are obtained by recurrence setting $a(0)=a_0$, $b(0)=b_0$ and
\begin{equation}
	\begin{split}
		& a((k+1)\tau )= a(k\tau)-2x(k\tau) \text{ with } x(k\tau)= \frac{2(a(k\tau)-b(k\tau))}{b(k\tau)(a(k\tau)+b(k\tau))} \tau+ O(\tau^2) \\
		&  b((k+1)\tau)= b(k \tau)+2y (k\tau) \text{ with } y(k\tau)= \frac{b(k\tau)x(k\tau) }{a(k\tau)-2x(k\tau)}, 
	\end{split}
\end{equation}
and share the following properties:
\begin{equation}
	\begin{split}
		&  a((k+1)\tau) b((k+1)\tau)=1  \text{, } a((k+1)\tau)+b((k+1)\tau)<a(k\tau)+b(k\tau) < \Lambda, \\
		& \vert a((k+1)\tau)- a(k\tau) \vert \leq \tau C(\Lambda) \text{, }  \vert b((k+1)\tau)- b(k\tau) \vert \leq \tau C(\Lambda).
	\end{split}
\end{equation}
Moreover there exists a unique flat $\mathrm{P}_{\vert \cdot \vert_1}$ area-preserving mean-curvature flow $ \{R_t\}_{t \geq 0}$ such that, for all $t \geq0$, $R_t=\left[-\frac{\tilde a(t)}{2}, \frac{\tilde a(t)}{2}\right] \times \left[-\frac{\tilde b(t)}{2},\frac{\tilde b(t)}{2}\right]$ with $ \tilde{a}(0)=a_0, \tilde{b}(0)=b_0$. The following properties of the flow hold true:
\begin{equation}
	\begin{split}
&	\lim_{\tau \rightarrow 0} \vert R_{t}^{\tau} \Delta R_t \vert=0, \quad \lim_{\tau \rightarrow 0} a^{\tau}(t)=\tilde a(t), \quad \lim_{\tau \rightarrow 0} b^{\tau}(t)=\tilde b(t), \\
 & \frac{d}{d t} \tilde a(t)= \frac{-4}{\tilde b(t)}+\frac{8}{\tilde a(t)+\tilde b(t)}, \quad \frac{d}{d t} \tilde b(t)= \frac{-4}{\tilde a(t)}+\frac{8}{\tilde a(t)+\tilde b(t)} 
\end{split}
\end{equation}
where for all $t\geq0$ we have defined
$ a^\tau (t):=a((k+1)\tau)$ for all $t \in [(k+1)\tau,(k+2)\tau)$ and $ b^\tau (t):=b((k+1)\tau)$ for all $t \in [(k+1)\tau,(k+2)\tau)$.
\end{theorem}
\begin{remark}\label{remexp}
As a consequence of the theorem above, $R_t $ converge to the unit square exponentially fast. These properties can be checked by setting
	$$f(t):= \frac{\tilde a(t)}{\tilde b(t)}-1.$$
	and observing that, by the theorem above, it holds that 
	$$ \frac{d}{d t}f(t)= \frac{\frac{d}{d t} \tilde a(t)}{\tilde b(t)}- \frac{\tilde a(t)\frac{d}{d t}\tilde b(t)}{\tilde b(t)^2}=\frac{-8}{\tilde a(t) \tilde b(t)+\tilde b(t)^2}\left[ \frac{\tilde a(t)}{\tilde b(t)}-1\right]< -4f(t).$$
	The exponential convergence follows from $ f(t) \leq f(0)e^{-4t} $ which is a results of the Gronwall Lemma applied to the previous inequality.
\end{remark}
\begin{proof}[Proof of Theorem \ref{thm12092023}]
	The first part of the statement regarding the properties of the approximate flat flow is a consequence of Theorem \eqref{THMesistenzadeiminimi} applied iteratively to the sets $R= R_t^{\tau}$. By Theorem \ref{THMesistenzadeiminimi} the functions $a^\tau$ and $b^\tau$ are piecewise constant and for all $ s>t$ $ \vert a^{\tau}(t)- a^{\tau}(s) \vert \leq C(\Lambda) \vert s-t+ \tau\vert $. Therefore we have 
	$$ a^{\tau}(\cdot) \rightarrow \tilde a(\cdot)\quad b^{\tau}(\cdot) \rightarrow \tilde b(\cdot) \text{ as $ \tau \rightarrow 0^+$ uniformly on the compact sets of }[0,+\infty).$$
	We observe that $ \tilde a, \tilde b$ are Lipschitz functions. 
Hence for almost every $t\geq 0$ there exist the derivative of $ \tilde a(\cdot)$ and $ \tilde b(\cdot)$ and
\begin{equation}\label{29032023pom1}
	\frac{d}{d t}  \tilde a(t)= \frac{-4}{\tilde b(t)}+\frac{8}{ \tilde a(t)+\tilde b(t)}, \quad \frac{d}{d t} \tilde b(t)= \frac{-4}{\tilde a(t)}+\frac{8}{\tilde a(t)+\tilde b(t)}. 
\end{equation}
Since the right hand sides of both formulas in \eqref{29032023pom1} are  Lipschitz continuous functions we have that \eqref{29032023pom1} holds true for all $t \geq 0$. 
\end{proof}

	
	\section{The discrete setting}\label{Sec2}
	For all $\varepsilon>0$ we consider the square lattice of size $\varepsilon$, namely $\varepsilon \mathbb{Z}^2$. We define its dual lattice as $ (\frac{\varepsilon}{2},\frac{\varepsilon}{2})+ \varepsilon\mathbb{Z}^2$. 
	For all $\mathcal{J} \subset \varepsilon \mathbb{Z}^2$ and for all $p, q \in \mathbb{Z}$ we define the sets
	\begin{equation}\label{Jrighe}
		\mathcal{J}(\cdot,\, q):= \left\{\varepsilon(j_1,\, q) \, \colon \, \varepsilon(j_1,\, q) \in \mathcal{J}\right\}
	\end{equation} 
	and 
	\begin{equation}\label{Jcolonne}
		\mathcal{J}(p,\, \cdot):= \left\{\varepsilon(p,\, j_2) \, \colon \, \varepsilon(p,\, j_2) \in \mathcal{J}\right\}.
	\end{equation} 
	The set $\mathcal{J}(\cdot, \, q)$ (respectively $\mathcal{J}(p, \cdot)$) is the subset of $\mathcal{J}$ whose elements have second entry taking the value $\varepsilon q$ (respectively the first entry taking the value $\varepsilon p$).
	In order to pass form a discrete to a continuous formulation of our problem it is convenient to identify sets $ \mathcal{J} \subset \varepsilon \mathbb{Z}^2$ with subsets of $\mathbb{R}^2$. To this end
	to all $\mathcal{J} \subset \varepsilon \mathbb{Z}^2$ we associate the set 
	\begin{equation}\label{29052023sera1}
		E_{\mathcal{J}}= \bigcup_{j \in \mathcal{J}} Q_\varepsilon(j) \subset \mathbb{R}^2
	\end{equation} 
	where $Q_\varepsilon(j):= j+ \varepsilon Q$ and $ Q=[-\frac{1}{2}, \, \frac{1}{2}]^2$. 
	For all $\varepsilon \in (0,1)$ we define the set
	\begin{equation}\label{insiemeC_varepsilong}
		\mathcal{C}_{\varepsilon }:= \left\{ E \subset \mathbb{R}^2\, \colon \, E= E_{\mathcal{J}} \text{ for some } \mathcal{J}\subset \varepsilon \mathbb{Z}^2  \right\}.
	\end{equation} 
In what follows if $E \in \mathcal{C}_{\varepsilon,}$ 
we write $E= E_{\mathcal{J}}$ to refer to the set $E_{\mathcal{J}}$ in \eqref{insiemeC_varepsilong}.
For all $ \mathcal{J} \subset \varepsilon \mathbb{Z}^2$, we define the discrete $l^\infty$-distance to $\partial \mathcal{J}$ as
\begin{equation}\label{distconsegno}
	d^\varepsilon_{\infty}(i,\, \partial \mathcal{J}):= \text{dist}(i, \, \mathcal{J})+ \text{dist}(i, \, \varepsilon \mathbb{Z}^2 \setminus \mathcal{J}) \quad \text{for all $ i \in \varepsilon\mathbb{Z}^2$},
\end{equation}
where 
$$ \text{dist}(i, \, \mathcal{J}):= \inf \left\{ \| i-j\|_{\infty} \colon \, j\in \mathcal{J} \right\} \quad\text{for all $ i \in \varepsilon\mathbb{Z}^2$.}$$
We observe that, if  $d_{\infty}$ denotes the usual $l^\infty$-distance in $ \mathbb{R}^2$, then
$$ d^{\varepsilon}_{\infty}(i,\, \partial \mathcal{J})= d_{\infty}(i,\, \partial E_{\mathcal{J}})+ \frac{\varepsilon}{2}. $$
The distance defined above can be extended to all $ \mathbb{R}^2 \setminus \partial E_{\mathcal{J}}$ by setting
$$ d^{\varepsilon}_{\infty}(x,\, \partial \mathcal{J}):=d^{\varepsilon}_{\infty}(i,\, \partial \mathcal{J}) \quad \text{if $x \in Q_{\varepsilon}(i)$.}$$
The following two polygons will play an important role in what follows.
\begin{definition}[Quasi-rectangle]\label{quasirettangolo}
	We say that a set $E=E_{\mathcal{J}} \in \mathcal{C}_\varepsilon$ is a quasi-rectangle if it satisfies the following properties. The cardinality of $\mathcal{J}$ is $\# \mathcal{J}= nm+r$ for some $n,\,m,\,r \in \mathbb{N}$ with $r < n= \max\{n,\,m\}$ and $E$ can be written as $E=R\cup Q$ with $R$ a rectangle of horizontal sidelength $\varepsilon n$ and vertical sidelength $ \varepsilon m$ and $Q$ a rectangle with horizontal sidelength $\varepsilon r$ and vertical sidelength $ \varepsilon$ and such that for every $j \in Q$ there exists $\hat{j} \in R $ such that $ \vert j- \hat{j} \vert= \varepsilon$. The set of all quasi rectangles $E \in \mathcal{C}_{\varepsilon}$ is denoted by $\mathcal{QR}_{\varepsilon}$. 
	\end{definition}
	In what follows we will be using the notation $QR$ for the quasi-rectangle $R\cup Q$.
	\begin{definition}[Pseudo-axial quasi-rectangle]\label{paquasirettangolo}
	 We say that a quasi-rectangle $QR=R\cup Q\in\mathcal{QR}_{\varepsilon}$ is pseudo axial symmetric if 
	\begin{eqnarray*}
		\mathrm{Bar}(R)=
		\begin{cases} (0,\,0) &\text{ if $n,\, m$ are odd,} \\ \left(\frac{\varepsilon}{2},\,0\right) &\text{ if $n$ is even and $m$ is odd,} \\
	\left(0,\,\frac{\varepsilon}{2}\right) & \text{ if $n$ is odd and $m$ is even,} \\ \left(\frac{\varepsilon}{2},\,\frac{\varepsilon}{2}\right) &\text{ if $n,\, m$ are even,}
	\end{cases}
	\end{eqnarray*}
	
	and 
	\begin{eqnarray*}
	Q= \begin{cases}
	\left[- \varepsilon\frac{r}{2}, \, \varepsilon\frac{r}{2}  \right] \times \left[L,\, L+\varepsilon\right] &\text{if $ r $ is odd,}\\
	\left[ -\varepsilon \frac{r}{2}+ \frac{\varepsilon}{2}, \, \varepsilon\frac{r}{2}+ \frac{\varepsilon}{2}  \right] \times \left[L,\, L+\varepsilon\right]& \text{if $ r $ is even } 
	\end{cases}
	\end{eqnarray*}
	where $L:=\mathcal{H}^{1}(\Pi_{2}(R)\cap (\{0\}\times[0,\,+\infty)))$.
\end{definition}


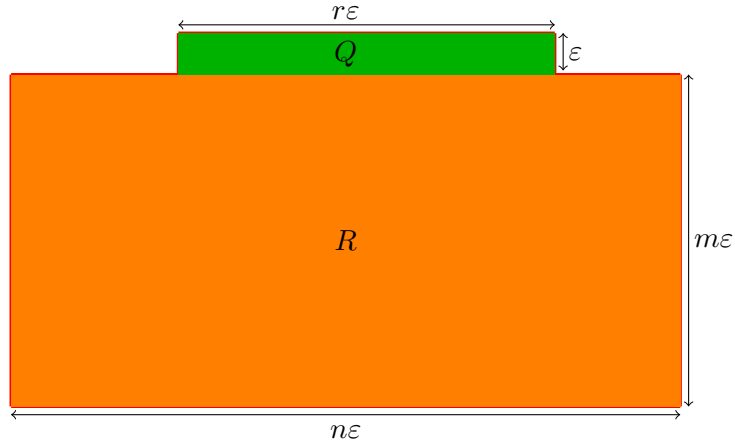
\begin{figure}[h!]
    \centering
    \begin{tikzpicture}[scale=0.55]
        \draw[red,very thick] 
        (-8,-4) -- (8,-4);
        \draw[red,very thick] 
        (-8,-4) -- (-8,4);
        \draw[red,very thick] 
        (8,-4) -- (8,4);
        \draw[red,very thick] 
        (-8,4) -- (-4,4);
        \draw[red,very thick] 
        (-4,4) -- (-4,5);
        \draw[red,very thick] 
        (-4,5) -- (5,5);
        \draw[red,very thick] 
        (5,4) -- (5,5);
        \draw[red,very thick] 
        (8,4) -- (5,4);
        \draw
        (5,4) -- (-4,4);
        \fill[orange, opacity=0.2]
            (-8,-4) rectangle (8,4);
        \fill[green!70!black, opacity=0.2]
            (5,4)  rectangle  (-4,5);
        \draw [black,very thick]    
        (0,0) node{$R$} ;
        \draw [black,very thick]    
        (0,4.51) node{$Q$} ;
        \draw[<->] 
        (-8,-4.2) -- (8,-4.2);
        \draw [black,very thick]     (0,-4.6) node{$n \varepsilon$};
        \draw[<->] 
        (8.2,-4) -- (8.2,4);
        \draw [black,very thick]     (8.8,0) node{$m \varepsilon$};
        \draw[<->] 
        (5.2,4.1) -- (5.2,5);
        \draw [black,very thick]     (5.5,4.5) node{$ \varepsilon$};
        \draw[<->] 
        (-4,5.2) -- (5,5.2);
        \draw [black,very thick]     (0,5.5) node{$  r \varepsilon$};
    \end{tikzpicture}
    \caption{The quasi-rectangle $QR= R \cup Q \in \mathcal{QR}_\varepsilon$}
    \label{fig3}
\end{figure}

\begin{definition}[Rhombus-like Shape]\label{formadelrombo}
	We say that a set $E=E_{\mathcal{J}} \in \mathcal{C}_{\varepsilon}$ has a rhombus-like shape with respect to $(0,0) \in \mathbb{R}^2$ if two functions
	$$ i \in \mathbb{Z} \longrightarrow \# \mathcal{J}^{}(i,\, \cdot) \textrm{ and } i \in \mathbb{Z} \longrightarrow \# \mathcal{J}^{}(\cdot,\, i)$$
	are increasing for $i\leq 0$ and non-increasing for $i \geq 0$.
\end{definition}
We now define the connectedness by rows or by columns.
\begin{definition}
	We say that a set $E=E_{\mathcal{J}}\in \mathcal{C}_{\varepsilon}$ is connected with respect to rows if for all $j \in \mathbb{Z}$ we have that
$$\varepsilon(i,\,j)\in \mathcal{J} \quad\forall (i,\,j) \in \mathbb{Z}^2 \quad \text{such that } \quad i \in \{i_{min}(j),\dots, \, i_{max}(j)\} $$ 
	where 
	$$ i_{min}(j):= \min \{i \in \mathbb{Z}\colon\, \varepsilon(i,\,j) \in \mathcal{J}\} \textrm{ and } i_{max}(j):=\max \{i \in \mathbb{Z}\colon\, \varepsilon(i,\,j) \in \mathcal{J}\}.$$
The connectedness with respect to columns is defined analogously.
\end{definition}
For all $ \mathcal{J} \subset \varepsilon \mathbb{Z}^2$ we define the discrete perimeter as 
\begin{equation}\label{perimetrodiscreto}
	\mathrm{P}^{\varepsilon}(\mathcal{J})= \varepsilon \,\# \left\{   (i,j) \in \mathbb{Z}^2 \times \mathbb{Z}^2 \, \colon \; \e i \in \mathcal{J},\; \e j \in \varepsilon \mathbb{Z}^2 \setminus \mathcal{J},\; \vert i- j \vert= 1 \right\}.
\end{equation}
We remark that for all $ \mathcal{J} \subset \varepsilon \mathbb{Z}^2$
\begin{equation}\label{EnergiaPer}
	\mathrm{P}^{\varepsilon}(\mathcal{J})=\mathrm{P}(E_{\mathcal{J}})=\mathcal{H}^1(\partial E_{\mathcal{J}}).
\end{equation}
In what follows, for $E\in\mathcal{C}_{\varepsilon}$ we will also drop the dependence on $\mathcal{J}$ and simply write $\mathrm{P}(E)=\mathrm{P}(E_{\mathcal{J}})$.
In the following we also make use of the next two definitions of perimeter of the horizontal and vertical projection of a set $E=E_{\mathcal{J}}$, namely
	\begin{equation}\label{P12} 
	\mathrm{P}_1(E):= \mathcal{H}^1(\Pi_{1}(E))=\max_{q \in \mathbb{Z}}\varepsilon\,\#\mathcal{J}(\cdot,\, q) \textrm{ and } \mathrm{P}_2(E):= \mathcal{H}^1(\Pi_{2}(E))=\max_{p \in \mathbb{Z} } \varepsilon\,\#\mathcal{J}(p,\, \cdot)
	\end{equation}
and observe that 
\begin{equation}\label{P12ineq}
	\mathrm{P}(E)\geq 	2\mathrm{P}_1(E)+2\mathrm{P}_2(E).
\end{equation}
For every $E,\,F \in \mathcal{C}_{\varepsilon}$ we introduce the $\e$-dissipation of $E$ and $F$ as
\begin{equation}\label{dissipazioneD_epsilon}
	D_\varepsilon(F,\, E):=\int_{F \Delta E} d^{\varepsilon}_{\infty}(x, \, \partial E)\,dx.
\end{equation}


\subsection{Discrete Steiner-like rearrangement of a set with respect to a quasi rectangle }\label{Section:Steiner-eps}
This section is devoted to the definition of a geometric rearrangement of a set with respect to a quasi-rectangle with pseudo axial symmetry. The procedure will turn the set into a rhombus-like shape with the same area and smaller perimeter. The precise construction is given in the proof of the next Proposition \ref{steinercostrizionelemma}. 
\begin{proposition}\label{steinercostrizionelemma}   Let $QR \in \mathcal{QR}_\varepsilon$ have pseudo axial symmetry and
	let us consider $E=E_{\mathcal{J}} \in \mathcal{C}_{\varepsilon}$. The discrete Steiner-like rearrangement $\mathcal{R}(E)\in \mathcal{C}_{\varepsilon}$ of the set $E$ with respect to $QR$ has a rhombus-like shape and is such that 
	\begin{equation}\label{proprietinedisteinerlike}
		\vert E \vert = \vert \mathcal{R}(E)\vert, \quad\mathrm{P}(E) \geq \mathrm{P}(\mathcal{R}(E)).
	\end{equation}
\end{proposition}
\begin{proof}
Let $QR= R \cup Q$ be a quasi-rectangle with pseudo axial symmetry. Let $n,\, m,\, r \in \mathbb{N}$ be such that $ \varepsilon n$ and $\varepsilon m$ are the horizontal and vertical sidelengths of $R$ and $ \varepsilon r$ and $\varepsilon$ are the horizontal and vertical sidelengths of $Q$.
	 We define the integer numbers $\underline p,\, \overline p, \, \overline q$ as
	\begin{equation*}
		 \underline p:= \min \left\{ p\colon\, \varepsilon(p,\,q) \in Q  \right\},\,
	 \overline p:= \max \left\{ p\colon\, \varepsilon(p,\,q) \in Q \right\},\,  \overline q:= \max \left\{ q \colon\, \varepsilon(p,\,q) \in QR \right\}. 
 \end{equation*}
	
	We divide the proof in two steps.	\\
	\textit{Step 1)}\\
	In this step we construct the set $E^{'}$ obtained via a symmetrization procedure of the rows of the set $E$.  
	Let $\mathcal{J}(p,\cdot),\, \mathcal{J}(\cdot,\, q)$ be the sets defined in \eqref{Jcolonne} and \eqref{Jrighe}.
	For all $q \in \mathbb{Z} $ we define the symmetrized
	set of $ \mathcal{J}(\cdot,\,q)$ with respect to the vertical axis as 
	$$\mathcal{J}^{'}(q):= \left\{ \varepsilon (j,\,q)\colon \, j \in \left\{ -\frac{\# \mathcal{J}(\cdot,\, q)-1 }{2}, \dots , \frac{\# \mathcal{J}(\cdot,\, q)-1 }{2} \right\}\right\} $$ 
	if $ \#  \mathcal{J}(\cdot,\, q)$ is odd, 
	and we define 
	$$\mathcal{J}^{'}(q):= \left\{ \varepsilon (j,\,q)\colon \, j \in \left\{ -\frac{\# \mathcal{J}(\cdot,\, q) }{2}+1, \dots , \frac{\# \mathcal{J}(\cdot,\, q) }{2} \right\}\right\} $$ 
	if $ \#  \mathcal{J}(\cdot,\, q)$ is even.

	We define the subsets $\mathcal{J}^{'} \subset \varepsilon \mathbb{Z}^2$ and $ E^{'} \subset \mathbb{R}^2$ as
	\begin{equation}\label{insiemeJ'def}
		\mathcal{J}^{'}:=  \bigcup_{q \in \mathbb{Z} } \mathcal{J}^{'}(q)\quad \text{and} \quad E^{'}:= E_{\mathcal{J}^{'}}.
	\end{equation}
	We notice that the set $E^{'}$ can be defined by recurrence as
	\begin{equation}\label{defE^{'}_{q+1}}
		\left\{
		\begin{aligned}
			& E_{q+1}^{'}:=\left(E^{'}_q \setminus E_{\mathcal{J}(\cdot,\,q)}\right) \cup E_{\mathcal{J}^{'}(q)} & \text{if $ q \geq q_{min} $,} \\
			& E_{q_{min}}^{'}=E & \text{if $q=q_{min}$}
		\end{aligned}
		\right.
	\end{equation}
	where 
	\begin{equation}\label{ordinataminimalediE}
		q_{min}:= \inf \left\{  q \in \mathbb{Z} \, \colon\, \, \mathcal{J}(\cdot,\,q) \neq \emptyset \right\}.
	\end{equation}
	We now observe that the function
	$$ i  \in \mathbb{Z} \longrightarrow \# \mathcal{J}^{'}(i,\, \cdot)$$ 
	is increasing for $ i\leq 0$ and non-increasing for $ i \geq 0$. Indeed,
	by construction, the set $ \mathcal{J}^{'}$ is connected with respect to rows, hence for all $j,\, q \in \N$ 
	$$\text{if} \quad \varepsilon(j+1,\, q) \in \mathcal{J}^{'} \quad \text{then} \quad \varepsilon(j,\, q) \in \mathcal{J}^{'} $$ which proves the monotonicity of the function for positive $i$. An analogous argument shows the monotonicity for $i$ negative. 

\begin{figure}[h!]
    \centering
    \begin{tikzpicture}[scale=0.5]
        \draw[red,ultra thick] 
        (-16,-3) -- (-4,-3);
        \draw[red,ultra thick] 
        (-16,3) -- (-14,3);
        \draw[red,ultra thick] 
        (-14,3) -- (-14,4);
        \draw[red,ultra thick] 
        (-14,4) -- (-6,4);
        \draw[red,ultra thick] 
        (-6,4) -- (-6,3);
        \draw[red,ultra thick] 
        (-6,3) -- (-4,3);
        \draw[red,ultra thick] 
        (-4,3) -- (-4,-3);
        \draw[red,ultra thick] 
        (-16,3) -- (-16,-3);
        \draw [red,ultra thick]     (-6,-3.6) node{$\partial QR $};
        \draw[blue,very thick] 
        (-17,1) -- (-17,5)--(-12,5) --(-12,1)-- cycle;
        \draw [blue,ultra thick]     (-12,5.5) node{$E $};
        \fill[blue, opacity=0.1]
        (-17,1) -- (-17,5)--(-12,5) --(-12,1)-- cycle;
        \draw[blue,very thick] 
        (-18,5)--(-20,5)--(-20,-3)--(-19,-3)--(-19,-4)--(-18,-4)-- cycle;
        \draw [blue,ultra thick]     (-18,5.5) node{$E $};
        \fill[blue, opacity=0.1]
        (-18,5)--(-20,5)--(-20,-3)--(-19,-3)--(-19,-4)--(-18,-4)-- cycle;
        \draw[blue,very thick] 
        (-11,-2)--(-11,0)--(-9,0)--(-9,-2)-- cycle;
        \draw [blue,ultra thick]     (-9,0.5) node{$E $};
        \fill[blue, opacity=0.1]
        (-11,-2)--(-11,0)--(-9,0)--(-9,-2)-- cycle;
        \draw[blue,very thick] 
        (-7,-2)--(-7,-1)--(-6,-1)--(-6,0)--(1,0)--(1,-2)-- cycle;
        \draw [blue,ultra thick]     (1,0.5) node{$E $};
        \fill[blue, opacity=0.1]
        (-7,-2)--(-7,-1)--(-6,-1)--(-6,0)--(1,0)--(1,-2)-- cycle;
        \draw[blue,very thick] 
        (-7,1)--(-7,5)--(-2,5)--(-2,1)-- cycle;
        \draw [blue,ultra thick]     (-2,5.5) node{$E $};
        \fill[blue, opacity=0.1]
        (-7,1)--(-7,5)--(-2,5)--(-2,1)-- cycle;
    \end{tikzpicture}
    \caption{The quasi-rectangle $QR $ and the set $E$}
    \label{fig4}
\end{figure}
	\textit{Step 2)}\\
		In this step we construct the set $\mathcal{R}(E)$ obtained from the set $E^{'}$ via a symmetrization of its columns. Also this operation  depends on the parity of $m$. In what follows we keep denoting by $p \in \mathbb{Z}$ the column index.
We first consider the case $m \in 2 \mathbb{N}$ . If $ p \in \{\underline p,\cdots,\overline p\}$ we define the symmetrized column according to the following two cases. If $ \# \mathcal{J}^{'}(p,\cdot) \in 2 \mathbb{N}+1$, then we define 
    	$$\mathcal{J}^{''}(p):= \left\{ \varepsilon (p,\,j)\colon \, j \in \left\{ 1-\frac{\# \mathcal{J}^{'}(p,\,\cdot)-1 }{2}, \dots ,1 + \frac{\# \mathcal{J}^{'}(p,\,\cdot)-1 }{2} \right\}\right\}. $$  
	If instead $\# \mathcal{J}^{'}(p,\cdot) \in 2 \mathbb{N}$, then we define 
    	$$\mathcal{J}^{''}(p):= \left\{ \varepsilon (p,\,j)\colon \, j \in \left\{ -\frac{\# \mathcal{J}^{'}(p,\,\cdot) }{2}+1, \dots , \frac{\# \mathcal{J}^{'}(p,\, \cdot) }{2} \right\}\right\}. $$

In the case $ p \in\mathbb{N} \setminus \{ \underline p,\cdots,\overline p\}$ we define the symmetrized column according to the following two cases.
If $ \# \mathcal{J}^{'}(p,\cdot) \in 2 \mathbb{N}+1$, then we define 
		$$\mathcal{J}^{''}(p):= \left\{ \varepsilon (p,\,j)\colon \, j \in \left\{ 1-\frac{\# \mathcal{J}^{'}(p,\,\cdot)-1 }{2}, \dots , 1+\frac{\# \mathcal{J}^{'}(p,\,\cdot)-1 }{2} \right\}\right\}. $$ 
If instead $\# \mathcal{J}^{'}(p,\cdot) \in 2 \mathbb{N}$, then we define 
		$$\mathcal{J}^{''}(p):= \left\{ \varepsilon (p,\,j)\colon \, j \in \left\{ -\frac{\# \mathcal{J}^{'}(p,\,\cdot) }{2}+1, \dots , \frac{\# \mathcal{J}^{'}(p,\, \cdot) }{2} \right\}\right\}. $$

Now we consider the case $m \in 2 \mathbb{N}+1$ . If $ p \in \{\underline p,\cdots,\overline p\}$ we have two possible cases.
If $ \# \mathcal{J}^{'}(p,\cdot) \in 2 \mathbb{N}+1$, then we define 
		$$\mathcal{J}^{''}(p):= \left\{ \varepsilon (p,\,j)\colon \, j \in \left\{ -\frac{\# \mathcal{J}^{'}(p,\,\cdot)-1 }{2}, \dots ,\frac{\# \mathcal{J}^{'}(p,\,\cdot)-1 }{2} \right\}\right\}.$$ 
If instead $\# \mathcal{J}^{'}(p,\cdot) \in 2 \mathbb{N}$, then we define 
		$$\mathcal{J}^{''}(p):= \left\{ \varepsilon (p,\,j)\colon \, j \in \left\{ -\frac{\# \mathcal{J}^{'}(p,\,\cdot) }{2}+1, \dots , \frac{\# \mathcal{J}^{'}(p,\, \cdot) }{2} \right\}\right\}. $$
%
\begin{figure}[h!]
    \centering
    \begin{tikzpicture}[scale=0.5]
        \draw[red,ultra thick] 
        (-16,-3) -- (-4,-3);
        \draw[red,ultra thick] 
        (-16,3) -- (-14,3);
        \draw[red,ultra thick] 
        (-14,3) -- (-14,4);
        \draw[red,ultra thick] 
        (-14,4) -- (-6,4);
        \draw[red,ultra thick] 
        (-6,4) -- (-6,3);
        \draw[red,ultra thick] 
        (-6,3) -- (-4,3);
        \draw[red,ultra thick] 
        (-4,3) -- (-4,-3);
        \draw[red,ultra thick] 
        (-16,3) -- (-16,-3);
        \draw [red,ultra thick]     (-6,-3.6) node{$\partial QR $};
        \draw[blue,very thick] 
        (-9,-4)--(-9,-2)--(-4,-2)--(-4,0)--(-9,0)--(-9,1)--(-3,1)--(-3,5)--(-16,5)--(-16,1)--(-11,1)--(-11,0)--(-15,0)--(-15,-1)--(-16,-1)--(-16,-2)--(-11,-2)--(-11,-3)--(-10,-3)--(-10,-4)--cycle;
        \fill[blue, opacity=0.1]
        (-9,-4)--(-9,-2)--(-4,-2)--(-4,0)--(-9,0)--(-9,1)--(-3,1)--(-3,5)--(-16,5)--(-16,1)--(-11,1)--(-11,0)--(-15,0)--(-15,-1)--(-16,-1)--(-16,-2)--(-11,-2)--(-11,-3)--(-10,-3)--(-10,-4)--cycle;
        \draw [blue,ultra thick]     (-2.5,0.8) node{$E' $};
    \end{tikzpicture}
    \caption{The quasi-rectangle $QR $ and the set $E'$}
    \label{fig5}
\end{figure}
In the alternative case that $ p \in\mathbb{N} \setminus \{ \underline p,\cdots,\overline p\}$ we define the new columns as follows.
If $ \# \mathcal{J}^{'}(p,\cdot) \in 2 \mathbb{N}+1$ then 
		$$\mathcal{J}^{''}(p):= \left\{ \varepsilon (p,\,j)\colon \, j \in \left\{ -\frac{\# \mathcal{J}^{'}(p,\,\cdot)-1 }{2}, \dots , \frac{\# \mathcal{J}^{'}(p,\,\cdot)-1 }{2} \right\}\right\}. $$ 
If instead $\# \mathcal{J}^{'}(p,\cdot) \in 2 \mathbb{N}$, then
		$$\mathcal{J}^{''}(p):= \left\{ \varepsilon (p,\,j)\colon \, j \in \left\{ -\frac{\# \mathcal{J}^{'}(p,\,\cdot) }{2}+1, \dots , \frac{\# \mathcal{J}^{'}(p,\, \cdot) }{2} \right\}\right\}. 
		$$

	We define the subsets $\mathcal{J}^{''} \subset \varepsilon \mathbb{Z}^2$ and 
	$\mathcal{R}(E) \subset \mathbb{R}^2$ 
	as
	$$  \mathcal{J}^{''}:= \bigcup_{p \in \mathbb{Z}} \mathcal{J}^{''}(p)\quad \textrm{and}\quad \mathcal{R}(E):=E_{\mathcal{J}^{''}} .$$	
	The set $\mathcal{R}(E)$ can be obtained by the following recurrence procedure starting from the initial datum $E^{'}$
	in full analogy with formula \eqref{defE^{'}_{q+1}}:
	\begin{equation}\label{defE^{''}_{p+1}}
		\left\{
		\begin{aligned}
			& E_{p+1}^{''}:=\left(E^{''}_p \setminus E^{'}_{\mathcal{J}(p,\cdot)}\right) \cup E_{\mathcal{J}^{''}(p)} & \text{if $ p \geq p_{min} $,} \\
			& E_{p_{min}}^{''}=E^{'} & \text{if $p=p_{min}$}
		\end{aligned}
		\right.
	\end{equation}
	where 
	\begin{equation}\label{ascissaminimalediE^'}
		p_{min}:= \inf \left\{  p\in \mathbb{Z} \, \colon\, \, E^{'}_{\mathcal{J}(p,\cdot)} \neq \emptyset \right\}.
	\end{equation}
	From the construction of the set $\mathcal{R}(E)$ we have that $ \vert E \vert = \vert E^{'}\vert=\vert \mathcal{R}(E) \vert$.
	We observe that the set $\mathcal{J}^{''}$ is such that the two functions 
	$$ i \in \mathbb{Z} \longrightarrow \# \mathcal{J}^{''}(i,\, \cdot) \quad i \in \mathbb{Z} \longrightarrow \# \mathcal{J}^{''}(\cdot,\, i)$$
	are increasing for $i\leq 0$ and non-increasing for $i \geq 0$. 
	Note that such a property follows by the same argument detailed above for $\# \mathcal{J}^{'}(i,\cdot)$.
	According to the definition \ref{formadelrombo} we have that the set $\mathcal{R}(E)$ has a rhombus-like shape and that
	$$ \mathrm{P}(\mathcal{R}(E))=2\max_{q \in \mathbb{Z}} \varepsilon \#\mathcal{J}(\cdot,\, q)+2\max_{p \in \mathbb{Z} } \varepsilon\#\mathcal{J}(p,\, \cdot)= 2 \mathrm{P}_1(E)+ 2 \mathrm{P}_2(E).$$

	Therefore we observe that 
	$$\mathrm{P}(E) \geq 2 \mathrm{P}_1(E)+ 2 \mathrm{P}_2(E)= \mathrm{P}(\mathcal{R}(E)).$$
%

\begin{figure}[h!]
    \centering
    \begin{tikzpicture}[scale=0.5]
        \draw[red,ultra thick] 
        (-16,-3) -- (-4,-3);
        \draw[red,ultra thick] 
        (-16,3) -- (-14,3);
        \draw[red,ultra thick] 
        (-14,3) -- (-14,4);
        \draw[red,ultra thick] 
        (-14,4) -- (-6,4);
        \draw[red,ultra thick] 
        (-6,4) -- (-6,3);
        \draw[red,ultra thick] 
        (-6,3) -- (-4,3);
        \draw[red,ultra thick] 
        (-4,3) -- (-4,-3);
        \draw[red,ultra thick] 
        (-16,3) -- (-16,-3);
        \draw [red,ultra thick]     (-5.9,4.5) node{$\partial QR $};
        \draw[blue,very thick] 
        (-4,-3)--(-4,-2)--(-3,-2)--(-3,2)--(-4,2)--(-4,3)--(-9,3)--(-9,5)--(-10,5)--(-10,4)--(-11,4)--(-11,3)--(-16,3)--(-16,-2)--(-15,-2)--(-15,-3)--(-11,-3)--(-11,-4)--(-9,-4)--(-9,-3)--cycle;
        \fill[blue, opacity=0.1]
        (-4,-3)--(-4,-2)--(-3,-2)--(-3,2)--(-4,2)--(-4,3)--(-9,3)--(-9,5)--(-10,5)--(-10,4)--(-11,4)--(-11,3)--(-16,3)--(-16,-2)--(-15,-2)--(-15,-3)--(-11,-3)--(-11,-4)--(-9,-4)--(-9,-3)--cycle;
        \draw [blue,ultra thick]     (-2.4,1.1) node{$E'' $};
    \end{tikzpicture}
    \caption{The quasi-rectangle $QR $ and the set $E''$}
    \label{fig6}
\end{figure}

\end{proof}

\begin{proposition}\label{menodissipazione}   Let $QR \in \mathcal{QR}_\varepsilon$ have pseudo axial symmetry and
	let us consider $E \in \mathcal{C}_{\varepsilon}$. Let $\mathcal{R}(E) \in \mathcal{C}_{\varepsilon}$ be the discrete Steiner-like rearrangement of the set $E$ with respect to $QR$. Then
	\begin{equation}\label{menodissipazionedis}
		D_\varepsilon(E,\,QR) \geq D_{\varepsilon}(\mathcal{R}(E),\, QR).
	\end{equation}
\end{proposition}
\begin{proof}
		By the definition of quasi rectangle \ref{quasirettangolo}, we write $QR=R \cup Q$, with $R$ a rectangle having sidelengths $\varepsilon n, \varepsilon m$ and $Q$ a rectangle with sidelengths $ \varepsilon r, \varepsilon$ where $ n, \, m, \, r \in \mathbb{N}$.
	By the very definition of quasi-rectangle with pseudo axial symmetry we have that
	$$  	 \mathrm{Bar}(R) \in \left\{ \left(0,\,0\right),\,\, \left(\frac{\varepsilon}{2},\,0\right),\,\left(0,\,\frac{\varepsilon}{2}\right),\,\, \left(\frac{\varepsilon}{2},\,\frac{\varepsilon}{2}\right) \right\}$$
	and that
	
$$
Q= \begin{cases}\left[- \frac{\varepsilon}{2}(r-1), \, \frac{\varepsilon}{2}(r-1)  \right] \times \left[L,\, L+\varepsilon\right] & \text{if $ r \in 2 \mathbb{N}+1$,}\\
\left[- \frac{\varepsilon}{2}(r+1), \, \frac{\varepsilon}{2}(r+1)  \right] \times \left[L,\, L+\varepsilon\right] &\text{if $ r \in 2 \mathbb{N}$}. 
\end{cases}
$$
	We define the followings lengths
	\begin{align*}
		&	L_1(+):=\mathcal{H}^{1}(\Pi_{1}(R)\cap ([0,\,+\infty)\times \{0\})), & L_1(-):=\mathcal{H}^{1}(\Pi_{1}(R)\cap ((-\infty,\,0]\times \{0\})), \\
		&	L_2(+):=\mathcal{H}^{1}(\Pi_{2}(R)\cap (\{0\}\times[0,\,+\infty))), & L_2(-):=\mathcal{H}^{1}(\Pi_{2}(R)\cap (\{0\}\times(-\infty,\,0])),\\
	        &l(+):= \mathcal{H}^{1}(\Pi_{1}(Q)\cap ([0,+\infty) \times \{0\})), & l(-):= \mathcal{H}^{1}(\Pi_{1}(Q)\cap ((-\infty,0] \times \{0\})),\\
		&	L_1:=L_1(+)+L_1(-)=n \varepsilon, \\
		& L_2:=L_2(+)+L_2(-)=m\varepsilon.
	\end{align*}
	By Proposition \ref{steinercostrizionelemma} there exists $\mathcal{R}(E)$ such that 
	\begin{equation*}
		\vert E \vert = \vert \mathcal{R}(E)\vert, \quad\mathrm{P}(E) \geq \mathrm{P}(\mathcal{R}(E)).
	\end{equation*}
	We recall that the set $\mathcal{R}(E)$ has been obtained by a two steps symmetrization procedure: in the first step we have symmetrized the rows while in the second one we have symmetrized the columns. The proof of the claim is accordingly subdivided in two steps in which we prove that the symmetrization of the rows (respectively the columns) decreases the dissipation. In what follows all the notation are the same as those used in Proposition \ref{steinercostrizionelemma}.\\
	\textit{Step 1)}\\
	We start by showing that the dissipation decreases under symmetrization by rows, i.e., 
	$$ D_{\varepsilon}(E^{'}_{q},\, QR) \geq D_{\varepsilon}(E^{'}_{q+1},\, QR) \quad \forall q \geq q_{min}$$
	where $E^{'}, \, E_{q}^{'}$ and $q_{min}$ are defined in \eqref{insiemeJ'def},  \eqref{defE^{'}_{q+1}} and \eqref{ordinataminimalediE}, respectively. We start considering the case 
	$q\neq \frac{1}{\varepsilon}\left(L_2(+)+\frac{\varepsilon}{2}\right)$. 
	We need to distinguish two cases, namely $ \varepsilon \# \mathcal{J}(\cdot,\,q) < L_1$ and $ \varepsilon \# \mathcal{J}(\cdot,\,q) \geq L_1$. In the first case if
	$$ \mathcal{J}(\cdot,\,q)  \cap QR^c  \neq \emptyset$$
	one can "move" the points of this set
	to the points of the set 
	$$ \mathcal{J}(\cdot,\, q)^{c} \cap (\left[-L_1(-),\, L_1(+) \right] \times \{q\})$$ 
	which is a subset of $QR$. This procedure decreases the dissipation since the latter is obtained integrating a distance on a subset of $E \Delta QR$.
	If instead 
	$$  \mathcal{J}(\cdot,\,q)  \cap QR^c = \emptyset,$$ we procede differently. We introduce the sets
	$$\mathcal{A}:= \left\{ \varepsilon(j,\,q) \colon \, \varepsilon(j,\,q) \in \mathcal{J}^{'}(q) \setminus \mathcal{J}(\cdot,\,q)\right\} \,\, \text{and } \, \,\mathcal{B}:= \left\{ \varepsilon( j,\, q) \colon \, \varepsilon(j,\,q) \in \mathcal{J}^{'}(q)^{c} \cap \mathcal{J}(\cdot,\,q) \right\}$$
	and define the function 
	$f: \mathcal{A} \rightarrow \mathcal{B}$ 
	as follows. If $i=i_{min}:= \min \{j \geq 0\, \colon\, \varepsilon(j,\,q) \in \mathcal{A}\}$
	$$f(\varepsilon( i,\,q)) \in \arg\min \left\{  \vert \varepsilon j \vert \, \colon \, \varepsilon(j,\,q) \in \mathcal{B}   \right\} $$
	with the choice that $f(\varepsilon(i,\,q))= \varepsilon(\hat{j},\,q)$
	in the case that the $\arg\min$ is not unique, namely
	$$ \arg\min \left\{  \vert \varepsilon j \vert \, \colon \, \varepsilon(j,\,q) \in \mathcal{B}    \right\} =\left\{ \varepsilon(-\hat{j},\, q),\,\varepsilon(\hat{j},\, q) \right\}.$$
	If $i> i_{min}$ we define
	$$f(\varepsilon( i,\,q)) \in \arg\min \left\{  \vert \varepsilon j \vert \, \colon \, \varepsilon(j,\,q) \in \mathcal{B} \setminus \left\{f(\varepsilon(i_{min},\,q)),\, \dots,\,f(\varepsilon(i-1,\,q))\right\}   \right\} $$
	with the choice $f(\varepsilon(i,\,q))= \varepsilon(\hat{j},\,q)$ in the case that the $\arg\min$ is not unique, namely
	$$  \arg\min \left\{  \vert \varepsilon j \vert \, \colon \, \varepsilon(j,\,q) \in \mathcal{B} \setminus \left\{f(\varepsilon(i_{min},\,q)),\, \dots,\,f(\varepsilon(i-1,\,q))\right\}   \right\}=\left\{ \varepsilon(-\hat{j},\, q),\,\varepsilon(\hat{j},\, q) \right\}. $$
	If $i <  i_{min}$ we define
	$$f(\varepsilon( i,\,q)) \in \arg\min \left\{  \vert \varepsilon j \vert \, \colon \, \varepsilon(j,\,q) \in \mathcal{B} \setminus \left\{ f(\varepsilon(\tilde{i},\,q))\colon\,  \tilde{i}\geq i+1 \right\}  \right\} $$
	with the choice $f(\varepsilon(i,\,q))= \varepsilon(\hat{j},\,q)$ in the case that the $\arg\min$ is not unique, namely
	$$  \arg\min \left\{  \vert \varepsilon j \vert \, \colon \, \varepsilon(j,\,q) \in \mathcal{B} \setminus \left\{ f(\varepsilon(\tilde{i},\,q))\colon\,  \tilde{i}\geq i+1 \right\}  \right\}=\left\{ \varepsilon(-\hat{j},\, q),\,\varepsilon(\hat{j},\, q) \right\}. $$
	
For all $ \varepsilon(i,\,q) \in \mathcal{A}$  and for all $ (x,\,y) \in Q_{\varepsilon}(\varepsilon(i,\,q))$ we uniquely define $ x':=x+ \Pi_{1}(f(\varepsilon(i,\,q))-\varepsilon(i,\,q))$, and observe that $ (x',\, y) \in Q_{\varepsilon}(f(\varepsilon(i,\,q)))$.
Without loss of generality we consider only the case 
$ \varepsilon(i,\,q) \in \mathcal{A}$ such that $q\geq 0,\, i\geq0$ and $\Pi_{1}(f(\varepsilon(i,\,q)))>0$, the argument in the other cases being fully analogous. For all $ (x,\,y) \in Q_{\varepsilon}(\varepsilon(i,\,q))$ with $x\geq0,\, y\geq0$
it holds that
\begin{equation}\label{distminorazione}
	\begin{split}
		&d_{\infty}^{\varepsilon}((x,\,y),\,\partial QR)= \min \left\{ d_{\infty}^{\varepsilon}((x,\,y),\, \partial QR(\pm)) , \, L_1(+)- \varepsilon i \right\} \geq \\
		&\min \left\{ d_{\infty}^{\varepsilon}((x,\,y),\, \partial QR(\pm)) , \, L_1(+)-\varepsilon i-\Pi_{1}(f(\varepsilon(i,\,q))-\varepsilon(i,\,q))\right\} = d_{\infty}^{\varepsilon}((x',\,y),\, \partial QR)
	\end{split}
\end{equation}
where $$\partial QR(\pm):= \partial QR \cap (\left( \mathbb{R}\times [L_2(+),+\infty)  \right) \cup \left( \mathbb{R}\times \{L_2(-)\} \right)).$$ Then, by the formula \eqref{distminorazione} we obtain
\begin{equation}\label{formulaDdecresciente}
	\int_{Q_{\varepsilon}(i,\,q)} d_{\infty}^{\varepsilon}((x,\,y),\,\partial QR) \, dx\, dy \geq \int_{Q_{\varepsilon}(f(\varepsilon(i,\,q)))} d_{\infty}^{\varepsilon}((x',\,y),\, \partial QR) \, dx'\, dy.
\end{equation}
Moreover, for all $q \in \mathbb{Z}$, we set $S(q):=[L_1(-),L_1(+)]\times\varepsilon\left[q-\frac{1}{2},\, q+\frac{1}{2}\right]$ and, referring to \eqref{defE^{'}_{q+1}}, we observe that
\begin{equation}\label{difsimmetricaEq+1eR}
	\begin{split}
		E^{'}_{q} \Delta QR \cap S(q)& = \bigcup_{\varepsilon(i,\,q)\in S(q)\cap \mathcal{A}^{c} \setminus \mathcal{J}(\cdot,\,q)} Q_{\varepsilon}(\varepsilon(i,\,q)) \cup \bigcup_{\varepsilon(i,\,q)\in \mathcal{A}} Q_{\varepsilon}(\varepsilon(i,\,q)),\\
		E^{'}_{q+1} \Delta QR \cap S(q)& = \bigcup_{\varepsilon(i,\,q)\in S(q)\cap \mathcal{A}^{c} \setminus \mathcal{J}(\cdot,\,q)} Q_{\varepsilon}(\varepsilon(i,\,q)) \cup \bigcup_{\varepsilon(i,\,q)\in \mathcal{A}} Q_{\varepsilon}(f(\varepsilon(i,\,q))).
	\end{split}
\end{equation}
As a result, thanks to \eqref{formulaDdecresciente} and \eqref{difsimmetricaEq+1eR}, we have that
\begin{equation}\label{formuladeiraginamenti}
	\begin{split}
		D_{\varepsilon}(E^{'}_{q},\,QR)& =\int_{E_{q}^{'} \Delta QR} d_{\infty}^{\varepsilon} (z,\, \partial QR)\, dz \\
		& =  \int_{E_{q}^{'} \Delta QR \cap S(q)^{c}} d_{\infty}^{\varepsilon} (z,\,  \partial QR)\, dz   +\int_{E_{q}^{'} \Delta QR \cap S(q)} d_{\infty}^{\varepsilon} (z,\,  \partial QR)\, dz  \\
		& =\int_{E_{q}^{'} \Delta QR \cap S(q)^c} d_{\infty}^{\varepsilon} (z,\, \partial QR)\, dz
		+ \int_{\bigcup_{\varepsilon(i,\,q)\in S(q)\cap \mathcal{A}^{c} \setminus \mathcal{J}(\cdot,\,q)} Q_{\varepsilon}(\varepsilon(i,\,q))} d_{\infty}^{\varepsilon} (z,\,  \partial QR)\, dz \\
		&\,\,\,\,+ \int_{\bigcup_{\varepsilon(i,\,q)\in \mathcal{A}} Q_{\varepsilon}(\varepsilon(i,\,q))} d_{\infty}^{\varepsilon} (z,\, \partial QR)\, dz \\
		& \geq \int_{E_{q}^{'} \Delta QR \cap S(q)^c} d_{\infty}^{\varepsilon} (z,\, \partial QR)\, dz 
		+ \int_{\bigcup_{\varepsilon(i,\,q)\in S(q)\cap \mathcal{A}^{c} \setminus \mathcal{J}(\cdot,\,q)} Q_{\varepsilon}(\varepsilon(i,\,q))} d_{\infty}^{\varepsilon} (z,\, \partial QR)\, dz \\
		& \,\,\,\,+ \int_{\bigcup_{\varepsilon(i,\,q)\in \mathcal{A}} Q_{\varepsilon}(f(\varepsilon(i,\,q)))} d_{\infty}^{\varepsilon} (z,\, \partial QR)\, dz \\
		&=\int_{E_{q+1}^{'} \Delta QR} d_{\infty}^{\varepsilon} (z,\, \partial QR)\, dz = D_{\varepsilon}(E^{'}_{q+1},\,QR).
	\end{split}
\end{equation} 
Hence we have proved that
$$ \text{if } \quad\varepsilon \# \mathcal{J}(\cdot,\, q)<L_1 \quad \text{ then } \quad D_{\varepsilon}(E^{'}_{q},\, QR) \geq D_{\varepsilon}(E^{'}_{q+1},\, QR).$$
We now show the proof in the second case, namely for $\varepsilon \# \mathcal{J}(\cdot,\, q)\geq L_1$. We start by observing that the dissipation $D_{\varepsilon}(\cdot,QR)$ decreases if we ''move'' points of the set $QR^{c} \cap \mathcal{J}(\cdot,\,q)$ to the set
$$ QR \cap \mathcal{J}(\cdot,\,q)^{c} \cap (\left[ -L_1(-),\,L_1(+) \right] \times \{q\}). $$
Such a procedure stops if 
$$ QR \cap \mathcal{J}(\cdot,\,q)^{c} \cap (\left[ -L_1(-),\,L_1(+) \right] \times \{q\} ) = \emptyset .$$
In the latter case we need to proceed differently. We introduce the sets 
$$ \mathcal{\tilde C}:= \mathcal{J}(\cdot,\,q) \cap \mathcal{J}^{'}(q) \cap QR^c, \quad \mathcal{\tilde A}:= \mathcal{J}(\cdot,\,q) \cap QR^{c} \setminus \mathcal{\tilde C}, \quad \mathcal{\tilde B}= \mathcal{J}^{'}(q) \cap QR^{c}\setminus \mathcal{\tilde C}$$
and define the function 
$\tilde f: \mathcal{\tilde A} \rightarrow \mathcal{\tilde B}$ 
as follows. If $i=i_{min}:= \min \{ j \geq 0\,\colon \, \varepsilon(j,\,q) \in \mathcal{\tilde A}\}$
$$f(\varepsilon( i,\,q)) \in \arg\min \left\{  \vert \varepsilon j \vert \, \colon \, \varepsilon(j,\,q) \in \mathcal{\tilde B}   \right\} $$
with the choice that $\tilde f(\varepsilon(i,\,q))= \varepsilon(\hat{j},\,q)$
in the case
$$ \arg\min \left\{  \vert \varepsilon j \vert \, \colon \, \varepsilon(j,\,q) \in \mathcal{B}    \right\} =\left\{ \varepsilon(-\hat{j},\, q),\,\varepsilon(\hat{j},\, q) \right\}.$$
If $i> i_{min}$
$$\tilde f(\varepsilon( i,\,q)) \in \arg\min \left\{  \vert \varepsilon j \vert \, \colon \, \varepsilon(j,\,q) \in \mathcal{\tilde B} \setminus \left\{\tilde f(\varepsilon(i_{min},\,q)),\, \dots,\,\tilde f(\varepsilon(i-1,\,q))\right\}   \right\} $$
with the choice $\tilde f(\varepsilon(i,\,q))= \varepsilon(\hat{j},\,q)$ in the case
$$  \arg\min \left\{  \vert \varepsilon j \vert \, \colon \, \varepsilon(j,\,q) \in \mathcal{\tilde B} \setminus \left\{\tilde f(\varepsilon(i_{min},\,q)),\, \dots,\,\tilde f(\varepsilon(i-1,\,q))\right\}   \right\}=\left\{ \varepsilon(-\hat{j},\, q),\,\varepsilon(\hat{j},\, q) \right\}. $$
If $i <  i_{min}$
$$\tilde f(\varepsilon( i,\,q)) \in \arg\min \left\{  \vert \varepsilon j \vert \, \colon \, \varepsilon(j,\,q) \in \mathcal{B} \setminus \left\{ \tilde f(\varepsilon(\tilde{i},\,q))\colon\,  \tilde{i}\geq i+1 \right\}  \right\} $$
with the choice $\tilde f(\varepsilon(i,\,q))= \varepsilon(\hat{j},\,q)$ in the case
$$  \arg\min \left\{  \vert \varepsilon j \vert \, \colon \, \varepsilon(j,\,q) \in \mathcal{\tilde B} \setminus \left\{ \tilde f(\varepsilon(\tilde{i},\,q))\colon\,  \tilde{i}\geq i+1 \right\}  \right\}=\left\{ \varepsilon(-\hat{j},\, q),\,\varepsilon(\hat{j},\, q) \right\}. $$

We observe that, as long as
$\tilde f(\varepsilon(i,\,q))\in \mathcal{\tilde B}$ then 
\begin{equation}\label{disminoreminore}
	d_{\infty}^{\varepsilon}(\varepsilon(i,\,q),\,\partial QR) \geq d^{\varepsilon}_{\infty}(f(\varepsilon(i,\,q)),\, \partial QR).
\end{equation}
For all $ \varepsilon(i,\,q) \in \mathcal{\tilde A}$  and for all $ (x,\,y) \in Q_{\varepsilon}(\varepsilon(i,\,q))$ we uniquely define $ x':=x+ \Pi_{1}(\tilde f(\varepsilon(i,\,q))-\varepsilon(i,\,q))$, and observe that $ (x',\, y) \in Q_{\varepsilon}(\tilde f(\varepsilon(i,\,q)))$
and by \eqref{disminoreminore} we have that for all $(x,\,y)\in Q_{\varepsilon}(\varepsilon(i,\,q))$
\begin{equation}\label{disminorminor123}
	d_{\infty}^{\varepsilon}((x,\,y),\,\partial QR) \geq d^{\varepsilon}_{\infty}((x+ \Pi_{1}(\tilde f(\varepsilon(i,\,q))-\varepsilon(i,\,q)),\,y),\, \partial QR)=d^{\varepsilon}_{\infty}((x',\,y),\, \partial QR).
\end{equation}
Moreover we have
\begin{equation}\label{formulainsiemiasdasdasd}
	\begin{split}
		E^{'}_{q} \Delta QR & = \left(\left(E_{q}^{'} \Delta QR \right) \cap S(q)^{c}\right) \cup \bigcup_{\varepsilon(i,\,q) \in \mathcal{\tilde A}} Q(\varepsilon(i,\,q)) \cup \bigcup_{\varepsilon(i,\,q) \in \mathcal{\tilde C}} Q(\varepsilon(i,\,q)),  \\
		E^{'}_{q+1} \Delta QR & = \left(\left(E_{q}^{'} \Delta QR \right) \cap S(q)^{c}\right) \cup \bigcup_{\varepsilon(i,\,q) \in \mathcal{\tilde A}} Q(\tilde f(\varepsilon(i,\,q))) \cup \bigcup_{\varepsilon(i,\,q) \in \mathcal{\tilde C}} Q(\varepsilon(i,\,q)).
	\end{split}
\end{equation}
Hence, thanks to \eqref{disminorminor123} and \eqref{formulainsiemiasdasdasd}, we obtain 
\begin{equation*}
	\begin{split}
		D_{\varepsilon}(E^{'}_{q},\,QR)&=\int_{E_{q}^{'} \Delta QR} d_{\infty}^{\varepsilon} (z,\, \partial QR)\, dz \\
		& =\int_{E_{q}^{'} \Delta QR \cap S(q)^c} d_{\infty}^{\varepsilon} (z,\, \partial QR)\, dz 
		+ \int_{\bigcup_{\varepsilon(i,\,q)\in \mathcal{\tilde A}} Q_{\varepsilon}(\varepsilon(i,\,q))} d_{\infty}^{\varepsilon} (z,\,  \partial QR)\, dz \\
		&\quad + \int_{\bigcup_{\varepsilon(i,\,q)\in \mathcal{\tilde C}} Q_{\varepsilon}(\varepsilon(i,\,q))} d_{\infty}^{\varepsilon} (z,\, \partial QR)\, dz \\
		& \geq \int_{E_{q}^{'} \Delta QR \cap S(q)^c} d_{\infty}^{\varepsilon} (z,\, \partial QR)\, dz 
		+ \int_{\bigcup_{\varepsilon(i,\,q)\in \mathcal{\tilde A}} Q_{\varepsilon}(\tilde f(\varepsilon(i,\,q)))} d_{\infty}^{\varepsilon} (z,\, \partial QR)\, dz \\
		& \,\,\,\,+ \int_{\bigcup_{\varepsilon(i,\,q)\in \mathcal{\tilde C}} Q_{\varepsilon}(\varepsilon(i,\,q))} d_{\infty}^{\varepsilon} (z,\, \partial QR)\, dz \\
		&=\int_{E_{q+1}^{'} \Delta QR} d_{\infty}^{\varepsilon} (z,\, \partial QR)\, dz = D_{\varepsilon}(E^{'}_{q+1},\,QR)
	\end{split}
\end{equation*} 
 which concludes the second case.\\ 
\noindent  To conclude Step 1 we are left to consider the case $ q= \frac{1}{\varepsilon}\left(L_2^{\varepsilon}(+)+\frac{\varepsilon}{2}\right)$. For such $q$ the proof of $ D_{\varepsilon}(E^{'}_{q},\,QR)  \geq  D_{\varepsilon}(E^{'}_{q+1},\,QR)$ follows the same lines outlined above replacing $L_1(-)$ and $L_1(+)$ by $l(-)$ and $l(+)$, respectively. \\
\textit{Step 2)}\\
In this step we prove that $$D_{\varepsilon}(E^{'},\,QR) \geq D_{\varepsilon}(\mathcal{R}(E),\, QR).$$
To this end we need to show that
$$ D_{\varepsilon}(E^{''}_{p},\, QR) \geq D_{\varepsilon}(E^{''}_{p+1},\, QR) \quad \forall p \geq p_{min}$$
where $E_{p}^{''},\, p_{min}$ are defined in \eqref{defE^{''}_{p+1}} and \eqref{ascissaminimalediE^'}.
The proof of this inequality follows by the same arguments of Step 1). This time the geometric construction involves a fully analogous optimization procedure on columns instead of the one used above on rows. 
 
\end{proof}


\subsection{The incremental problem}
We start by introducing the incremental minimum problem which defines the discrete-in-time area-preserving flow for $\varepsilon \mathbb{Z}^{2}$ crystals.  
We begin by introducing, for all $\varepsilon \in (0,\,1)$, the space of admissible configurations
\begin{equation}\label{defDvareprho}
	\mathcal{AD}_{\varepsilon}:= \big\{ E\in\mathcal{C}_{\varepsilon},\, 
	\vert E \vert =1  \big\}.
\end{equation}	
Let $QR$ be a quasi rectangle of unitary area, that is $QR\in\mathcal{AD}_{\varepsilon}$. Fix a real number $\tau>0$ and consider the problem
\begin{equation*}\label{probdiminimo}
	\min \left\{ \mathrm{P}(E)+\frac{1}{\tau}D_\varepsilon(E,\, QR) \, \colon \, E \in \mathcal{AD}_{\varepsilon} \right\}, \tag{$P_1$} 
\end{equation*}
where $D_{\varepsilon}(E,\,QR)$ is defined in \eqref{dissipazioneD_epsilon}. 
In the next theorem, we prove that if $QR$ is pseudo-axially symmetric, then the minimum problem \eqref{probdiminimo} is equivalent to 
\begin{equation*}\label{probminimoQuasRect}
	\min \left\{ \mathrm{P}(E)+\frac{1}{\tau}D_\varepsilon(E,\, QR) \, \colon \, E \in \mathcal{QR}_{\varepsilon}, \, E \text{ pseudo-axially symmetric, } \vert E \vert= 1 \right\} \tag{$P_2$}.
\end{equation*}

\begin{theorem}\label{thmprinc}
	Let $QR \in \mathcal{QR}_{\varepsilon}$ be a quasi-rectangle with pseudo-axial symmetry. Then the problem \eqref{probdiminimo} is equivalent to the problem \eqref{probminimoQuasRect}.
\end{theorem}
\begin{proof}
	According to the definition of quasi-rectangle in \ref{quasirettangolo} we let $QR=R \cup Q \in \mathcal{QR}_\varepsilon$ with $R$ a rectangle with horizontal sidelength $ \varepsilon n$ and vertical sidelength $ \varepsilon m$ and $Q$ a rectangle with horizontal sidelength $ \varepsilon r $ and vertical sidelength $\varepsilon$, where $n,\,m,\,r \in \mathbb{N}$ and moreover $ m \leq n$ and $r \leq \max\{n,m\}$.
	We claim that for all $E \in \mathcal{AD}_{\varepsilon}$ there exists $\widetilde{QR} \in \mathcal{QR}_{\varepsilon}$ with pseudo-axial symmetry such that $\vert \widetilde{QR} \vert=1$ and
	\begin{equation}\label{rettangolominore}
		\mathrm{P}(E)\geq \mathrm{P}(\widetilde{QR}), \quad D_{\varepsilon}(E,\, QR) \geq D_{\varepsilon}(\widetilde{QR},\,QR).
	\end{equation}
	Thanks to Propositions \ref{steinercostrizionelemma} and \ref{menodissipazione} we have
	\begin{equation*} \mathrm{P}(E)\geq \mathrm{P}(\mathcal{R}(E)), \quad D_{\varepsilon}(E,\, QR) \geq D_{\varepsilon}(\mathcal{R}(E),\,QR).
	\end{equation*}
	Hence the claim is proved provided we show that 
	\begin{equation*} \mathrm{P}(\mathcal{R}(E))\geq \mathrm{P}(\widetilde{QR}), \quad D_{\varepsilon}(\mathcal{R}(E),\,QR)\geq D_{\varepsilon}(\widetilde{QR},\,QR).
	\end{equation*}
	In what follows we need to distinguish three cases.\\
	\textit{Case $1$}\\
	 We assume that 
	\begin{equation}\label{caso1P12geqL12}
		\mathrm{P}_1(\mathcal{R}(E))=\mathrm{P}_1(E) \geq \mathrm{P}_1(QR)= \, \text{ and } \, \mathrm{P}_2(\mathcal{R}(E))=\mathrm{P}_2(E) \geq \mathrm{P}_2(QR),
	\end{equation}
where $\mathrm{P}_1$ and $\mathrm{P}_2$ are defined in \eqref{P12}.
	In this case we set $\widetilde{QR}:= QR$. By formula \eqref{caso1P12geqL12} and \eqref{P12ineq} we have that
	$$ \mathrm{P}(E)\geq 2 \mathrm{P}_1(E)+2 \mathrm{P}_2(E)  \geq 2\mathrm{P}_1(QR)+ 2 \mathrm{P}_2(QR)= \mathrm{P}(QR)$$
	and
	$$ D_{\varepsilon}(E,\, QR) \geq 0 = D_{\varepsilon}(QR,\, QR).$$
	Hence we obtain \eqref{rettangolominore}. \\
	\textit{Case $2$}\\
	In this case we assume
	\begin{equation}\label{170120231209}
	 \mathrm{P}_1(\mathcal{R}(E))=\mathrm{P}_1(E) \geq \mathrm{P}_1(QR) \quad \text{and} \quad \mathrm{P}_2(\mathcal{R}(E))=\mathrm{P}_2(E)<\mathrm{P}_2(QR).
 \end{equation}
	We claim that $\widetilde{QR}= QR$. In order to prove the claim we first assume that 
	\begin{equation}\label{Case2.1}
	\mathrm{P}_2 (\mathcal{R}(E))= \mathrm{P}_2(QR)-\varepsilon.
	\end{equation}
	We start observing that the following implication, whose proof can be obtained by a contradiction argument, holds true:
	$$ \text{if }\mathrm{P}_2 (\mathcal{R}(E))= \mathrm{P}_2(QR)-\varepsilon \quad\text{ then }\quad \mathrm{P}_1(\mathcal{R}(E)) \geq \mathrm{P}_1(QR)+ \varepsilon.$$
	Therefore by \eqref{Case2.1} we have 
	$$ \mathrm{P}(\mathcal{R}(E))= 2 \mathrm{P}_1(\mathcal{R}(E))+ 2\mathrm{P}_2(\mathcal{R}(E))\geq 2 \mathrm{P}_1(QR)+ 2 \varepsilon + 2 \mathrm{P}_2(QR)- 2 \varepsilon = \mathrm{P}(QR)$$
	and $D_\varepsilon(\mathcal{R}(E),\,QR) \geq D_{\varepsilon}(QR,\,QR)=0$. Hence the claim is proved under the assumption \eqref{Case2.1}. We are left to consider the case when 
	\begin{equation}\label{Case2.2}
	\mathrm{P}_2(\mathcal{R}(E)) \leq \mathrm{P}_2(QR)-2 \varepsilon.
	\end{equation}
	Let $N \in \mathbb{N}$ be such that $ 1= \vert E \vert= N \varepsilon^2$ and let $k_1,\,k_2 \in \mathbb{N}$ be such that $\mathrm{P}_1(E)=\varepsilon k_1$ and $\mathrm{P}_2(E)= \varepsilon k_2$.
	By the Euclidean division we have that there exist $q,\, l \in \mathbb{N}$ such that $ N= k_2 q +l$ with $0 \leq l < k_2$. We define the set $QR':=R^{'} \cup Q^{'}$ to be a quasi rectangle with axial symmetry constructed as follows: $R^{'}$ is a rectangle with horizontal sidelength $\varepsilon q$ and vertical sidelength $ \varepsilon k_2$ while $Q^{'}$ is a rectangle with horizontal sidelength $ \varepsilon l$ and with vertical sidelength $\varepsilon$. The construction is complete once we fix the position of the baricenter according to the following alternatives:
	\begin{itemize}
		\item[] if $n \equiv_{2} q $ and $ m \equiv_{2} k_2 $ then $\mathrm{Bar}(R^{'})= \mathrm{Bar}(R)$,
		\item[]  if $n \equiv_2 q$ and $m \not\equiv_{2} k_2$ then $\mathrm{Bar}(R^{'})=(\mathrm{Bar}(R)_x,\, \mathrm{Bar}(R)_y- \frac{\varepsilon}{2})$ if $k_2 \in 2 \mathbb{N} +1$ and $\mathrm{Bar}(R^{'})=(\mathrm{Bar}(R)_x,\, \mathrm{Bar}(R)_y+ \frac{\varepsilon}{2})$ if $k_2 \in 2 \mathbb{N} $,
		\item[] if $n \not\equiv_2 q$ and $m \equiv_{2} k_2$ then $\mathrm{Bar}(R^{'})=(\mathrm{Bar}(R)_x - \frac{\varepsilon}{2},\, \mathrm{Bar}(R)_y)$ if $k_2 \in 2 \mathbb{N} +1$ and $\mathrm{Bar}(R^{'})=(\mathrm{Bar}(R)_x + \frac{\varepsilon}{2},\, \mathrm{Bar}(R)_y)$ if $k_2 \in 2 \mathbb{N} $,
		\item[] if $n \not\equiv_{2} q $ and $ m \not\equiv_{2} k_2 $ then $\mathrm{Bar}(R^{'})=0$ if $k_2 \in 2 \mathbb{N}+1$ and $\mathrm{Bar}(R^{'})=(\frac{\varepsilon}{2}, \, \frac{\varepsilon}{2})$ if $k_2 \in 2 \mathbb{N}$.
	\end{itemize}
    We observe that the set $ QR'$ is a subset of the rectangle with horizontal sidelength $ \varepsilon q$ and vertical sidelength $\mathrm{P}_2(QR)-\varepsilon$, hence the set $Q^{'}$ is a subset of $QR$. 
    We claim that $$\mathrm{P}(\mathcal{R}(E)) \geq \mathrm{P}(QR'),$$ where $\mathrm{P}(\mathcal{R}(E))=2 \varepsilon k_1+ 2 \varepsilon k_2$ and $\mathrm{P}(QR')=2 \varepsilon q+ 2 \varepsilon k_2+ 2 \varepsilon\chi_{\mathbb{N} \setminus \{0\}}(l)$.   We first consider $l=0$. In this case observing that $\mathcal{R}(E)$ is a subset of a rectangle with horizontal sidelength $ \mathrm{P}_1(E)=\varepsilon k_1$ and vertical sidelength $\mathrm{P}_2(E)=\varepsilon k_2$ and the same barycenter of $ \mathcal{R}(E)$, we obtain
    $$ k_1 k_2 \varepsilon^2 \geq \vert \mathcal{R}(E)\vert=1 = \vert E \vert= N \varepsilon^2 =k_2 q \varepsilon^2  \implies k_1 \geq q.$$
      Hence $ \mathrm{P}(\mathcal{R}(E)) \geq \mathrm{P}(QR')$.
    If instead $l \neq 0$, arguing as above, we observe that 
    $$ k_2 k_1 \varepsilon^2 \geq\vert \mathcal{R}(E)\vert=1 = \vert E \vert= N \varepsilon^2= k_2 q \varepsilon^2 + l \varepsilon^2 \quad \text{and} \quad 0<l < k_2  \implies k_1>q .$$ Therefore $ \mathrm{P}(\mathcal{R}(E)) \geq \mathrm{P}(QR')$. 
Now if $Q' \neq \emptyset$ we iterate the argument above replacing the set $\mathcal{R}(E)$ by $QR'$. In this way we obtain again a set with less perimeter. If instead $Q' = \emptyset$ the procedure above can still be applied, the only difference being in the Euclidean division. In fact in this case we divide $N$ by $\frac{\mathrm{P}_2(QR')}{\varepsilon}+1$. This procedure can be repeated until the vertical component of the perimeter is equal to $\mathrm{P}_2(QR)-\varepsilon$. When this happens we fall in the case \eqref{Case2.1} and conclude as before.\\ 
\textit{Case $3$}\\
In this case we assume 

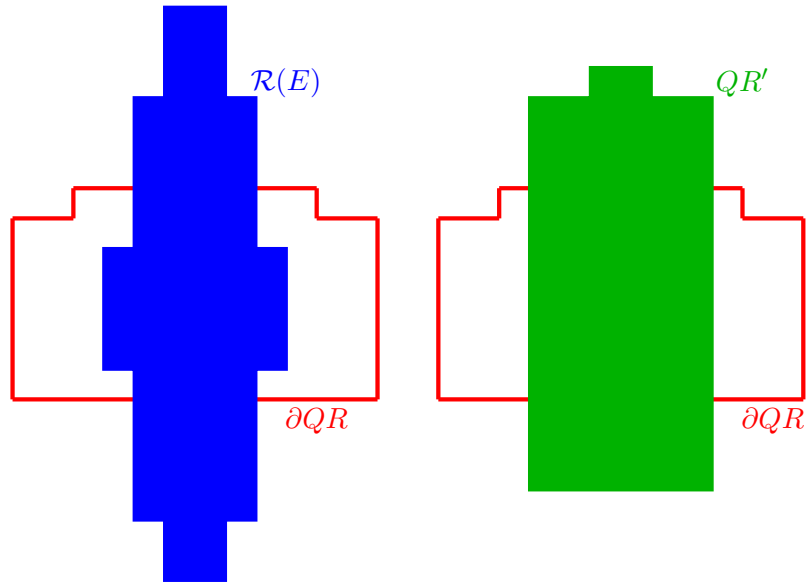
\begin{figure}[h!]
    \centering
    \begin{tikzpicture}[scale=0.4]
        \draw[red,ultra thick] 
        (-16,-3) -- (-4,-3);
        \draw[red,ultra thick] 
        (-16,3) -- (-14,3);
        \draw[red,ultra thick] 
        (-14,3) -- (-14,4);
        \draw[red,ultra thick] 
        (-14,4) -- (-6,4);
        \draw[red,ultra thick] 
        (-6,4) -- (-6,3);
        \draw[red,ultra thick] 
        (-6,3) -- (-4,3);
        \draw[red,ultra thick] 
        (-4,3) -- (-4,-3);
        \draw[red,ultra thick] 
        (-16,3) -- (-16,-3);
        \draw [red,ultra thick]     (-6,-3.7) node{$\partial QR $};
        \draw [red,ultra thick]     (9,-3.7) node{$\partial QR $};
        \draw[blue,very thick] 
        (-9,-9)--(-9,-7)--(-8,-7)--(-8,-2)--(-7,-2)--(-7,2)--(-8,2)--(-8,7)--(-9,7)--(-9,10)--(-11,10)--(-11,7)--(-12,7)--(-12,2)--(-13,2)--(-13,-2)--(-12,-2)--(-12,-7)--(-11,-7)--(-11,-9)-- cycle;
        \fill[blue, opacity=0.1]
        (-9,-9)--(-9,-7)--(-8,-7)--(-8,-2)--(-7,-2)--(-7,2)--(-8,2)--(-8,7)--(-9,7)--(-9,10)--(-11,10)--(-11,7)--(-12,7)--(-12,2)--(-13,2)--(-13,-2)--(-12,-2)--(-12,-7)--(-11,-7)--(-11,-9)-- cycle;
        \draw [blue,ultra thick]     (-7,7.5) node{$\mathcal{R}(E) $};
        \draw[red,ultra thick] 
        (-2,-3) -- (10,-3);
        \draw[red,ultra thick] 
        (-2,3) -- (0,3);
        \draw[red,ultra thick] 
        (0,3) -- (0,4);
        \draw[red,ultra thick] 
        (0,4) -- (8,4);
        \draw[red,ultra thick] 
        (8,4) -- (8,3);
        \draw[red,ultra thick] 
        (8,3) -- (10,3);
        \draw[red,ultra thick] 
        (-2,3) -- (-2,-3);
        \draw[red,ultra thick] 
        (10,3) -- (10,-3);
        \draw[green!70!black,very thick] 
        (7,-6)--(7,7)--(5,7)--(5,8)--(3,8)--(3,7)--(1,7)--(1,-6)-- cycle;
        \fill[green!70!black, opacity=0.2]
        (7,-6)--(7,7)--(5,7)--(5,8)--(3,8)--(3,7)--(1,7)--(1,-6)-- cycle;
        \draw [green!70!black,ultra thick]     (8,7.5) node{$QR' $};
    \end{tikzpicture}
    \caption{The quasi-rectangle $QR $ and the sets $\mathcal{R}(E), \,QR' $}
    \label{fig18062023sera}
\end{figure}


	$$ \mathrm{P}_{1}(\mathcal{R}(E)) \leq \mathrm{P}_1(QR) \,\, \text{ and }\, \,  \mathrm{P}_{2}(\mathcal{R}(E)) > \mathrm{P}_2(QR). $$
Let $N\in \mathbb{N}$ be such that $1= \varepsilon^2 N = \vert E \vert$ and let $k_1 \in \mathbb{N}$ be such that $\mathrm{P}_1(E)=\varepsilon k_1$. By the Euclidean division we have that there exist $q,\,l \in \mathbb{N}$ such that $ N= q k_1+l$ with $l < k_1$. 
We define the set $QR':=R^{'} \cup Q^{'}$ to be a quasi rectangle with axial symmetry constructed as follows: $R^{'}$ is a rectangle with horizontal sidelength $\varepsilon k_1$ and vertical sidelength $\e q$ while $Q^{'}$ is a rectangle with horizontal sidelength $ \varepsilon l$ and with vertical sidelength $\varepsilon$.
Moreover the coordinate of the barycenter of $R'$ depends on $ (n,m)$ as follows:
\begin{itemize}
	\item[] if $n \equiv_{2} k_1 $ and $ m \equiv_{2} q $ then $\mathrm{Bar}(R^{'})= \mathrm{Bar}(R)$,
	\item[]  if $n \equiv_2 k_1$ and $m \not\equiv_{2} q$ then $\mathrm{Bar}(R^{'})=(\mathrm{Bar}(R)_x,\, \mathrm{Bar}(R)_y+ \frac{\varepsilon}{2})$ if $q \in 2 \mathbb{N} +1$ and $\mathrm{Bar}(R^{'})=(\mathrm{Bar}(R)_x,\, \mathrm{Bar}(R)_y+ \frac{\varepsilon}{2})$ if $q \in 2 \mathbb{N} $, 
	\item[] if $n \not\equiv_2 k_1$ and $m \equiv_{2} q$ then $\mathrm{Bar}(R^{'})=(\mathrm{Bar}(R)_x - \frac{\varepsilon}{2},\, \mathrm{Bar}(R)_y)$ if $k_1 \in 2 \mathbb{N} +1$ and $\mathrm{Bar}(R^{'})=(\mathrm{Bar}(R)_x + \frac{\varepsilon}{2},\, \mathrm{Bar}(R)_y)$ if $k_1 \in 2 \mathbb{N} $,
	\item[] if $n \not\equiv_{2} k_1 $ and $ m \not\equiv_{2} q $ then $\mathrm{Bar}(R^{'})=\mathrm{Bar}(R)+(\frac{\varepsilon}{2}, \, \frac{\varepsilon}{2})$ if $k_1 \in 2\N $ and $ q \in 2\N+1$ while $\mathrm{Bar}(R^{'})=\mathrm{Bar}(R)+(-\frac{\varepsilon}{2}, \, \frac{\varepsilon}{2})$ if $k_1 \in 2\N+1 $ and $ q \in 2\N$.
\end{itemize}


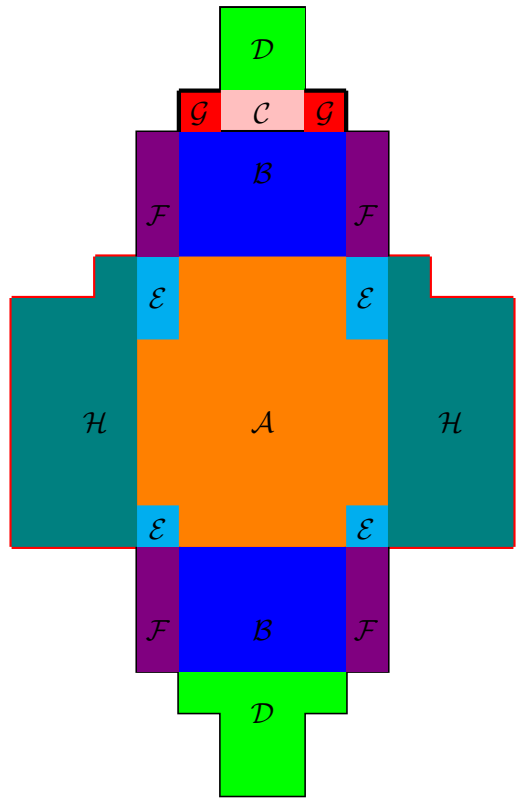
\begin{figure}[h!]
	\begin{center}
	\begin{tikzpicture}[scale=0.55]
		\draw[red,ultra thick] 
		(-16,-3) -- (-4,-3);
		\draw[red,ultra thick]
		(-16,3) -- (-14,3);
		\draw[red,ultra thick]
		(-14,3) -- (-14,4);
		\draw[red,ultra thick]
		(-14,4) -- (-6,4);
		\draw[red,ultra thick]
		(-6,4) -- (-6,3);
		\draw[red,ultra thick] 
		(-6,3) -- (-4,3);
		\draw[red,ultra thick]
		(-4,3) -- (-4,-3);
		\draw[red,ultra thick]
		(-16,3) -- (-16,-3);
		\draw[black,very thick] 
		(-9,-9)--(-9,-7)--(-8,-7)--(-8,-2)--(-7,-2)--(-7,2)--(-8,2)--(-8,7)--(-9,7)--(-9,10)--(-11,10)--(-11,7)--(-12,7)--(-12,2)--(-13,2)--(-13,-2)--(-12,-2)--(-12,-7)--(-11,-7)--(-11,-9)-- cycle;
		\draw[black,very thick] 
		(-7,-6)--(-7,7)--(-9,7)--(-9,8)--(-11,8)--(-11,7)--(-13,7)--(-13,-6)-- cycle;
		\fill[green, opacity=0.2] (-8,-6)--(-8,-7)--(-9,-7)--(-9,-9)--(-11,-9)--(-11,-7)--(-12,-7)--(-12,-6)--cycle;
		\fill[green, opacity=0.2] (-11,8)--(-9,8)--(-9,10)--(-11,10)--cycle;
		\draw [black ,ultra thick]     (-10,9) node{$\mathcal{D} $};
		\draw [black ,ultra thick]     (-10,-7) node{$\mathcal{D} $};
		\fill[pink, opacity=0.2] (-11,8)--(-9,8)--(-9,7)--(-11,7)--cycle;
		\draw [black ,ultra thick]     (-10,7.45) node{$\mathcal{C} $};
		\draw [black ,ultra thick] (-11,7) -- (-9,7);
		\fill[red, opacity=0.2] (-11,7)--(-11,8)--(-12,8)--(-12,7)--cycle;
		\fill[red, opacity=0.2] (-9,7)--(-9,8)--(-8,8)--(-8,7)--cycle;
		\draw [black ,ultra thick] (-12,7) -- (-12,8);
		\draw [black ,ultra thick] (-8,8) -- (-8,7);
		\draw [black ,ultra thick] (-9,8) -- (-8,8);
		\draw [black ,ultra thick] (-12,8) -- (-11,8);
		\draw [black ,ultra thick]     (-8.5,7.45) node{$\mathcal{G} $};
		\draw [black ,ultra thick]     (-11.5,7.45) node{$\mathcal{G} $};
		\fill[blue, opacity=0.2] (-12,7)--(-12,4)--(-8,4)--(-8,7)--cycle;
		\fill[blue, opacity=0.2] (-12,-3)--(-12,-6)--(-8,-6)--(-8,-3)--cycle;
		\draw [black ,ultra thick]     (-10,6) node{$\mathcal{B} $};
		\draw [black ,ultra thick]     (-10,-5) node{$\mathcal{B} $};
		\fill[violet, opacity=0.2] (-12,-3)--(-13,-3)--(-13,-6)--(-12,-6)--cycle;
		\fill[violet, opacity=0.2] (-12,4)--(-13,4)--(-13,7)--(-12,7)--cycle;
		\fill[violet, opacity=0.2] (-8,-3)--(-7,-3)--(-7,-6)--(-8,-6)--cycle;
		\fill[violet, opacity=0.2] (-8,4)--(-7,4)--(-7,7)--(-8,7)--cycle;
		\draw [black ,ultra thick]     (-7.531,-5) node{$\mathcal{F} $};
		\draw [black ,ultra thick]     (-7.531,5) node{$\mathcal{F} $};
		\draw [black ,ultra thick]     (-12.5,-5) node{$\mathcal{F} $};
		\draw [black ,ultra thick]     (-12.5,5) node{$\mathcal{F} $};
		\fill[cyan, opacity=0.2] (-12,-3)--(-13,-3)--(-13,-2)--(-12,-2)--cycle;
		\fill[cyan, opacity=0.2] (-12,4)--(-13,4)--(-13,2)--(-12,2)--cycle;
		\fill[cyan, opacity=0.2] (-8,-3)--(-7,-3)--(-7,-2)--(-8,-2)--cycle;
		\fill[cyan, opacity=0.2] (-8,4)--(-7,4)--(-7,2)--(-8,2)--cycle;
		\draw [black ,ultra thick]     (-7.531,3) node{$\mathcal{E} $};
		\draw [black ,ultra thick]     (-12.5,-2.56) node{$\mathcal{E} $};
		\draw [black ,ultra thick]     (-12.5,3) node{$\mathcal{E} $};
		\draw [black ,ultra thick]     (-7.531,-2.56) node{$\mathcal{E} $};
		\fill[orange, opacity=0.2] (-8,4)--(-12,4)--(-12,2)--(-13,2)--(-13,-2)--(-12,-2)--(-12,-3)--(-8,-3)--(-8,-2)--(-7,-2)--(-7,2)--(-8,2)--cycle;
		\draw [black ,ultra thick]     (-10,0) node{$\mathcal{A} $};
		\fill[teal, opacity=0.2] (-16,-3)--(-16,3)--(-14,3)--(-14,4)--(-13,4)--(-13,-3)--cycle;
		\fill[teal, opacity=0.2] (-7,-3)--(-7,4)--(-6,4)--(-6,3)--(-4,3)--(-4,-3)--cycle;
		\draw [black ,ultra thick]     (-14,0) node{$\mathcal{H} $};
		\draw [black ,ultra thick]     (-5.5,0) node{$\mathcal{H} $};
		
	\end{tikzpicture}
\end{center}
\caption{The sets $\mathcal{A},\mathcal{B},\mathcal{C},\mathcal{D},\mathcal{E}, \mathcal{F}, \mathcal{G}, \mathcal{H}$}
\label{fig18062023}
\end{figure}

 We claim that $$\mathrm{P}(\mathcal{R}(E)) \geq \mathrm{P}(QR'),$$ where $\mathrm{P}(\mathcal{R}(E))=2 \varepsilon k_1+ 2 \varepsilon k_2$ and $\mathrm{P}(QR')=2 \varepsilon k_1+ 2 \varepsilon q+ 2 \varepsilon\chi_{\mathbb{N} \setminus \{0\}}(l)$.   We first assume that $l=0$. We observe that $\mathcal{R}(E)$ is a subset of a rectangle with horizontal sidelength $ \mathrm{P}_1(E)=\varepsilon k_1$, vertical sidelength $\mathrm{P}_2(E)=\varepsilon k_2$ and with the same barycenter of $ \mathcal{R}(E)$. We obtain 
$$ k_1 k_2 \varepsilon^2 \geq \vert \mathcal{R}(E)\vert=1 = \vert E \vert= N \varepsilon^2 =k_1 q \varepsilon^2  \implies k_2 \geq q. $$
Therefore $ \mathrm{P}(\mathcal{R}(E)) \geq \mathrm{P}(QR')$.
If otherwise $l \neq 0$, arguing as above, we observe that 
$$ k_2 k_1 \varepsilon^2 \geq\vert \mathcal{R}(E)\vert=1 = \vert E \vert= N \varepsilon^2= k_1 q \varepsilon^2 + l \varepsilon^2 \quad \text{and} \quad 0<l < k_1  \implies k_2>q .$$
Hence $ \mathrm{P}(\mathcal{R}(E)) \geq \mathrm{P}(QR')$. 
We claim
\begin{equation}\label{disdecres12321321}
D_{\varepsilon}(\mathcal{R}(E),\, QR) \geq D_{\varepsilon}(QR',\,QR).
\end{equation} 
 We define the following sets 
\begin{align*}
	& \mathcal{A}:= \varepsilon \mathbb{Z}^2 \cap \mathcal{R}(E) \cap QR 
	& \quad \mathcal{B}:= \varepsilon \mathbb{Z}^2 \cap \mathcal{R}(E)\cap QR^{c} \cap R^{'} \\
	 & \mathcal{C}:= \varepsilon \mathbb{Z}^2 \cap \mathcal{R}(E) \cap QR^{c} \cap Q^{'}  
	& \quad \mathcal{D}:= \varepsilon \mathbb{Z}^2 \cap \mathcal{R}(E) \cap QR^{c} \cap (QR')^c \\ 
	& 
	\mathcal{E}:= \varepsilon \mathbb{Z}^2 \cap R^{'} \cap \mathcal{R}(E)^{c} \cap QR 
	&  \quad \mathcal{F}:= \varepsilon \mathbb{Z}^2 \cap R^{'} \cap \mathcal{R}(E)^{c} \cap QR^{c} \\
	& \mathcal{G}:= \varepsilon \mathbb{Z}^2 \cap Q^{'} \cap QR^{c} \cap \mathcal{R}(E)^{c} & \quad \mathcal{H}:=\varepsilon \mathbb{Z}^2 \cap QR \cap (QR')^{c}.
\end{align*}
We have that $N= \# \varepsilon \mathbb{Z}^2 \cap \mathcal{R}(E)= \#\mathcal{A} +\#\mathcal{B} +\#\mathcal{C}+ \# \mathcal{D} $
and that $N= \# \varepsilon \mathbb{Z}^2 \cap QR'= \#\mathcal{A} +\#\mathcal{B}+\#\mathcal{C} +\# \mathcal{E} +\#\mathcal{F}+ \# \mathcal{G}.$
As a consequence $ \#\mathcal{D}= \#\mathcal{E} +\#\mathcal{F}+ \#\mathcal{G} $, hence there exists a bijection $f \colon \mathcal{D} \rightarrow \mathcal{E} \cup \mathcal{F} \cup \mathcal{G}$.
 We observe that 
if $\varepsilon(i,\,j) \in \mathcal{D} $ then
\begin{equation*}
	\begin{split}
	&f(\varepsilon(i,\,j)) \in QR  \text{ or } \\
	&f(\varepsilon(i,\,j)) \in QR^c \, \text{ and } d_{\infty}^{\varepsilon}(\varepsilon(i,\,j),\, \partial QR) \geq d_{\infty}^{\varepsilon}(f(\varepsilon(i,\,j)),\, \partial QR).
\end{split}
\end{equation*} 
The last inequality follows from $ f(\varepsilon(i,\,j)) \in QR'$, that is, if $ j>0$ (the case $j \leq 0$ being fully equivalent)
\begin{equation*}
	\begin{split}
		d_{\infty}^{\varepsilon}(\varepsilon(i,\,j),\, \partial QR) \geq&   \mathcal{H}^{1}(\Pi_{2}(QR' \cap \{0\}\times [0,\,+ \infty) ))
		\\
		&-\mathcal{H}^{1}(\Pi_{2}( QR\cap \{0\} \times [0,\,+ \infty)  )) \geq d_{\infty}^{\varepsilon}(f(\varepsilon(i,\,j)),\, \partial QR).
	\end{split}
\end{equation*}
We observe that
\begin{equation}\label{difsimmetrica123123}
\begin{split}
	\mathcal{R}(E) \Delta QR	=&  \bigcup_{\varepsilon(i,\,j) \in \mathcal{B}} Q_{\varepsilon}(\varepsilon(i,\,j)) \cup \bigcup_{\varepsilon(i,\,j) \in \mathcal{C}} Q_{\varepsilon}(\varepsilon(i,\,j))  \cup \bigcup_{\varepsilon(i,\,j) \in \mathcal{D}} Q_{\varepsilon}(\varepsilon(i,\,j))\\ & \quad \cup\bigcup_{\varepsilon(i,\,j) \in \mathcal{E}} Q_{\varepsilon}(\varepsilon(i,\,j)) \cup 
	\bigcup_{\varepsilon(i,\,j) \in \mathcal{H}} Q_{\varepsilon}(\varepsilon(i,\,j)).  \\
	QR' \Delta QR	=&  \bigcup_{\varepsilon(i,\,j) \in \mathcal{B}} Q_{\varepsilon}(\varepsilon(i,\,j)) \cup \bigcup_{\varepsilon(i,\,j) \in \mathcal{C}} Q_{\varepsilon}(\varepsilon(i,\,j)) \cup
	\bigcup_{\varepsilon(i,\,j) \in \mathcal{H}}  Q_{\varepsilon}(\varepsilon(i,\,j)) \\ & \quad \cup\bigcup_{\varepsilon(i,\,j) \in \mathcal{F}} Q_{\varepsilon}(\varepsilon(i,\,j)) \cup 
	\bigcup_{\varepsilon(i,\,j) \in \mathcal{G}} Q_{\varepsilon}(\varepsilon(i,\,j)). 
\end{split}
\end{equation}
Then by \eqref{difsimmetrica123123} we have
\begin{equation*}
\begin{split}
	& D_{\varepsilon}(\mathcal{R}(E),\,QR) 
	=\int_{\mathcal{R}(E) \Delta QR} d_{\infty}^{\varepsilon} (z,\, \partial QR)\, dz  \\
	&= \int_{\bigcup_{\varepsilon(i,\,j) \in \mathcal{B}} Q_{\varepsilon}(\varepsilon(i,\,j)) } d_{\infty}^{\varepsilon} (z,\,  \partial QR)\, dz + \int_{\bigcup_{\varepsilon(i,\,j) \in \mathcal{C}} Q_{\varepsilon}(\varepsilon(i,\,j)) } d_{\infty}^{\varepsilon} (z,\,  \partial QR)\, dz \\
	& \quad +  \int_{\bigcup_{\varepsilon(i,\,j) \in \mathcal{D}} Q_{\varepsilon}(\varepsilon(i,\,j)) } d_{\infty}^{\varepsilon} (z,\,  \partial QR)\, dz +  \int_{\bigcup_{\varepsilon(i,\,j) \in \mathcal{E}} Q_{\varepsilon}(\varepsilon(i,\,j)) } d_{\infty}^{\varepsilon} (z,\,  \partial QR)\, dz  \\
	& \quad +  \int_{\bigcup_{\varepsilon(i,\,j) \in \mathcal{H}} Q_{\varepsilon}(\varepsilon(i,\,j)) } d_{\infty}^{\varepsilon} (z,\,  \partial QR) \\
	&\geq \int_{\bigcup_{\varepsilon(i,\,j) \in \mathcal{B}} Q_{\varepsilon}(\varepsilon(i,\,j)) } d_{\infty}^{\varepsilon} (z,\,  \partial QR)\, dz + \int_{\bigcup_{\varepsilon(i,\,j) \in \mathcal{C}} Q_{\varepsilon}(\varepsilon(i,\,j)) } d_{\infty}^{\varepsilon} (z,\,  \partial QR)\, dz \\
	& \quad +  \int_{\bigcup_{\varepsilon(i,\,j) \in f(\mathcal{D})\setminus\mathcal{E}} Q_{\varepsilon}(\varepsilon(i,\,j)) } d_{\infty}^{\varepsilon} (z,\,  \partial QR)\, dz +  \int_{\bigcup_{\varepsilon(i,\,j) \in \mathcal{H}} Q_{\varepsilon}(\varepsilon(i,\,j)) } d_{\infty}^{\varepsilon} (z,\,  \partial QR)\, dz  \\
	& = \int_{\bigcup_{\varepsilon(i,\,j) \in \mathcal{B}} Q_{\varepsilon}(\varepsilon(i,\,j)) } d_{\infty}^{\varepsilon} (z,\,  \partial QR)\, dz +\int_{\bigcup_{\varepsilon(i,\,j) \in \mathcal{C}} Q_{\varepsilon}(\varepsilon(i,\,j)) } d_{\infty}^{\varepsilon} (z,\,  \partial QR)\, dz   \\
	& +\quad \int_{\bigcup_{\varepsilon(i,\,j) \in  \mathcal{F} \cup \mathcal{G}} Q_{\varepsilon}(\varepsilon(i,\,j)) } d_{\infty}^{\varepsilon} (z,\,  \partial QR)\, dz +  \int_{\bigcup_{\varepsilon(i,\,j) \in \mathcal{H}} Q_{\varepsilon}(\varepsilon(i,\,j)) } d_{\infty}^{\varepsilon} (z,\,  \partial QR)\, dz\\
	&
	=\int_{QR' \Delta QR} d_{\infty}^{\varepsilon} (z,\, \partial QR)\, dz = D_{\varepsilon}(QR',\,QR).
\end{split}
\end{equation*} 
which is the claim. Hence Case 3 is proven for $ \widetilde{QR}= QR'$. 
\end{proof}

\subsection{Flat $P_{|\cdot|_1}$ area-preserving mean-curvature flow of a quasi rectangle in $\varepsilon \mathbb{Z}^2$}\label{flat-eps}
	In this section we introduce the area-preserving mean-curvature flow of a quasi rectangle in the lattice $\varepsilon \mathbb{Z}^2$. We introduce a notion of global
flat solution to the area-preserving mean-curvature flow in the lattice $\varepsilon \mathbb{Z}^2$ which is analogous to the one introduced in Section \ref{Sec1}. To this end, given $E, F \in \mathcal{AD}_{\varepsilon}$ we define
\begin{equation}\label{29052023mat1}
	\mathcal{F}_{\varepsilon,\,\tau}(E,\,F):= \mathrm{P}(E)+ \frac{1}{\tau} D_{\varepsilon}(E,\,F).
\end{equation}

\begin{definition}\label{Deflatinc}
	Let $QR \in \mathcal{QR}_{\varepsilon}$ be a quasi-rectangle with pseudo-axial symmetry such that $\vert QR \vert =1$. Let $\{ E_{k}^{(\varepsilon,\tau)}\}_{k \in \mathbb{N}}$ be a family of sets defined iteratively as
	$$ E_0^{(\varepsilon,\tau)}=QR \quad \text{and} \quad E_{k }^{(\varepsilon,\tau)} \in \underset{E \in \mathcal{AD}_{\varepsilon}}{\arg \min} \left\{\mathcal{F}_{\varepsilon,\tau}(E,\,E_{k-1}^{(\varepsilon,\tau)})\right\} \quad k\geq 1. $$ 
	We define 
	$$ E_{t}^{(\varepsilon, \tau)}:= E_{k }^{(\varepsilon, \tau)} \quad \text{for any } t \in [k \tau,\, (k+1)\tau)$$
	and we call $\{E_t^{(\varepsilon,\tau)}\}_{t \geq 0}$ an approximate flat solution of the area-preserving  mean-curvature flow in the lattice $\varepsilon\mathbb{Z}^2$ with initial datum $QR$. 
\end{definition}
\begin{remark}
	Let $QR$ be a quasi rectangle with pseudo-axial symmetry.
	Thanks to Theorem \ref{thmprinc} we have that every minimizer of \eqref{29052023mat1} with $ E \in \mathcal{AD}_{\varepsilon}$ and $F= QR$ is a quasi-rectangle with pseudo-axial symmetry. Therefore, we have that an approximate flat solution of the area-preserving mean-curvature flow in the lattice $\varepsilon \mathbb{Z}^2$ with initial datum $QR$ is a family of quasi-rectangles with pseudo-axial symmetry.
\end{remark}
Let us consider a family of initial data, say $\{QR_{\varepsilon}\}_\e$, such that for every $\e$ the set $QR_{\varepsilon}$ is a quasirectangle with axial symmetry. We assume that, as $ \varepsilon \rightarrow 0$ $QR_\e$ converge in the Hausdorff distance to the rectangle $[-\frac{a}{2}, \frac{a}{2}]\times [-\frac{b}{2}, \frac{b}{2}]$. In what follows we are interested in proving the existence and give some geometric properties of a limit point, as $ \varepsilon,\tau \rightarrow 0$, of a subsequence of approximated flat solution of the area-preserving mean-curvature flow $E_t^{\varepsilon,\tau}$ whose initial datum is $QR_\e$. As already explained in the introduction of this paper, we focus on the regime $ \varepsilon \sim \tau$. More precisely we will compute the minimizer of the incremental problem in \eqref{29052023mat1} in a subset of $\mathcal{AD}_{\varepsilon}$ and we will describe the asymptotic behavior, as $ \varepsilon \rightarrow 0$, of the approximate flat solution of the area-preserving mean-curvature flow in the lattice $\varepsilon \mathbb{Z}^2$ within this subclass. For the sake of simplicity, in the sequel we assume that $ \tau= \alpha \varepsilon \text{ with } \alpha \in (0,+\infty).$
\begin{definition}\label{defSQR(QR)}
	Let $QR \in \mathcal{QR}_{\varepsilon}$ be a quasi rectangle with pseudo-axial symmetry. We assume that $QR= R \cup Q$ and that the horizontal sidelength of $Q$ is $\varepsilon r, \, r \in \N$ (see the definition \ref{quasirettangolo}). We define  $\mathcal{SQR}_{\varepsilon}(QR)$ as the set of all $QR' \in \mathcal{QR}_{\varepsilon}$ with  $QR'= R' \cup Q'$ such that the horizontal sidelength of $R'$ is $\varepsilon n, \, n \in \N$, $ \vert QR' \vert= \vert QR \vert$, $\mathrm{Bar}(R)= \mathrm{Bar}(R')$ and $r \leq n$.
\end{definition}
In the next theorem, given $QR\in\mathcal{QR}_{\varepsilon}$, we want to compute the minimizer of 
\begin{equation}\label{03062023pom1}
\min \left\{ \mathrm{P}(\widetilde{QR})+ \frac{1}{\alpha \varepsilon}D_{\varepsilon}(\widetilde{QR},\, QR) \colon \, \widetilde{QR} \in \mathcal{SQR}_{\varepsilon}(QR)\right\}.
\end{equation}

\begin{theorem}\label{THMesistenzadeiminimidis}
	Let $a,b,c \in \R$ be such that $0<b<a$, $0 \leq c < a$ and let us consider $\e>0$. We assume that $a=\varepsilon A, \,b= \varepsilon B, \, c= \varepsilon C$ with $A,B,C \in \N$ and $ab+ \varepsilon c=1$. Let $QR \in \mathcal{QR}_\varepsilon$ be the quasi-rectangle with sides $n=A$, $m=B$ and $r=C$ (see definition \ref{quasirettangolo}). For all $\alpha>0$ and for all $\Lambda>a+b$ there exists $\varepsilon_0:=\varepsilon_0(\Lambda, \alpha)$ such that for every $ \varepsilon < \varepsilon_0$ there exists a minimizer $QR'$ of \eqref{03062023pom1}.
Such a minimizer is a quasirectangle with pseudo-axial symmetry $QR^{'}= R' \cup Q' $ with $R'$ and $Q'$ characterized as follows. Setting 
\begin{equation}
\begin{split}
&x'=\frac{2\alpha(-b+a)}{b(a+b)}-\frac{a}{a+b},\\
&X_1= \lfloor x' \rfloor \text{, } X_2=\lceil x' \rceil,\, Y_1 = \bigg\lfloor \frac{ c  + 2  b X_1}{ 2 a -4 \varepsilon X_1 } \bigg\rfloor, Y_2 = \bigg\lfloor \frac{ c  + 2  b X_2 }{ 2 a -4 \varepsilon X_2 } \bigg\rfloor,
\end{split}
\end{equation}
it holds that
\begin{itemize}
\item[1)]
\begin{equation}
	\begin{split}
R' &= \left[-\frac{a'}{2}, \frac{a'}{2}\right] \times \left[-\frac{b'}{2}, \frac{b'}{2}\right]+ \mathrm{Bar}(R) \text{ where }\\
		 &a'\in \{a-2 X_1 \varepsilon,a-2  X_2 \varepsilon \}, \quad  b'\in \{ b+2 Y_1 \varepsilon, b+2 Y_2 \varepsilon \}, \\
\end{split}
\end{equation}		 
\item[2)] 
$$\textrm{the horizontal sidelength of } Q' \textrm{ is }d:= \varepsilon D \in\{ \varepsilon(C + 2X_i B - Y_i (2A -4 X_i)) \text{ for $i=1,2$}\}, $$
\item[3)]
\begin{eqnarray*}
&&a'+b'\leq a+b ,\, \vert a'-a \vert \leq \varepsilon C(\Lambda,\alpha),\,\\ &&\vert b'-b \vert \leq \varepsilon C(\Lambda,\alpha) \text{ and } \vert QR \Delta QR' \vert \leq \varepsilon C(\Lambda,\alpha).
\end{eqnarray*}
\end{itemize}
\end{theorem}
\begin{proof}
	In what follows all constants depend only on $\Lambda$ and $\alpha$.
	 We observe that the problem \eqref{03062023pom1} is equivalent to
	\begin{equation}
		\min \left\{ \alpha \varepsilon \mathrm{P}(\widetilde{QR})+ D_{\varepsilon}(\widetilde{QR},\, QR) \colon \, \widetilde{QR} \in \mathcal{SQR}_{\varepsilon}(QR)\right\}.
	\end{equation} 
To shorten notation we define 
\begin{equation}\label{15052023energia}
	E_\varepsilon(\widetilde{QR}):= \alpha \varepsilon \mathrm{P}(\widetilde{QR})+ D_{\varepsilon}(\widetilde{QR},\, QR).
\end{equation}
We divide the proof into several steps.\\
	\textit{Step 1)}\\
	In this step we want to prove that the minimum value and the minimizer of the energy $ E_\e$ can be obtained considering the minimum value and the minimizer of a function defined on $\N^3$. According to the definition of $\widetilde{QR} \in \mathcal{SQR}_{\varepsilon}(QR)$, setting $\widetilde{QR}=\widetilde{R}\cup\widetilde{Q}$, we have that $ \vert QR \vert= \vert \widetilde{QR} \vert $ and that $\mathrm{Bar}(R)= \mathrm{Bar}(\widetilde{R})$. Moreover there exist $X, \, Y \in \mathbb{N}$ such that
	the horizontal sidelength of $\widetilde{R}$ is $A-2X$, the vertical sidelength of $\widetilde{R}$ is $B-2Y$ and the horizontal sidelength of $\widetilde{Q}$ is $d:=\e D$ with $d < c$. Since  $ \vert QR \vert= \vert \widetilde{QR} \vert $, we have that
	\begin{equation}
		ab+c\varepsilon= (a-2X\varepsilon)(b+2Y \varepsilon)+\varepsilon d \; \implies \; \varepsilon Y= \frac{c\varepsilon-d \varepsilon +2 \varepsilon b}{2a-4\varepsilon X}.
	\end{equation} 
The perimeter of the set $\widetilde{QR}$ is 
\begin{equation}\label{perdiscTHM}
	\mathrm{P}(\widetilde{QR})=2(a-2\varepsilon X)+ 2(b+2\varepsilon Y)+ 2\varepsilon \chi_{(0,+\infty)}(d).
\end{equation}
We observe that the function $d \rightarrow 2\varepsilon \chi_{(0,+\infty)}(d)$ will not affect the calculation of the energy minimizer, as it only raises or lowers the energy by a constant amount.
By the symmetry of the set $\widetilde{QR} \Delta QR$, the dissipation is 

\begin{figure}[h!]
    \centering
    \begin{tikzpicture}[scale=0.5]
        \draw[red,ultra thick] 
        (-9,-4)--(-9,4)--(-6,4)--(-6,5)--(6,5)--(6,4)--(9,4)--(9,-4)-- cycle;
        \draw[blue,ultra thick] (-7,-5.5)--(-7,6)--(-1,6)--(-1,7)--(1,7)--(1,6)--(7,6)--(7,-5.5)--cycle;
        \fill[green!70!black, opacity=0.2] (-9,-4)--(-9,4)--(-7,4)--(-7,-4)--cycle;
        \fill[green!70!black, opacity=0.2] (9,-4)--(9,4)--(7,4)--(7,-4)--cycle;
        \fill[orange, opacity=0.2] (-7,-5.5)--(7,-5.5)--(7,-4)--(-7,-4)--cycle;
        \fill[purple, opacity=0.2] (-7,4)--(-7,6)--(-1,6)--(-1,7)--(1,7)--(1,6)--(7,6)--(7,4)--(6,4)--(6,5)--(-6,5)--(-6,4)--cycle;
        \draw [black,ultra thick]     (0,6) node{$C $};
        \draw [blue,ultra thick]      (2.1,7) node{$\partial \widetilde{QR} $};
        \draw [black,ultra thick]     (0,-4.7) node{$A $};
        \draw [black,ultra thick]     (8,0) node{$B $};
        \draw [black,ultra thick]     (-8,0) node{$B' $};
        \draw [red,ultra thick]       (-10,3) node{$\partial QR $};
    \end{tikzpicture}
    \caption{The quasi-rectangles $QR $ and $\widetilde{QR} $}
    \label{fig24062023mat1}
\end{figure}
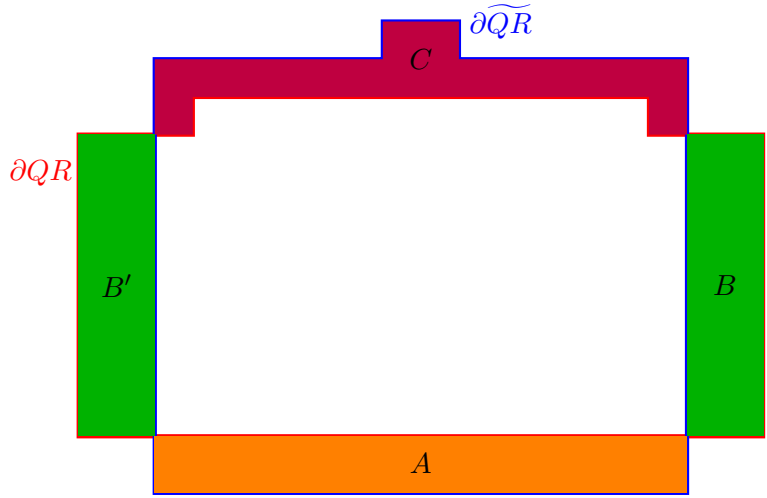

\begin{equation}
	\begin{split}
	\mathcal{D}_{\varepsilon}&(\widetilde{QR},QR)= \int_{\widetilde{QR}\Delta QR} d_{\infty}^{\varepsilon}(z, \partial QR)dz \\
	 &=\int_{A} d_{\infty}^{\varepsilon}(z, \partial QR)dz + \int_{B} d_{\infty}^{\varepsilon}(z, \partial QR)dz + \int_{B'} d_{\infty}^{\varepsilon}(z, \partial QR)dz+  \int_{C} d_{\infty}^{\varepsilon}(z, \partial QR)dz
\end{split}
\end{equation}
where
\begin{align}
&A= \left[-\frac{a-2\varepsilon X}{2}, \frac{a-2\varepsilon X}{2}\right]\times \left[-\frac{\varepsilon Y}{2}, \frac{\varepsilon Y}{2}\right]+\left(\mathrm{Bar}_x(R),\,\mathrm{Bar}_y(R)-\frac{b}{2}- \frac{\varepsilon Y}{2}\right) ,\\
& B= \left[- \frac{\varepsilon X}{2}, \frac{\varepsilon X}{2} \right] \times\left[-\frac{b}{2}, \frac{b}{2}\right]+ \left(\mathrm{Bar}_x(R)+ \frac{a+ \varepsilon X}{2}, \mathrm{Bar}_y(R)\right), \\
& B'= \left[ - \frac{\varepsilon X}{2}, \frac{\varepsilon X}{2} \right] \times\left[-\frac{b}{2}, \frac{b}{2}\right]+ \left(\mathrm{Bar}_x(R)- \frac{a+ \varepsilon X}{2}, \mathrm{Bar}_y(R)\right),\\
& C= \widetilde{QR} \setminus \left(QR \cup A \cup B \cup B' \right).
\end{align}
By a change of variables we have that 
\begin{equation}\label{15052023sera7}
	\begin{split}
		\mathcal{D}_{\varepsilon}&(\widetilde{QR},QR)= \int_{\widetilde{QR}\Delta QR} d_{\infty}^{\varepsilon}(z, \partial QR)dz \\
		&=\int_{A} d_{\infty}^{\varepsilon}(z, \partial QR)dz + 2\int_{B} d_{\infty}^{\varepsilon}(z, \partial QR)dz +  \int_{C} d_{\infty}^{\varepsilon}(z, \partial QR)dz.
	\end{split}
\end{equation}
We compute the different terms in the expression above. We have that
\begin{equation}\label{15052023sera6}
	\mathcal{D}_{A}:= \int_{A} d_{\infty}^{\varepsilon}(z, \partial QR)dz= \sum_{j=1}^{Y} j\varepsilon \varepsilon (a-2\varepsilon X)= (a-2\varepsilon X)\frac{\varepsilon Y (\varepsilon Y+\varepsilon)}{2}
\end{equation}
and
\begin{equation}
	\begin{split}
	\mathcal{D}_{C}:=& \int_{C} d_{\infty}^{\varepsilon}(z, \partial QR)dz= \int_{D} d_{\infty}^{\varepsilon}(z, \partial QR)dz+ \int_{E} d_{\infty}^{\varepsilon}(z, \partial QR)dz +\int_{E'} d_{\infty}^{\varepsilon}(z, \partial QR)dz \\
	&+ \int_{F} d_{\infty}^{\varepsilon}(z, \partial QR)dz +\int_{F'} d_{\infty}^{\varepsilon}(z, \partial QR)dz 
	+ \int_{G} d_{\infty}^{\varepsilon}(z, \partial QR)dz +\int_{G'} d_{\infty}^{\varepsilon}(z, \partial QR)dz
\end{split}
\end{equation}
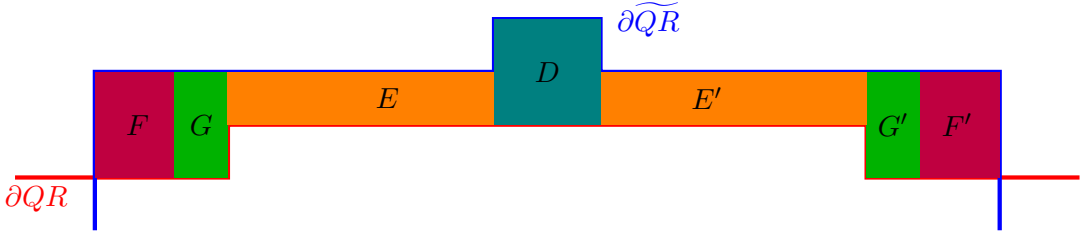
\begin{figure}[h!]
    \centering
    \begin{tikzpicture}[scale=0.7]
        \draw[red,ultra thick]
        (10,4)--(6,4)--(6,5)--(-6,5)--(-6,4)--(-10,4);
        \draw[blue,ultra thick]
        (-8.5,3)--(-8.5,6)--(-1,6)--(-1,7)--(1,7)--(1,6)--(8.5,6)--(8.5,3);
        \fill[purple, opacity=0.3] (-8.5,4)--(-8.5,6)--(-7,6)--(-7,4)--cycle;
        \draw [black,ultra thick]     (-7.7,5) node{$F $};
        \fill[purple, opacity=0.3] (8.5,4)--(8.5,6)--(7,6)--(7,4)--cycle;
        \draw [black,ultra thick]     (7.7,5) node{$F' $};
        \fill[green!70!black, opacity=0.3] (7,6)--(7,4)--(6,4)--(6,6)--cycle;
        \draw [black,ultra thick]     (6.5,5) node{$G' $};
        \fill[green!70!black, opacity=0.3] (-7,6)--(-7,4)--(-6,4)--(-6,6)--cycle;
        \draw [black,ultra thick]     (-6.5,5) node{$G $};
        \fill[orange, opacity=0.3] (-6,5)--(-6,6)--(-1,6)--(-1,5)--cycle;
        \draw [black,ultra thick]     (-3,5.5) node{$E$};
        \fill[orange, opacity=0.3] (6,5)--(6,6)--(1,6)--(1,5)--cycle;
        \draw [black,ultra thick]     (3,5.5) node{$E'$};
        \fill[teal, opacity=0.3] (1,7)--(1,5)--(-1,5)--(-1,7)--cycle;
        \draw [black,ultra thick]     (0,6) node{$D$};
        \draw [blue,ultra thick]      (1.9,7) node{$\partial \widetilde{QR}$};
        \draw [red,ultra thick]       (-9.6,3.6) node{$\partial QR$};
    \end{tikzpicture}
    \caption{The sets $F,F',G,G',E,E',D$}
    \label{fig24062023mat2}
\end{figure}

where
\begin{align}
	& D:= \left[ -\frac{d}{2}, \frac{d}{2}\right] \times  \left[ -\frac{\varepsilon Y}{2}, \frac{\varepsilon Y}{2}\right]+ \left(\mathrm{Bar}_x(R), \mathrm{Bar}_y(R)+ \frac{b}{2}+ \varepsilon + \frac{\varepsilon Y}{2}\right),  \\
	& E:= \left[-\frac{c-d}{2},\frac{c-d}{2} \right] \times \left[-\frac{\varepsilon Y-\varepsilon}{2}, \frac{\varepsilon Y-\varepsilon}{2} \right]+ \left( \mathrm{Bar}_x(R)-\frac{c+d}{4}, \mathrm{Bar}_y+\frac{b}{2}+\frac{\varepsilon Y+ \varepsilon}{2}\right), \\
	& E':= \left[-\frac{c-d}{2},\frac{c-d}{2} \right] \times \left[-\frac{\varepsilon Y-\varepsilon}{2}, \frac{\varepsilon Y-\varepsilon}{2} \right]+ \left( \mathrm{Bar}_x(R)+\frac{c+d}{4}, \mathrm{Bar}_y+\frac{b}{2}+ \frac{\varepsilon Y+ \varepsilon}{2}\right), \\
	& G:= \left[-\frac{\varepsilon}{2}, \frac{\varepsilon }{2}\right] \times \left[-\frac{\varepsilon Y}{2}, \frac{\varepsilon Y }{2}\right] + \left(\mathrm{Bar}_x(R)-\frac{c}{2}-\frac{\varepsilon}{2}, \mathrm{Bar}_y(R)+ \frac{b}{2}+ \frac{\varepsilon Y}{2}\right), \\
&	G':= \left[-\frac{\varepsilon}{2}, \frac{\varepsilon }{2}\right] \times \left[-\frac{\varepsilon Y}{2}, \frac{\varepsilon Y }{2}\right] + \left(\mathrm{Bar}_x(R)+\frac{c}{2}+\frac{\varepsilon}{2}, \mathrm{Bar}_y(R)+ \frac{b}{2}+ \frac{\varepsilon Y}{2}\right),
\end{align}
\begin{multline}
	F:= \left[-\frac{a-c-2\varepsilon X}{2},\frac{a-c-2\varepsilon X}{2}\right] \times \left[-\frac{\varepsilon Y}{2},\frac{\varepsilon Y}{2}\right] \\+ \left(\mathrm{Bar}_x(R)- \frac{a-2\varepsilon X+2\varepsilon+c}{4}, \mathrm{Bar}_y(R)+ \frac{b+\varepsilon Y}{2}\right)
\end{multline}
and 
\begin{multline}
	F':= \left[-\frac{a-c-2\varepsilon X}{2},\frac{a-c-2\varepsilon X}{2}\right] \times \left[-\frac{\varepsilon Y}{2},\frac{\varepsilon Y}{2}\right]
	\\+  \left(\mathrm{Bar}_x(R)  + \frac{a-2\varepsilon X+c+2\varepsilon}{4}, \mathrm{Bar}_y(R)+ \frac{b+\varepsilon Y}{2}\right).
\end{multline}
By a change of variables we obtain that
\begin{equation}\label{15052023sera5}
	\begin{split}
		\mathcal{D}_{C}=& \int_{C} d_{\infty}^{\varepsilon}(z, \partial QR)dz= \int_{D} d_{\infty}^{\varepsilon}(z, \partial QR)dz+ 2\int_{E} d_{\infty}^{\varepsilon}(z, \partial QR)dz \\
		&+2 \int_{F} d_{\infty}^{\varepsilon}(z, \partial QR)dz 
		+ 2\int_{G} d_{\infty}^{\varepsilon}(z, \partial QR)dz,
	\end{split}
\end{equation}

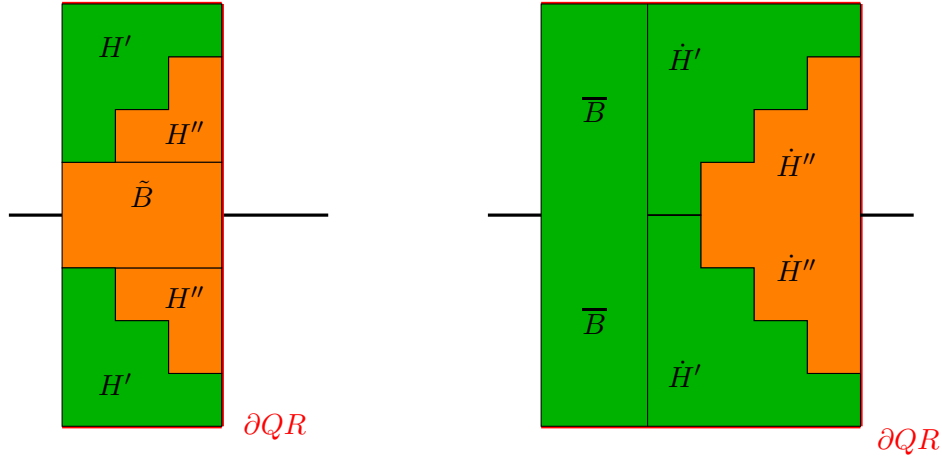
\begin{figure}[h!]
	\begin{center}
		\begin{tikzpicture}[scale=0.7]
			\draw[black, very thick]
			(-12,0) -- (-6,0);
			\draw[red,ultra thick] 
			(-11,-4) -- (-8,-4);
			\draw[red,ultra thick] 
			(-11,4) -- (-8,4);
			\draw[red,ultra thick]
			(-8,-4) -- (-8,4);
			\draw (-11,1) -- (-8,1);
			\draw (-11,-1) -- (-8,-1);
			\fill [orange, opacity=0.2]
			(-10,1)-- (-10,1)--(-10,2)--(-9,2)--(-9,3)--(-8,3)--(-8,1)--cycle;
			\draw 
			(-10,1)-- (-10,1)--(-10,2)--(-9,2)--(-9,3)--(-8,3)--(-8,1)--cycle;
			\fill [orange, opacity=0.2]
			(-10,-1)--(-10,-1)--(-10,-2)--(-9,-2)--(-9,-3)--(-8,-3)--(-8,-1)--cycle;
			\draw
			(-10,-1)--(-10,-1)--(-10,-2)--(-9,-2)--(-9,-3)--(-8,-3)--(-8,-1)--cycle;
			\fill [orange, opacity=0.2]
			(-11,1)--(-8,1)--(-8,-1)--(-11,-1)--cycle;
			\draw 
			(-11,1)--(-8,1)--(-8,-1)--(-11,-1)--cycle;
			\fill [green!70!black, opacity=0.2]
			(-11,4)--(-11,1)--(-10,1)--(-10,2)--(-9,2)--(-9,3)--(-8,3)--(-8,4)--cycle;
			\draw 
			(-11,4)--(-11,1)--(-10,1)--(-10,2)--(-9,2)--(-9,3)--(-8,3)--(-8,4)--cycle;
			\fill [green!70!black, opacity=0.2]
			(-11,-4)--(-11,-1)--(-10,-1)--(-10,-2)--(-9,-2)--(-9,-3)--(-8,-3)--(-8,-4)--cycle;
			\draw
			(-11,-4)--(-11,-1)--(-10,-1)--(-10,-2)--(-9,-2)--(-9,-3)--(-8,-3)--(-8,-4)--cycle;
			\draw [black,very thick]	
			(-10,-3.2) node{$H' $} ;
			\draw [black,very thick]	
			(-8.7,-1.55) node{$H'' $} ;
			\draw [black,very thick]	
			(-10,3.2) node{$H' $} ;
			\draw [black,very thick]	
			(-8.7,1.56) node{$H'' $} ;
			\draw [black,very thick]	
			(-9.5,0.4) node{$\tilde{B} $} ;
			\draw [red,very thick]	
			(-7,-4) node{$\partial QR $} ;
			\draw[red,ultra thick] 
			(-2,-4) -- (4,-4);
			\draw[red,ultra thick] 
			(4,-4) -- (4,4);
			\draw[red,ultra thick]
			(-2,4) -- (4,4);
			\draw[black,very thick]
			(-3,0) -- (5,0);
			 \fill[orange, opacity=0.2] 
			(1,0)--(1,1)--(2,1)--(2,2)--(3,2)--(3,3)--(4,3)--(4,-3)--(3,-3)--(3,-2)--(2,-2)--(2,-1)--(1,-1)--(1,0)--cycle; 
			\draw 
			(1,0)--(1,1)--(2,1)--(2,2)--(3,2)--(3,3)--(4,3)--(4,-3)--(3,-3)--(3,-2)--(2,-2)--(2,-1)--(1,-1)--(1,0)--cycle;
			 \fill[green!70!black, opacity=0.2] 
			(0,0)--(1,0)--(1,1)--(2,1)--(2,2)--(3,2)--(3,3)--(4,3)--(4,4)--(-2,4)--(-2,-4)--(4,-4)--(4,-3)--(3,-3)--(3,-2)--(2,-2)--(2,-1)--(1,-1)--(1,0)--(0,0)--cycle;
			\draw
			(0,0)--(1,0)--(1,1)--(2,1)--(2,2)--(3,2)--(3,3)--(4,3)--(4,4)--(-2,4)--(-2,-4)--(4,-4)--(4,-3)--(3,-3)--(3,-2)--(2,-2)--(2,-1)--(1,-1)--(1,0)--(0,0)--cycle;
			\draw (0,4)--(0,-4);
			\draw [black,very thick]	
			(-1,2) node{$\overline{B} $} ;
			\draw [black,very thick]	
			(-1,-2) node{$\overline{B} $} ;
			\draw [black,very thick]	
			(0.7,3) node{$\dot{H}' $} ;
			\draw [black,very thick]	
			(0.7,-3) node{$\dot{H}' $} ;
			\draw [black,very thick]	
			(2.8,1) node{$\dot{H}'' $} ;\draw [black,very thick]	
			(2.8,-1) node{$\dot{H}'' $} ;
			\draw [red,very thick]	
			(4.9,-4.3) node{$\partial QR $} ;
		\end{tikzpicture}
	\end{center}
	\caption{In the left side the case $ \e X \leq  \frac{b}{2}$ and in the right side the case $ \e X > \frac{b}{2}$ }
	\label{fig24062023pom1}
\end{figure}

where
\begin{align}\label{15052023sera4}
	\qquad& \int_{D} d_{\infty}^{\varepsilon}(z, \partial QR)dz= \sum_{j=1}^{Y} j \varepsilon^2 d= \frac{d \varepsilon^2 Y(Y+1)}{2}, \\
	 & \int_{E} d_{\infty}^{\varepsilon}(z, \partial QR)dz= \sum_{j=1}^{Y-1} \frac{c-d}{2}\varepsilon^2 j= \frac{(c-d)\varepsilon^2(Y-1)Y }{4},  \\
	&  \int_{F} d_{\infty}^{\varepsilon}(z, \partial QR)dz= \sum_{j=1}^{Y} \frac{(a-2\varepsilon X-c-2\varepsilon)}{2} \varepsilon^2 j= \frac{(a-2\varepsilon X-c-2\varepsilon)\varepsilon^2(Y+1)Y}{4}, \\
	& \int_{G} d_{\infty}^{\varepsilon}(z, \partial QR)dz= \sum_{j=1}^{Y} j \varepsilon^3= \varepsilon^3 \frac{Y(Y+1)}{2}.
\end{align}
In order to compute 
\begin{equation}
\mathcal{D}_B:=	2\int_{B} d_{\infty}^{\varepsilon}(z, \partial QR)dz
\end{equation}
we distinguish two cases, namely $\varepsilon X \leq \frac{b}{2}$ or $\varepsilon X > \frac{b}{2}$. \\

\textit{Case $\varepsilon X \leq \frac{b}{2}$ } By the symmetry of the quasi rectangle we can write
\begin{equation}\label{15052023pom1}
	\mathcal{D}_B:=	2\int_{B} d_{\infty}^{\varepsilon}(z, \partial QR)dz= 2 \int_{\tilde{B}}d_{\infty}^{\varepsilon}(z, \partial QR)dz
	+  4 \int_{\hat{B}}d_{\infty}^{\varepsilon}(z, \partial QR)dz  
\end{equation}
where 
\begin{align}
	& \tilde{B}:= \left[ -\frac{\varepsilon X}{2}, \frac{\varepsilon X}{2}\right] \times \left[ -\frac{b-2\varepsilon X}{2}, \frac{b-2\varepsilon X}{2}\right]+ \left(\mathrm{Bar}_x(R)+ \frac{2a -3\varepsilon X}{4}, \mathrm{Bar}_y(R)\right), \\
	& \hat{B}:= \left[ -\frac{\varepsilon X}{2}, \frac{\varepsilon X}{2}\right] \times \left[ -\frac{\varepsilon X}{2}, \frac{\varepsilon X}{2}\right]+ \left(\mathrm{Bar}_x(R)+ \frac{a+\varepsilon X}{2}, \mathrm{Bar}_y(R)+ \frac{b-\varepsilon X}{2}\right).
\end{align}
For the first term we have that
\begin{equation}\label{bitilde}
2	\int_{\tilde{B}}  d_{\infty}^{\varepsilon}(z, \partial QR)dz= 2 \varepsilon^2 (b-2\varepsilon X)\sum_{j=1}^{X}j= (b-2\varepsilon X)\varepsilon X (\varepsilon X+ \varepsilon).
\end{equation}

For the second term we first perform a change of variables and write that
\begin{equation}\label{15052023pom2}
	\int_{\hat{B}}d_{\infty}^{\varepsilon}(z, \partial QR)dz  = \int_{H}d_{\infty}^{\varepsilon}(z, L)dz  
\end{equation}
where 
\begin{equation}
	H:= \left[ 0, \varepsilon X\right] \times \left[ 0, \varepsilon X\right], \quad L:= \left[ 0, \varepsilon X\right] \times \left\{\varepsilon X\right\} \cup  \left\{\varepsilon X \right\} \times \left[ 0, \varepsilon X\right].
\end{equation}
We then simplify the computation by writing 
\begin{equation}\label{15052023pom3}
	\int_{H}d_{\infty}^{\varepsilon}(z, \partial QR)dz  = \int_{H'}d_{\infty}^{\varepsilon}(z, L)dz  + \int_{H''}d_{\infty}^{\varepsilon}(z, L)dz  
\end{equation}
where
\begin{equation}
	H':= \bigcup_{i=1}^{X} \bigcup_{j=1}^{X+1-j}Q_{\varepsilon}(\varepsilon(i,j)), \quad H'':= \bigcup_{j=1}^{X} \bigcup_{i=1}^{X-j}Q_{\varepsilon}(\varepsilon(i,j)).
\end{equation}
We eventually obtain
\begin{align}\label{15052023pom4}
& 4\int_{H}d_{\infty}^{\varepsilon}(z, \partial QR)dz   = 4 \int_{H'}d_{\infty}^{\varepsilon}(z, L)dz  + 4\int_{H''}d_{\infty}^{\varepsilon}(z, L)dz  \\
&  \quad =4 \sum_{i=1}^{X} \sum_{j=1}^{X-i+1} j \varepsilon^3 + 4 \sum_{j=1}^{X} \sum_{i=1}^{X-j} i \varepsilon^3= 4 \varepsilon^3 \sum_{k=1}^{X} \left(2 \sum_{h=1}^{X-k}h +X-k+1\right) \\
& \quad  = 4 \varepsilon^3 \sum_{k=1}^{X}(X-k+1)^2= \frac{2}{3} \varepsilon^3 X (2 X^2 +3 X +1).
\end{align}
Hence, thanks to \eqref{15052023pom1}, \eqref{bitilde}, \eqref{15052023pom2}, \eqref{15052023pom3}, \eqref{15052023pom4} we have
\begin{equation}\label{15052023sera3}
	\mathcal{D}_B= \frac{2}{3} \varepsilon^3 X (2 X^2 +3 X +1)+ (b-2\varepsilon X)\varepsilon X (\varepsilon X+ \varepsilon).
\end{equation}\\

\textit{Case $ \varepsilon X > \frac{b}{2}$ } We need to consider two sub-cases. The first one is when $b= \varepsilon B$ and  $B= 2B' $ with $ B,B' \in \N$, while the second one corresponds to 
$B=2B'+1$ with $B,B' \in \N$. In the first case, using the symmetries of the quasi rectangle we can write that 

\begin{equation}\label{15052023pom5}
 \mathcal{D}_B:=	2\int_{B} d_{\infty}^{\varepsilon}(z, \partial QR)dz= 4 \int_{\overline{B}}d_{\infty}^{\varepsilon}(z, \partial QR)dz
 +  4 \int_{\dot{B}}d_{\infty}^{\varepsilon}(z, \partial QR)dz 
\end{equation}
where
\begin{multline}
	\overline{B}:= \left[-\frac{(B'-X)\varepsilon}{2}, \frac{(B'-X)\varepsilon}{2} \right] \times \left[ - \frac{B' \varepsilon }{2}, \frac{ B' \varepsilon }{2}\right] \\+ \left( \mathrm{Bar}_x(R)+ \frac{a+ (B'-X)\varepsilon}{2}, \mathrm{Bar}_y(R)+ \frac{B' \varepsilon}{2}\right) 
\end{multline}
and
\begin{equation}
 \dot{B}:= \left[- \frac{B' \varepsilon}{2}, \frac{B' \varepsilon}{2}\right] \times \left[- \frac{B' \varepsilon}{2}, \frac{B' \varepsilon}{2}\right] + \left(\mathrm{Bar}_x +\frac{a+3 B' \varepsilon -2X \varepsilon}{2}, \mathrm{Bar}_y+ \frac{B' \varepsilon}{2}\right).
\end{equation}
For the first term we have that
\begin{equation}\label{15052023pom6}
	4 \int_{\overline{B}}d_{\infty}^{\varepsilon}(z, \partial QR)dz= 4 \varepsilon (X-B') \sum_{j=1}^{B'} \varepsilon^2 j= 2 (X-B')\varepsilon^3 B' (B'+1).
\end{equation}
For the second term, by a change of variables, it holds true that
\begin{equation}\label{15052023pom7}
	\int_{\dot{B}}d_{\infty}^{\varepsilon}(z, \partial QR)dz  = \int_{\dot{H}}d_{\infty}^{\varepsilon}(z, \dot{L})dz  
\end{equation}
where
\begin{equation}
	\dot{H}:= \left[0, B' \varepsilon\right] \times \left[0, B' \varepsilon\right] , \quad \dot{L}:= \left[0, B' \varepsilon\right] \times \left\{B' \varepsilon \right\} \cup \left\{B' \varepsilon \right\} \times \left[0, B' \varepsilon\right].
\end{equation}
We can further decompose the integral as follows
\begin{equation}\label{15052023pom8}
	\int_{\dot{H}}d_{\infty}^{\varepsilon}(z, \partial QR)dz  = \int_{\dot{H}'}d_{\infty}^{\varepsilon}(z, L)dz  + \int_{\dot{H}''}d_{\infty}^{\varepsilon}(z, L)dz  
\end{equation}
where
\begin{equation}
	\dot{H}':= \bigcup_{i=1}^{B'} \bigcup_{j=1}^{B'+1-j}Q_{\varepsilon}(\varepsilon(i,j)), \quad \dot{H}'':= \bigcup_{j=1}^{B'} \bigcup_{i=1}^{B'-j}Q_{\varepsilon}(\varepsilon(i,j)).
\end{equation}
We eventually obtain
\begin{align}\label{15052023pom9}
	& 4 \int_{\dot{H}'}d_{\infty}^{\varepsilon}(z, L)dz  + 4 \int_{\dot{H}''}d_{\infty}^{\varepsilon}(z, L)dz   =4 \sum_{i=1}^{B'} \sum_{j=1}^{B'-j+1} j \varepsilon^3 + 4 \sum_{j=1}^{B'} \sum_{i=1}^{b'-j}i \varepsilon^3 \\
	&= 4 \varepsilon^3 \sum_{k=1}^{B'} \left(2 \sum_{h=1}^{B'-k} h + B'-k +1\right)= \frac{2}{3} \varepsilon^3 \left(2 \left(B'\right)^2+3 B'+1\right)B'. 
\end{align}
Gathering together \eqref{15052023pom5},\eqref{15052023pom6}, \eqref{15052023pom7}, \eqref{15052023pom8}, \eqref{15052023pom9} we have
\begin{equation}\label{15052023sera2}
	\mathcal{D}_B=  \frac{b}{3} \left(2 \left(\frac{b}{2}\right)^2+ 3 \varepsilon \frac{b}{2}+ \varepsilon^2\right)+b(\frac{b}{2}+ \varepsilon)(\varepsilon X- \frac{b}{2}).
\end{equation}
We consider the second case, that is when $B=2B'+1$ with $B,B' \in \N$. By the symmetry of the quasi rectangle, we can write
\begin{equation}
	\begin{split}
	\mathcal{D}_B:=&	2\int_{B} d_{\infty}^{\varepsilon}(z, \partial QR)dz= 4 \int_{\ddot{B}}d_{\infty}^{\varepsilon}(z, \partial QR)dz
	 \\
	& +  4 \int_{\breve{B}}d_{\infty}^{\varepsilon}(z, \partial QR)dz + 2 \int_{V}d_{\infty}^{\varepsilon}(z, \partial QR)dz +2 \int_{W}d_{\infty}^{\varepsilon}(z, \partial QR)dz
\end{split}
\end{equation}
where
\begin{multline}
	\ddot{B}:= \left[-\frac{\varepsilon(X-B')}{2},\frac{\varepsilon(X-B')}{2}\right]\times \left[- \frac{\varepsilon B'}{2},\frac{\varepsilon B'}{2}\right]\\+\left(\mathrm{Bar}_x(R)+ \frac{a+X\varepsilon-\varepsilon B'}{2}, \mathrm{Bar}_y(R)+ \frac{\varepsilon+ \varepsilon B'}{2}\right)
\end{multline}
and
\begin{align}
	& \breve{B}:=\left[- \frac{\varepsilon B'}{2},\frac{\varepsilon B'}{2}\right] \times \left[- \frac{\varepsilon B'}{2},\frac{\varepsilon B'}{2}\right]+ \left(\mathrm{Bar}_x(R)+\frac{a+\varepsilon B'+2 \varepsilon X}{2}, \mathrm{Bar}_y(R)+ \frac{\varepsilon+ \varepsilon B'}{2}\right),\\
	& V:=\left[-\frac{\varepsilon(X-B')}{2},\frac{\varepsilon(X-B')}{2}\right]\times \left[- \frac{\varepsilon }{2},\frac{\varepsilon }{2}\right]+\left(\mathrm{Bar}_x(R)+ \frac{a+X\varepsilon-\varepsilon B'}{2}, \mathrm{Bar}_y(R)\right),\\
	& W:= \left[-\frac{\varepsilon B'}{2},\frac{\varepsilon B'}{2}\right]\times \left[- \frac{\varepsilon }{2},\frac{\varepsilon }{2}\right]+\left(\mathrm{Bar}_x(R)+ \frac{a+\varepsilon B'+2\varepsilon X}{2}, \mathrm{Bar}_y(R)\right) .
\end{align}
Arguing as before we show that
\begin{multline}
	\mathcal{D}_B= \frac{2}{3} \varepsilon^3 \left(2 \left(B'\right)^2+ 3 B'+1\right)B'\\+ 2\left(X-B'\right)\varepsilon^3 B' \left(B'+1\right)+ \varepsilon^3 B' \left(B'+1\right)+\left(X-B'\right)\left(B'+1\right)\varepsilon^3.
\end{multline}
Recalling that $ \varepsilon B' = \frac{b-\varepsilon}{2}$ we eventually obtain
\begin{equation}\label{15052023sera1}
	\begin{split}
	\mathcal{D}_B=& \frac{2}{3}  \left(2 \left(\frac{b-\varepsilon}{2}\right)^2+ 3 \varepsilon \frac{b-\varepsilon}{2}+\varepsilon^2\right)\frac{b-\varepsilon}{2}+ 2\left(X\varepsilon- \frac{b-\varepsilon}{2}\right) \frac{b-\varepsilon}{2} \left(\frac{b-\varepsilon}{2}+\varepsilon\right)\\
	&+ \varepsilon  \frac{b-\varepsilon}{2} \left(\frac{b-\varepsilon}{2}+\varepsilon\right)+\left(X\varepsilon-\frac{b-\varepsilon}{2}\right)\left(\frac{b-\varepsilon}{2}+\varepsilon\right)\varepsilon.
	\end{split}
\end{equation}
In order to conveniently write the energy functional we introduce the notation
$$ x:= \varepsilon X ,\quad y:= \varepsilon Y$$
and we rewrite the area constraint as $ y=\frac{c \varepsilon -\varepsilon d +2xb}{2a-4x}$. We are now in position to write the energy \eqref{15052023energia} as a function of $x$ and $d$. We still need to distinguish the two previous cases.

\textit{Case $x <\frac{b}{2}$ } By \eqref{perdiscTHM}, \eqref{15052023sera7}, \eqref{15052023sera6}, \eqref{15052023sera5}, \eqref{15052023sera4}, \eqref{15052023sera3} the dissipation part of the energy is
\begin{equation}
	\begin{split}
	\mathcal{D}(x,y,d):= \,&\mathcal{D}_{\varepsilon}(\widetilde{QR},QR)= \frac{2}{3}x (2x^2+3 \varepsilon x + \varepsilon^2)+(b-2x)x(x+\varepsilon) +(a-2x)\frac{y(y+\varepsilon)}{2} \\&+ \frac{d (y+\varepsilon)y}{2}+ \frac{y(y-\varepsilon)(c-d)}{2}+\varepsilon y (y+ \varepsilon)+ \frac{y(y+\varepsilon)}{2}(a-2x-c-2\varepsilon).
	\end{split}
\end{equation}
The energy functional becomes
\begin{equation}\label{28052023pom4}
	\begin{split}
	E_\varepsilon(x,y,d)=\,& \mathcal{D}(x,y,d)+ \varepsilon \alpha \left( -4x+4y\right) +R(d) 	= a \varepsilon y + a y^2  \\
	&-4 \alpha \varepsilon x+ 4 \varepsilon \alpha y + b \varepsilon x+ b x^2  -c \varepsilon y + d \varepsilon y + \frac{ 2 \varepsilon^2 x}{3}- 2 \varepsilon x y - \frac{ 2}{3} x^3 - 2 x y^2,
\end{split}
\end{equation}  
where $ R(d):= \,2a + 2b+ 2 \varepsilon \chi_{(0,+\infty)}(d)$. Now by the area constraint $ y=y(x)=\frac{c \varepsilon -\varepsilon d +2xb}{2a-4x}$ we obtain that 
\begin{equation}\label{28052023pom5}
	\begin{split}
E_\varepsilon(x,y(x),d):=\,& \frac{1}{12a-24x}\big(16x^4+(-24b-8a)x^3 \\ & +(-16\varepsilon^2+(96\alpha-48b)\varepsilon+12b^2+12ab)x^2\\
&+((12d-12c+8a)\varepsilon^2+((24a+48A)b-48\alpha a)\varepsilon)x \\
&+(-3d^2+(6c-6a-24\alpha)d-3c^2+(6a+24\alpha)c)\varepsilon^2 \big).
	\end{split}
\end{equation}
In what follows we are going to find the minimizers of our energy above, originally defined in the discrete setting, analyzing its extension in the continuum, that is for $x$ and $d$ real variables. We have that
\begin{equation}\label{parderxE}
	\begin{split}
		\frac{\partial E_\varepsilon(x,y(x),d)}{\partial x}=&\frac{-1}{(24x^2-24ax+6a^2)}\big(
		48x^4+(-48b-48a)x^3 \\&+(-16\varepsilon^2+(96\alpha-48b)\varepsilon+12b^2+48ab+12a^2)x^2 \\
		& +(16a\varepsilon^2+(48ab-96\alpha a)\varepsilon-12ab^2-12a^2 b)x \\
		& +(3d^2+(24\alpha-6c)d+3c^2-2\alpha c-4a^2)\varepsilon^2+((-12a^2-24\alpha a)b+24\alpha a^2)\varepsilon \big).
	\end{split}
\end{equation}
\textit{Case $x \geq \frac{b}{2}$ } Also in this step we consider the energy as defined on real numbers. We need to distinguish two sub-cases depending wether $\frac{b}{\varepsilon} \equiv_2 0$ or not. If $\frac{b}{\varepsilon} \equiv_2 0$, thanks to
\eqref{perdiscTHM}, \eqref{15052023sera7}, \eqref{15052023sera6}, \eqref{15052023sera5}, \eqref{15052023sera4}, \eqref{15052023sera2}, we have that
\begin{equation}\label{20052023sera1} 
	\begin{split}
		E_\varepsilon(x,y(x),d)=& \frac{1}{12a-24x} \big(\varepsilon x^2(48b-96\alpha)\\
		&+x((12c -12 d+8b)\varepsilon^2+((-24a-48\alpha)b+48\alpha a) \varepsilon -2b^3 -6 a b^2) \\
		&  + \varepsilon^2 (3 d^2+ (-6c+6a+24\alpha)d+ 3c^2+c(-6a-24 \alpha)-4ab) +ab^3\big)
	\end{split}
\end{equation} 
and 
\begin{equation}\label{20052023sera3}
	\begin{split}
		\frac{\partial E_\varepsilon(x,y(x),d)}{\partial x}=& \frac{1}{8 x^2 -8a x+ 2 a^2} \big((16 b-32 \alpha)\varepsilon x^2+ (32\alpha a-16 ab)\varepsilon x\\
		&+(-d^2+(2c-8 \alpha)d-c^2+8 \alpha c)\varepsilon^2  +((4a^2+8 \alpha a)b-8 \alpha a^2)\varepsilon +a^2 b^2\big).
	\end{split}
\end{equation}
If instead we are in the case $ \frac{b}{\varepsilon} \equiv_{2} 1$ we can use 
\eqref{perdiscTHM}, \eqref{15052023sera7}, \eqref{15052023sera6}, \eqref{15052023sera5}, \eqref{15052023sera4}, \eqref{15052023sera1} to write
\begin{equation}\label{20052023sera2}
	\begin{split}
		E_\varepsilon(x,y(x),d)=& \frac{1}{24 x-12 a} \big(x^2(12 \varepsilon^2+(48b-96\alpha)\varepsilon) \\& + 
		 x((-12d+12c+2b-6a)\varepsilon^2 +((-24a-48\alpha)b+48\alpha a)\varepsilon-2b^3+-6ab^2)\\
		 &+ \varepsilon^2(3d^2+(-6c+6a+24\alpha)d+3c^2+(-6a-24\alpha)c-ab)+ab^3\big)
	\end{split}
\end{equation} 
and
\begin{equation}\label{20052023sera4}
	\begin{split}
		\frac{\partial E_\varepsilon(x,y(x),d)}{\partial x}=& \frac{1}{8x^2-8ax+2a^2}\big(x^2(4\varepsilon^2 +(16b-32\alpha)\varepsilon)\\
		&+ (-4\varepsilon^2+(32\alpha a-16 ab)\varepsilon)x+(-d^2+(2c-8\alpha)d-c^2 +8\alpha c+a^2)\varepsilon^2 \\
		& +((4a^2 +8 \alpha a)b-8\alpha a^2)\varepsilon +a^2 b^2\big).
	\end{split}
\end{equation}
\\
\textit{Step 2)}\\
In this step we compute the minimzers and the minimal values of the function $ E_\varepsilon $ under the 
area constraint assuming that $x$ and $d$ are real variables. We distinguish two cases. \\

 \textit{Case $x \in [\frac{b}{2},\frac{a}{2})$ } In this case we have to consider the two energies in \eqref{20052023sera1} and \eqref{20052023sera2}. We claim that there exists $\varepsilon_1:= \varepsilon_1 (\Lambda,\alpha)$ such that for all $ \varepsilon< \varepsilon_1$
 \begin{equation}
 	\frac{\partial E_\varepsilon(x,y(x),d)}{\partial x}\geq0
 \end{equation}
for all $0\leq d < a$ fixed.
We detail the argument only in the case of formula \eqref{20052023sera3}, the other case being analogous. To this end we observe that the denominator of \eqref{20052023sera3} is positive. Hence the claim is proved provided one can show the positivity of the numerator in \eqref{20052023sera3}. The latter follows by the following chain of inequalities
\begin{equation}
	\begin{split}
	&(16 b-32 \alpha)\varepsilon x^2+ (32\alpha a-16 ab)\varepsilon x+(-d^2+(2c-8 \alpha)d-c^2+8 \alpha c)\varepsilon^2  \\
	&+((4a^2+8 \alpha a)b-8 \alpha a^2)\varepsilon +a^2 b^2> \frac{a^2b^2}{2}=\frac{1+\varepsilon^2 c^2 -2 \varepsilon c}{2}> \frac{1}{4}>0
\end{split}
\end{equation} 
which holds for $ \varepsilon< \varepsilon_1:= (\Lambda+\alpha)^{-k_1}$ with $k_1 \in \N$ sufficiently large and uses the assumption $ab+\e c=1$.
By the previous argument the only minimizer of $x \rightarrow E_\varepsilon(x,y(x),d)$ for $ x \in \left[\frac{b}{2},\frac{a}{2}\right) $ and $ a > d\geq0$ fixed is $ x= \frac{b}{2}$, where the energy takes the value 
\begin{equation}\label{28052023pom2}
	E_\varepsilon\left(\frac{b}{2}, y\left(\frac{b}{2}\right),d\right)=\frac{ab^3+b^4 +O(\varepsilon)}{12(a-b)}>0
\end{equation}
for all $ \varepsilon< \varepsilon_1$. The latter estimate can be extended to the function $(x,d) \rightarrow E_\varepsilon(x,y(x),d)$ thanks to the following observation. In formulas \eqref{20052023sera1}, \eqref{20052023sera2} the quantities $d$ and $d^2$ have $\varepsilon^2$ as a prefactor. Therefore it is possible to  write  $E_\varepsilon(x,y(x),d)= d \varepsilon^2 g(x)+ d^2 \varepsilon^2 h(x)+ k(x)$ with $g,h,k$ being rational functions of $x$. Note that the functions $g$ and $h$ have coefficients depending on $a,\,b,\,c$ only, while the function $k$ has coefficients depending on $a,\,b,\,c$ and $\varepsilon$. As a result there exists $ \varepsilon_2:= \varepsilon_2(\Lambda,\alpha)<\varepsilon_1$ such that for all $ \varepsilon < \varepsilon_2$ the minimal value of the energy $(x,d) \rightarrow E_\varepsilon(x,y(x),d)$ can be written as  $E_\varepsilon\left(\frac{b}{2}, y\left(\frac{b}{2}\right),d\right)=\frac{ab^3+b^4 +O(\varepsilon)}{12(a-b)}$.  \\

\textit{Case $x \in [0,\frac{b}{2})$ } We again look for stationary points of the energy $x \rightarrow E_\varepsilon(x, y(x),d)$ for fixed $ 0 \leq d < a$. To this end we investigate the zeros of the numerator of \eqref{parderxE} that we denote by
\begin{equation}
	\begin{split}
	f(\varepsilon,x):=& -\big(
	48x^4+(-48b-48a)x^3 \\&+(-16\varepsilon^2+(96\alpha-48b)\varepsilon+12b^2+48ab+12a^2)x^2 \\
	& +(16a\varepsilon^2+(48ab-96\alpha a)\varepsilon-12ab^2-12a^2 b)x \\
	& +(3d^2+(24\alpha-6c)d+3c^2-2\alpha c-4a^2)\varepsilon^2+((-12a^2-24\alpha a)b+24\alpha a^2)\varepsilon \big).
	 	\end{split}
\end{equation}
We observe that there exists $C(\Lambda,\alpha)>0$ such that
\begin{equation}\label{20052023pom1}
	\max_{x \in [0,\frac{b}{2}]} \vert f(\varepsilon,x)- f(0,x) \vert \leq \varepsilon C(\Lambda,\alpha)
\end{equation}
where 
\begin{equation}
	f(0,x)= -12x \left( -ab^2 -a^2 b +x(a^2+4ab+b^2)+(-4b-4a)x^2+4x^3 \right).
\end{equation}
Thanks to \eqref{20052023pom1} we find out that there exists $\varepsilon_3:=\varepsilon_3(\Lambda,\alpha) < \varepsilon_2$ such that for all $ \varepsilon< \varepsilon_3$ the points $(\varepsilon,x)$ such that $f(\varepsilon,x)=0$ lie in a neighbourhood of a zero of $f(0,x)$. For this reason we investigate the roots of the polynomial $f(0,x)$.
We have that
$$ f(0,x)= -12x p(x), \,\, p(x):=-ab^2 -a^2 b +x(a^2+4ab+b^2)+(-4b-4a)x^2+4x^3.$$
We observe that for all $x \in [0,\frac{b}{2}]$
$$ \frac{d}{d x}\,p(x)= 12x^2 +(-8b+8a)x+a^2+4ab+b^2>0.$$
Therefore, since $ p(\frac{b}{2})$ is negative we have that $ p(x)<0$ for all $ x \in [0,\frac{b}{2}]$. Hence the only root of $f(0,x)$ in $[0,\frac{b}{2}]$ is $x=0$ and moreover $ f(0,x)\geq0$ for all $ x \in [0,\frac{b}{2}]$. In order to obtain information on the zeros of $f(\e,x)$ from the zeros of $f(0,x)$ we will make use of the implicit function theorem. We first compute 
\begin{equation}\label{28052023pom3}
	\begin{split}
	\frac{\partial f}{\partial x}(\varepsilon,x)=&
	-192x^3+(144b+144a)x^2\\&+(32\varepsilon^2+(96b-192\alpha)-24b^2-96ab-24a^2)x\\&-16a\varepsilon^2+(96\alpha a-48ab)\varepsilon+12ab^2+12a^2b.
\end{split}
\end{equation}
From the expression above we obtain that
\begin{equation}\label{20052023pom2}
	\frac{\partial f}{\partial x}(0,0)= 12(a^2b+ab^2)>24>0
\end{equation}
where we have used that $ ab=1$ and $ a+b>2$. Moreover there exists $ C(\Lambda,\alpha)>0$ such that
\begin{equation}\label{20052023pom3}
	\max_{x \in [0,\frac{b}{2}]} \left| \frac{\partial f}{\partial x}(\varepsilon,x)- \frac{\partial f}{\partial x}(0,x)\right| \leq \varepsilon C(\Lambda,\alpha).
\end{equation}
As consequence of \eqref{20052023pom2}, \eqref{20052023pom3} there exists $ \varepsilon_4:=\varepsilon_4(\Lambda,\alpha)< \varepsilon_3$ such that for all $ \varepsilon<\varepsilon_4$ 
$$ \frac{ \partial f}{\partial x}(\varepsilon,x)>0 \text{ for all } \varepsilon< \varepsilon_4 \text{ and for all $ x \in (-\varepsilon_4,\varepsilon_4)$}.$$
We are now in a position to apply the implicit function theorem to obtain that for every $\alpha,a,b>0$ and $a+b< \Lambda$ there exists a smooth function $\bar{x}: (-\frac{\varepsilon_4}{2},\frac{\varepsilon_4}{2}) \rightarrow \R$ such that $ f(\varepsilon, \bar{x}(\varepsilon))=0$ and $ \bar x(0)=0$. As a consequence we have that for all $ \varepsilon<\varepsilon_4$
$$ f(\varepsilon,x) <0 \text{ for all }x\in (0,\bar{x}(\varepsilon)) \text{ and } f(\varepsilon,x)>0 \text{ for all } x \in \left(\bar{x}(\varepsilon), \frac{b}{2}\right).$$
Hence we obtain that the minimizer of $ x \rightarrow E_\varepsilon(x,y(x),d)$ is either $\bar{x}(\varepsilon)$ or $ \frac{b}{2}$. Now we want to compare the energy $E_\e$ at the points $(\bar x(\e),y(\bar x(\e), d)$ and $(\frac{b}{2}, y(\frac{b}{2}),d)$. To obtain a uniform estimate depending only on $\Lambda$ and $\alpha$ we first need to investigate the function $\varepsilon \rightarrow \bar{x}(\varepsilon)$ in a neighborhood of the origin. To this end we compute the first coefficient of its Maclaurin expansion. We can write that $\bar{x}(\varepsilon)= x' \varepsilon +o(\varepsilon)$ where $x'$ is obtained observing that, by \eqref{parderxE},
\begin{equation}
	\frac{\partial E_\varepsilon(\bar{x}(\varepsilon),y(\bar{x}(\varepsilon)),d)}{\partial x}=0 \,\iff\, o(\varepsilon)+ x' \varepsilon (12ab^2+12a^2b)+\varepsilon((12a^2+24 \alpha a)b-24\alpha a^2)=0,
\end{equation}
hence
\begin{equation}\label{280520223pom1}
	x'=\frac{2 \alpha(a-b)}{b(a+b)}- \frac{a}{a+b}= \lim_{\varepsilon \rightarrow 0^{+}}\frac{2 \alpha(a-b)}{b(a+b)}- \frac{a}{a+b}+ \frac{o(\varepsilon)}{\varepsilon (12ab^2+12a^2b)}.
\end{equation}
From the formula above we have that $\vert x' \vert \leq C(\Lambda,\alpha)$.
Computing the first and the second derivative of $\varepsilon \rightarrow \bar{x}(\varepsilon)$  we have that there exists $\varepsilon_5:= \varepsilon_5(\Lambda,\alpha)< \varepsilon_4$ such that for all $\varepsilon < \varepsilon_5$
\begin{equation*}
	\left| \frac{d}{d \varepsilon} \bar{x}(\varepsilon) \right| = \left| \frac{\frac{\partial f}{\partial \varepsilon}(\varepsilon,\bar{x}(\varepsilon))}{\frac{\partial f}{\partial x}(\varepsilon,\bar{x}(\varepsilon))}\right| \leq  C(\Lambda, \alpha)
\end{equation*}
and
\begin{equation}
	\left| \frac{d^2}{d \varepsilon^2} \bar{x}(\varepsilon) \right| = \left| \frac{-\frac{d}{d \varepsilon}\bar{x}(\varepsilon) \frac{\partial^2 f}{\partial x^2} (\varepsilon,\bar{x}(\varepsilon))+\frac{\partial^2 f}{\partial \varepsilon^2} (\varepsilon,\bar{x}(\varepsilon)) }{\frac{\partial f}{\partial x}(\varepsilon,\bar{x}(\varepsilon))} \right| \leq C(\Lambda,\alpha)
\end{equation} 
where we have used the formula \eqref{20052023pom3} to say that $ \left|\frac{\partial f}{\partial x}(\varepsilon,\bar x (\varepsilon))\right|>C(\Lambda,\alpha)$.
As a result we eventually have that $\bar x (\varepsilon)= x' \varepsilon + O(\varepsilon)$.
We now check that for $\e$ small enough the only minimizer of $ E_{\varepsilon}(x,y(x),d)$ is $\bar{x}(\varepsilon)$ (and not $ \frac{b}{2}$). Indeed, thanks to \eqref{28052023pom2}, we have that there exists $\varepsilon_6:= \varepsilon_6(\Lambda,\alpha)<\varepsilon_5$ such that for all $\varepsilon < \varepsilon_6$
\begin{equation}
	\begin{split}
	E_{\varepsilon}(\bar{x}(\varepsilon),y(\bar{x}(\varepsilon)),d)= & \frac{O(\varepsilon)}{12a+O(\varepsilon)} \\
	&< \frac{ab^3+b^4 }{12(a+b)} < \frac{ab^3+b^4 +O(\varepsilon)}{12(a-b)}=E_{\varepsilon}\left(\frac{b}{2}, y\left(\frac{b}{2}\right),d\right).
\end{split}
\end{equation}
Arguing as at the end of  \textit{Case $x \in [\frac{b}{2},\frac{a}{2})$}, one can find $ \varepsilon_7:= \varepsilon_7(\Lambda,\alpha)< \varepsilon_6$ such that for all $ \varepsilon < \varepsilon_7$ the minimal value of the energy $(x,d) \rightarrow E_\varepsilon(x,y(x),d)$ is  $E_{\varepsilon}(\bar{x}(\varepsilon),y(\bar{x}(\varepsilon)),d)=  \frac{O(\varepsilon)}{12a+O(\varepsilon)} $. \\

\textit{Step 3) } This last step is devoted to the computation of the minimizers of the energy \eqref{15052023energia} in $\mathcal{SQR}_{\varepsilon}(QR)$. Thanks to step 1) we know that this is equivalent to minimize the function
\begin{equation}\label{28052023sera1}
(X,Y,D) \rightarrow E_{\varepsilon}(\varepsilon X, \varepsilon Y, \varepsilon D)
\end{equation} 
where $X,Y,D \in \N$. 
By Step 2) we have that
$$E_{\varepsilon}(\bar x(\varepsilon), y(\bar x(\varepsilon)),d) \leq   E_\varepsilon \Big(\varepsilon X, \frac{c \varepsilon -d \varepsilon+ 2 \varepsilon
 b}{ 2a -4 \varepsilon X},d\Big)$$
for all $X \in \N $ such that $\varepsilon X \in (0,a)$, $\varepsilon Y=  \frac{c \varepsilon -d \varepsilon+ 2 \varepsilon
	b}{ 2a -4 \varepsilon X}$ with $Y \in \N$ and $d = \varepsilon D \leq a$ with $D \in \N$. 
 In step 2) we have proved that the function $ x \rightarrow E_\varepsilon (x,y(x),d)$ is monotone decreasing in $(0,\bar x(\varepsilon))$ and monotone increasing in $ (\bar x(\varepsilon),a)$. Hence the only possible values of $X \in \N$ that can minimize $E_\varepsilon$ are
 $ \lfloor x' \rfloor \text{,} \lceil x' \rceil $, where $x'$ is defined in \eqref{280520223pom1}. For such given $X$ we can exploit the definition of quasi rectangle to determine the values of $Y$ and $d$ which minimize the energy. We recall that, given a quasirectangle $QR=R\cup Q$, the horizontal sidelength of $Q$ is the reminder in the Euclidean division of $|QR|/\e^2$ (the number of lattice points in $QR$) by the number of lattice points in the horizontal sidelength of $R$. In our case this description corresponds to consider as optimal $Y \in \N$ associated to a minimal $X$, those values  for which the remainder $d= \frac{1}{\varepsilon} \left[ c \varepsilon +  2 X \varepsilon b - \varepsilon Y (2a -4 \varepsilon X) \right]$ takes the minimal value, namely
 \begin{equation}\label{28052023sera2}
 	Y_1 \in \arg \min  \left\{ \frac{1}{\varepsilon} \left( c \varepsilon +  2 \lfloor x' \rfloor  \varepsilon b - \varepsilon Y (2a -4 \varepsilon \lfloor x' \rfloor ) \right), \,  \varepsilon Y \in (0,+\infty)    \right\}
 \end{equation}
 and
 \begin{equation}\label{28052023sera3}
 	Y_2 \in \arg \min  \left\{ \frac{1}{\varepsilon} \left( c \varepsilon +  2 \lceil x' \rceil \varepsilon b - \varepsilon Y (2a -4 \varepsilon \lceil x' \rceil) \right), \,  \varepsilon Y \in (0,+\infty)     \right\}.
 \end{equation}
  We denote the candidate minimzers as
 \begin{equation}\label{03062023elenco1}
 	(\lfloor x' \rfloor , Y_1, D_1),\, (\lceil x' \rceil, Y_2, D_2) \in \N^3
 \end{equation} 
 where $d_1= \varepsilon D_1$ is the minimum value of \eqref{28052023sera2} and $d_2 = \varepsilon D_2$ is the minimum value of \eqref{28052023sera3}.
   By \eqref{28052023sera2} and \eqref{28052023sera3} we can also write 
  \begin{equation}\label{03062023elenco2}
  	Y_1 = \bigg\lfloor \frac{ c  + 2  b \lfloor x' \rfloor }{ 2 a -4 \varepsilon \lfloor x' \rfloor } \bigg\rfloor, Y_2 = \bigg\lfloor \frac{ c  + 2  b \lceil
  		 x' \rceil }{ 2 a -4 \varepsilon \lceil x' \rceil } \bigg\rfloor. 
  \end{equation} 
Exploiting again the area constraint we can also write for $ i =1,2$
 \begin{equation}\label{03062023elenco3}
 	D_i= \frac{1}{\varepsilon^2 } \left[ C \varepsilon^2 + 2 X_i \varepsilon^2 B - \varepsilon Y_i (2 \varepsilon A-4 \varepsilon X_i)\right]= C + 2X_i B - Y_i (2A -4 X_i)
 \end{equation} 
  where $ X_1:= \lfloor x' \rfloor $ and $ X_2 := \lceil x' \rceil$, $c=C \varepsilon $, $ a= \varepsilon A$ and $ b = \varepsilon B$ with $ A,\, B, \, C \in \N$ as in the assumptions of the theorem. We recall that in Step 3) we have proved that $\vert x' \vert \leq C(\Lambda,\alpha)$. As a result we have that $  Y_i   \leq C(\Lambda,\alpha)$ for $i=1,2$. Finally we obtain that 
  the minimizer $QR'=R' \cup Q'$ is such that $R'= \left[-\frac{a'}{2}, \frac{a'}{2}\right] \times \left[-\frac{b'}{2}, \frac{b'}{2}\right]+ \mathrm{Bar}(R)$ and
  \begin{equation}
  	\begin{split}
  	& a'= a-2X_i \varepsilon \text{ with } X_1= \lfloor x' \rfloor \text{ or } X_2=\lceil x' \rceil   \text{ and } x'=\frac{2\alpha(-b+a)}{b(a+b)}-\frac{a}{a+b},\\
  	&  b'= b+2 Y_i \varepsilon  \text{ with }  Y_1 = \bigg\lfloor \frac{ c  + 2  b \lfloor x' \rfloor }{ 2 a -4 \varepsilon \lfloor x' \rfloor }  \bigg\rfloor \text{ or } Y_2 = \bigg\lfloor \frac{ c  + 2  b \lceil
  		x' \rceil }{ 2 a -4 \varepsilon \lceil x' \rceil } \bigg\rfloor, \\
  	& \text{the horizontal sidelength of $Q'$ is } \varepsilon D_i=  \varepsilon(C + 2X_i B  + Y_i (2A -4 X_i)) \text{ for $i=1,2$}, \\
  	&   \vert a'-a \vert= \vert 2 X_i \varepsilon \vert \leq \varepsilon C(\Lambda,\alpha),\, \vert b'-b \vert = \vert 2 Y_i \varepsilon \vert  \leq \varepsilon C(\Lambda,\alpha) ,
  	\\
  	&  a'+b'\leq a+b ,\, \vert QR \Delta QR' \vert \leq  2bX_i \varepsilon+ 2(a-2X_i \varepsilon)Y_i \varepsilon \leq \varepsilon C(\Lambda,\alpha).
  \end{split}
  \end{equation}
   
\end{proof}
For all $ QR, \widetilde{QR} \in \mathcal{QR}_\varepsilon $ we define the energy 
$$ \mathcal{F}_{\varepsilon}(\widetilde{QR},QR):=\mathrm{P}(\widetilde{QR})+ \frac{1}{\alpha \varepsilon}D_{\varepsilon}(\widetilde{QR},\, QR). $$
\begin{definition}\label{def5.5}
	Let $\{QR_\varepsilon\}_{\varepsilon \in (0,1)} \subset \mathcal{QR}_{\varepsilon}$ have pseudo-axial symmetry, $\vert QR_\varepsilon \vert =1$ and such that $ QR_\varepsilon \rightarrow \left[-\frac{a}{2}, \frac{a}{2}\right] \times \left[-\frac{b}{2}, \frac{b}{2}\right]$ as $\varepsilon \rightarrow 0$ with respect to the Hausdorff distance. For all $\varepsilon\in(0,1)$ let $\{ QR_{k}^{\varepsilon}\}_{k \in \mathbb{N}}$ be a family of sets defined iteratively as
	$$ QR_0^{\varepsilon}=QR_\varepsilon \quad \text{and} \quad QR_{k }^{\varepsilon} \in \underset{QR \in \mathcal{SQR}_{\varepsilon}(QR_{k-1}^{\varepsilon})}{\arg \min} \left\{\mathcal{F}_{\varepsilon}(QR,\,QR_{k-1}^{\varepsilon})\right\} \quad k\geq 1. $$ 
	We define 
	$$ QR_{t}^{\varepsilon}:= QR_{k }^{\varepsilon} \quad \text{for any } t \in [k \alpha\varepsilon,\, (k+1)\alpha\varepsilon).$$
For all $\varepsilon$ we call $\{QR_t^{\varepsilon}\}_{t \geq 0}$ a symmetric approximate flat solution of the area-preserving  mean-curvature flow in the lattice $\varepsilon\mathbb{Z}^2$ with initial datum $QR_{\varepsilon}$. 
\end{definition}
\begin{theorem}\label{theorem:approxflat-eps}
	Let $R$ be a rectangle with horizontal sidelength $a_0$ and vertical sidelength $b_0$ and such that $a_0 b_0=1$ and $b_0 < a_0$ and $a_0+b_0<\Lambda$.  For any $ 0<\varepsilon< \varepsilon_0$ (where $\varepsilon_0$ is the number obtained in the Theorem \ref{THMesistenzadeiminimidis} for $QR= QR_{\varepsilon}$), let $ \{QR_t^{\varepsilon}\}_{t\geq 0}$ be a symmetric approximate flat
	solution to the area-preserving mean-curvature flow in the lattice $\varepsilon \mathbb{Z}^2$ with initial datum $QR_{\varepsilon}$ defined according to Definition \eqref{def5.5}.
	Then for all $k\geq 0$ we have $ QR_{t}^{\varepsilon}= R_{t}^{\varepsilon} \cup Q_{t}^{\varepsilon}$ where
$ R_t^{\varepsilon}= \left[-\frac{a^{\varepsilon}(k)}{2}, \frac{a^{\varepsilon}(k)}{2}\right] \times \left[-\frac{b^{\varepsilon}(k)}{2}, \frac{b^{\varepsilon}(k)}{2}\right] + \mathrm{Bar}(R^{\varepsilon})$ for all $ t \in [k\alpha\varepsilon,(k+1)\alpha\varepsilon) $.
	Here $a^{\varepsilon}(k)$ and $ b^{\varepsilon}(k)$ and are obtained by recurrence starting from $a^\varepsilon(0), b^\varepsilon(0)$ being the horizontal sidelength and the vertical sidelength of $R_{\varepsilon}$, respectively. We denote by $c^\varepsilon(k)$ the horizontal sidelength of $Q_t^{\varepsilon}$ for all $ t \in [k\alpha\varepsilon,(k+1)\alpha\varepsilon) $. Setting $c^\varepsilon(0)$ to be the horizontal sidelength of $Q_{\varepsilon}$ it holds true that, for all $k\geq 0$
\begin{equation}\label{03062023sera1}
	\begin{split}
		& X_1^\varepsilon(k)= \lfloor x^{\varepsilon}(k) \rfloor \text{, } X_2^{\varepsilon}(k)=\lceil x^{\varepsilon}(k) \rceil   \\
		& x^\varepsilon(k)=\frac{2\alpha(-b^{\varepsilon}(k)+a^{\varepsilon}(k))}{b^{\varepsilon}(k)(a^\varepsilon(k)+b^\varepsilon(k))}-\frac{a^\varepsilon(k)}{a^\varepsilon(k)+b^\varepsilon(k)},\\
		& 
		Y_1^\varepsilon(k) = \bigg\lfloor \frac{ c^\varepsilon(k)  + 2  b^{\varepsilon}(k) \lfloor x^{\varepsilon}(k) \rfloor }{ 2 a^\varepsilon(k) -4 \varepsilon \lfloor x^{\varepsilon}(k) \rfloor } \bigg\rfloor, Y_2^\varepsilon(k) = \bigg\lfloor \frac{ c^\varepsilon(k)  + 2  b^\varepsilon(k) \lceil
			x^{\varepsilon}(k) \rceil }{ 2 a^\varepsilon(k) -4 \varepsilon \lceil x^{\varepsilon}(k) \rceil } \bigg\rfloor, \\
		&
		a^\varepsilon(k+1)\in \{a^\varepsilon(k)-2\e X_1^{\varepsilon}(k),a^\varepsilon(k)-2\e X_2^{\varepsilon}(k) \}, \\
		& b^\varepsilon(k+1)\in \{ b^\varepsilon(k)+2 Y_1^{\varepsilon}(k) \varepsilon, b^\varepsilon(k)+2 Y_2^{\varepsilon}(k) \varepsilon \},   \\
		&  a^\varepsilon(k+1)+b^\varepsilon(k+1)\leq a^\varepsilon(k)+b^\varepsilon(k) ,\, \vert a^\varepsilon(k+1)-a^\varepsilon(k) \vert \leq \varepsilon C(\Lambda,\alpha),\,\\
		& \vert b^\varepsilon(k+1)-b^\varepsilon(k) \vert \leq \varepsilon C(\Lambda,\alpha),\, \vert QR_k^\varepsilon \Delta QR_{k+1}^\varepsilon \vert \leq \varepsilon C(\Lambda,\alpha),\, c^\e(k)\leq \Lambda\\
		&c^\varepsilon(k+1) \in\{ c^\varepsilon(k) + 2X_i^\varepsilon(k) b^\varepsilon(k) - Y_i^\varepsilon(k) (2a^\varepsilon(k) -4 X_i^\varepsilon(k)\varepsilon) \quad\text{for }i=1,2 \}.
	\end{split}
\end{equation}
Setting $a^\varepsilon(t)= a^\varepsilon(\lfloor \frac{t}{\alpha\varepsilon} \rfloor)$, $b^\varepsilon(t)= b^\varepsilon(\lfloor \frac{t}{\alpha\varepsilon} \rfloor)$ and $c^\varepsilon(t)= c^\varepsilon(\lfloor \frac{t}{\alpha\varepsilon} \rfloor)$ we have that $QR_t^\varepsilon$ converges locally uniformly in time as $\varepsilon \rightarrow 0$ to a rectangle $R(t)= \left[-\frac{a(t)}{2}, \frac{a(t)}{2}\right] \times \left[-\frac{b(t)}{2}, \frac{b(t)}{2}\right]$ with $a(t)b(t)=1$. 
\end{theorem}
\begin{proof}
		Formula \eqref{03062023sera1} is a consequence of Theorem \ref{THMesistenzadeiminimidis} applyed iteratively to the sets $QR= QR_{k}^{\varepsilon}$.  By Theorem \ref{THMesistenzadeiminimidis} we have that the function $a^\varepsilon(t):= a^\varepsilon(\lfloor\frac{t}{\varepsilon}\rfloor)$ and $b^\varepsilon(t)= b^\varepsilon(\lfloor\frac{t}{\varepsilon}\rfloor)$ satisfy 
$$
\vert a^\varepsilon(t)-a^\varepsilon(s)\vert\leq C(\Lambda,\alpha)\vert t-s+\e\vert,\quad
\vert b^\varepsilon(t)-b^\varepsilon(s)\vert\leq C(\Lambda,\alpha)\vert t-s+\e\vert
$$
for all $t,s>0$. Therefore there exist Lipschitz functions $a$ and $b$ obtained as locally uniform limit of $a^\varepsilon$ and $b^\varepsilon$. The fact that the limit rectangle has unitary area follows by passing to the limit as $\e\to 0$ in the area constraint formula
\begin{equation}
a^\e(t)b^\e(t)+\e c^\e(t)=1.
\end{equation}
\end{proof}
\begin{remark}\label{conjectured-flow}
It is not clear if one can provide a differential inclusion for the motion of the limit rectangle without assuming additional assumptions on the convergence of the reminders $t\mapsto c^\e(t)$. If one assumes the existence of the limit $c(t):= \lim_{\varepsilon \rightarrow 0} c^\varepsilon(t)$ locally uniformly in time (according to the Ascoli-Arzel\'a theorem this happens if $c^\e$ has a $\e$-uniform modulo of continuity on all compact subsets of $[0,+\infty)$), passing to the limit in \eqref{03062023sera1} as $\e\to 0$ one can prove that the following system of differential inclusions holds true. We set $x(t)= \frac{2\alpha}{b(t)}-\frac{4\alpha+a(t)}{a(t)+b(t)}$. If $ x(t) \notin \N $ then
 \begin{equation}
 	\begin{cases}\displaystyle
 	\frac{d a}{d t}  \in -\frac{2}{\alpha} \left[ \left\lfloor x(t)\right\rfloor \left\lceil x(t)  \right\rceil \right], \\ \\ \displaystyle
 	\frac{d b}{d t}  \in \frac{c(t)}{\alpha\, a(t)}+\frac{2b(t)}{\alpha\, a(t)} \left[ \left\lfloor x(t)\right\rfloor \left\lceil x(t)  \right\rceil \right],
 \end{cases}
 \end{equation}
If $ x(t) \in \N $ then
\begin{equation}
	\begin{cases}\displaystyle
	\frac{d a}{d t}  \in -\frac{2}{\alpha} \left[ x(t)+1, x(t)-1\right], \\ \\ \displaystyle
 	\frac{d b}{d t}  \in \frac{c(t)}{\alpha\, a(t)}+\frac{2b(t)}{\alpha\, a(t)}  \left[ x(t)-1, x(t)+1\right],
\end{cases}
\end{equation}
\end{remark}
In what follows we overcome the existence problem pointed out in the previous remark introducing a different notion of approximate flat solution, that we call rectangle approximate flat solution. The algorithm we propose does not keep the area fixed at each time step, but still produces an area-preserving flow in the limit as $\e\to 0$.
\begin{definition}\label{def05022024}
	Let $\{R^\varepsilon\}_{\varepsilon \in (0,1)} $ by a family of rectangles with $\vert R^\varepsilon \vert =1$ such that $ R^\varepsilon \rightarrow \left[-\frac{a}{2}, \frac{a}{2}\right] \times \left[-\frac{b}{2}, \frac{b}{2}\right]$ as $\varepsilon \rightarrow 0$ with respect to the Hausdorff distance. For all $\varepsilon\in(0,1)$ let $\{ R_{k}^{\varepsilon}\}_{k \in \mathbb{N}}$ be a family of rectangles defined iteratively according to the following scheme. Set $QR_0^{\varepsilon}=R^\varepsilon$ and define
	\begin{equation}\label{05022024pom1}
	R_k^\varepsilon :  \quad Q^\varepsilon_k \cup R^\varepsilon_k=QR_{k }^{\varepsilon} \in \underset{QR \in \mathcal{SQR}_{\varepsilon}(QR_{k-1}^{\varepsilon})}{\arg \min} \left\{\mathcal{F}_{\varepsilon}(QR,\,QR_{k-1}^{\varepsilon})\right\} \quad k\geq 1.
	\end{equation}   
	We finally set  
	$$ R_{t}^{\varepsilon}:= R_{k }^{\varepsilon} \quad \text{for any } t \in [k \alpha\varepsilon,\, (k+1)\alpha\varepsilon)$$
and for all $\varepsilon$ we call $\{R_t^{\varepsilon}\}_{t \geq 0}$ a rectangular approximate flat solution of the area-preserving  mean-curvature flow in the lattice $\varepsilon\mathbb{Z}^2$ with initial datum $R_{\varepsilon}$. 
\end{definition}
\begin{theorem}
	Let $R$ be a rectangle with horizontal sidelenght $a_0$ and vertical sidelenght $b_0$ and such that $a_0 b_0=1$ and $b_0 < a_0$ and $a_0+b_0<\Lambda$.  For any $ 0<\varepsilon< \varepsilon_0$ ($\varepsilon_0$ as in Theorem \ref{THMesistenzadeiminimidis} for $QR= QR_{\varepsilon}$), let $ \{R_t^{\varepsilon}\}_{t\geq 0}$ be a rectangular approximate flat
	solution to the area-preserving mean-curvature flow in the lattice $\varepsilon \mathbb{Z}^2$ with initial datum $R_{\varepsilon}$ defined according to Definition \eqref{def05022024}.
	Then for all $k\geq 0$ we have 
	$ R_t^{\varepsilon}= \left[-\frac{a^{\varepsilon}(k)}{2}, \frac{a^{\varepsilon}(k)}{2}\right] \times \left[-\frac{b^{\varepsilon}(k)}{2}, \frac{b^{\varepsilon}(k)}{2}\right] + \mathrm{Bar}(R^{\varepsilon})$  for all $ t \in [k\alpha\varepsilon,(k+1)\alpha\varepsilon) $.
	Here $a^{\varepsilon}(k)$ and $ b^{\varepsilon}(k)$ are obtained by recurrence starting from $a^\varepsilon(0), b^\varepsilon(0)$ being the horizontal and the vertical sidelenght of $R_{\varepsilon}$, respectively.  For all $k\geq 0$ it holds
	\begin{equation}\label{03062023sera1asd}
		\begin{split}
			& X_1^\varepsilon(k)= \lfloor x^{\varepsilon}(k) \rfloor \text{, } X_2^{\varepsilon}(k)=\lceil x^{\varepsilon}(k) \rceil   \\
			& x^\varepsilon(k)=\frac{2\alpha(-b^{\varepsilon}(k)+a^{\varepsilon}(k))}{b^{\varepsilon}(k)(a^\varepsilon(k)+b^\varepsilon(k))}-\frac{a^\varepsilon(k)}{a^\varepsilon(k)+b^\varepsilon(k)},\\
			& 
			Y_1^\varepsilon(k) = \bigg\lfloor \frac{  2  b^{\varepsilon}(k) \lfloor x^{\varepsilon}(k) \rfloor }{ 2 a^\varepsilon(k) -4 \varepsilon \lfloor x^{\varepsilon}(k) \rfloor } \bigg\rfloor, Y_2^\varepsilon(k) = \bigg\lfloor \frac{  2  b^\varepsilon(k) \lceil
				x^{\varepsilon}(k) \rceil }{ 2 a^\varepsilon(k) -4 \varepsilon \lceil x^{\varepsilon}(k) \rceil } \bigg\rfloor, \\
			&
			a^\varepsilon(k+1)\in \{a^\varepsilon(k)-2\e X_1^{\varepsilon}(k),a^\varepsilon(k)-2\e X_2^{\varepsilon}(k) \}, \\
			& b^\varepsilon(k+1)\in \{ b^\varepsilon(k)+2 Y_1^{\varepsilon}(k) \varepsilon, b^\varepsilon(k)+2 Y_2^{\varepsilon}(k) \varepsilon \},   \\
			&  a^\varepsilon(k+1)+b^\varepsilon(k+1)\leq a^\varepsilon(k)+b^\varepsilon(k) ,\, \vert a^\varepsilon(k+1)-a^\varepsilon(k) \vert \leq \varepsilon C(\Lambda,\alpha),\,\\
			& \vert b^\varepsilon(k+1)-b^\varepsilon(k) \vert \leq \varepsilon C(\Lambda,\alpha),\, \vert QR_k^\varepsilon \Delta QR_{k+1}^\varepsilon \vert \leq \varepsilon C(\Lambda,\alpha),\, c^\e(k)\leq \Lambda.
		\end{split}
	\end{equation}
	Setting $a^\varepsilon(t)= a^\varepsilon(\lfloor \frac{t}{\alpha\varepsilon} \rfloor)$, $b^\varepsilon(t)= b^\varepsilon(\lfloor \frac{t}{\alpha\varepsilon} \rfloor)$  we have that $R_t^\varepsilon$ converges locally uniformly in time as $\varepsilon \rightarrow 0$ to a rectangle $R(t)= \left[-\frac{a(t)}{2}, \frac{a(t)}{2}\right] \times \left[-\frac{b(t)}{2}, \frac{b(t)}{2}\right]$ with $a(t)b(t)=1$. Moreover the following system of differential inclusion holds true. Setting $x(t)= \frac{2\alpha}{b(t)}-\frac{4\alpha+a(t)}{a(t)+b(t)}$, if $ x(t) \notin \N $ then
 \begin{equation}\label{05022024pom2asd}
 	\begin{cases}\displaystyle
 	\frac{d a}{d t}  \in -\frac{2}{\alpha} \left[ \left\lfloor x(t)\right\rfloor \left\lceil x(t)  \right\rceil \right], \\ \\ \displaystyle
 	\frac{d b}{d t}  \in \frac{2b(t)}{\alpha\, a(t)} \left[ \left\lfloor x(t)\right\rfloor \left\lceil x(t)  \right\rceil \right],
 \end{cases}
 \end{equation}
if $ x(t) \in \N $ then
\begin{equation}\label{05022024pom3}
	\begin{cases}\displaystyle
	\frac{d a}{d t}  \in -\frac{2}{\alpha} \left[ x(t)+1, x(t)-1\right], \\ \\ \displaystyle
 	\frac{d b}{d t}  \in \frac{2b(t)}{\alpha\, a(t)}  \left[ x(t)-1, x(t)+1\right],
\end{cases}
\end{equation}
%
%
%
%
\end{theorem}
\begin{proof}
	By the definition \ref{def05022024} we have that the formula \eqref{03062023sera1asd} is a consequence of Theorem \ref{THMesistenzadeiminimidis} applyed iteratively to the sets $QR= QR_{k}^{\varepsilon}$.  By Theorem \ref{THMesistenzadeiminimidis} we have that the function $a^\varepsilon(t):= a^\varepsilon(\lfloor\frac{t}{\varepsilon}\rfloor)$ and $b^\varepsilon(t)= b^\varepsilon(\lfloor\frac{t}{\varepsilon}\rfloor)$ satisfy 
	$$
	\vert a^\varepsilon(t)-a^\varepsilon(s)\vert\leq C(\Lambda,\alpha)\vert t-s+\e\vert,\quad
	\vert b^\varepsilon(t)-b^\varepsilon(s)\vert\leq C(\Lambda,\alpha)\vert t-s+\e\vert
	$$
	for all $t,s>0$. Therefore there exist Lipschitz functions $a$ and $b$ obtained as locally uniform limit of $a^\varepsilon$ and $b^\varepsilon$. Let
$c^\varepsilon(t):= \mathrm{P}_1(Q_k^\varepsilon)$ for all $ t \in [k\alpha \varepsilon, (k+1)\alpha \varepsilon) $ where $Q_k^\varepsilon$ is define in \eqref{05022024pom1} and $\mathrm{P}_1$ is the function of the definition \eqref{P12}.
The fact that the limit rectangle has unitary area is a consequence of
$a^\e(t)b^\e(t)+\e c^\e(t)=1$ and then
\begin{equation}
	a^\e(t)b^\e(t)=\lim_{\varepsilon \rightarrow 0} a^\e(t)b^\e(t)=	\lim_{\varepsilon \rightarrow 0}a^\e(t)b^\e(t)+\e c^\e(t)=1
\end{equation}
where we have used that $c^\varepsilon (t) \leq \Lambda$. In the end passing to the limit in \eqref{03062023sera1asd} as $\e\to 0$ one can prove \eqref{05022024pom2asd} and \eqref{05022024pom3}.
\end{proof}

\begin{remark}\label{ultimopomeriggio}
It is worth observing that our system of differential inclusions never reduces to a single system of ODEs. As a result, the pinning phenomenon described in \cite{BGN}, according to which the flow is given by the family of sets identically equal to the initial datum, can only be one of the possible motions. This phenomenon is characterized by the following condition on the initial data $a$ and $b$
\begin{eqnarray*}
0\leq x(0)<1 \text{ and }a\cdot b=1\iff \frac{b^3}{2(1-b^2)}<\alpha.
\end{eqnarray*}
\end{remark}




\begin{thebibliography}{99}

\bibitem{ABC} R. Alicandro, A.Braides, and M.Cicalese,
Phase and anti-phase boundaries in binary discrete systems: a variational viewpoint,
{\em Netw. Heterog. Media}, {\bf 1} (2006), 85--107.

\bibitem{ABCS} R. Alicandro, A.Braides, M.Cicalese and M. Solci,
{\em Discrete Variational Problems with Interfaces}, Cambridge Monographs on Applied and Computational Mathematics, Cambridge University Press, Cambridge, (2023).

\bibitem{ACR} R. Alicandro, M.Cicalese and M. Ruf. Domain formation in magnetic polymer composites: an approach via stochastic homogenization,
{\em Arch. Ration. Mech. Anal.}, {\bf 218} (2015), 945--984.


\bibitem{ATW} F. Almgren, J.E. Taylor, and L. Wang,
Curvature driven flows: a variational approach,
{\em SIAM J. Control Optim.}, {\bf 50} (1983), 387--438.


\bibitem{BCCN} G. Bellettini, V. Caselles, A. Chambolle and M. Novaga, The volume preserving crystalline mean curvature 
flow of convex sets in $\R^N$, {\em J. Math. Pure Appl.} {\bf 92} (5), (2009), 499--527.



\bibitem{BC}
A. Braides and M. Cicalese,
Interfaces, modulated phases and textures in lattice systems,
{\em Arch. Ration. Mech. Anal.}, {\bf 223} (2017), 977--1017.

\bibitem{BCY}
A. Braides, M. Cicalese and N. K. Yip,
Crystalline motion of interfaces between patterns,
{\em J. Stat. Phys.}, {\bf 165} (2016), no.2, 274--319.

\bibitem{BGN}
 A. Braides, M.S. Gelli and M. Novaga, Motion and pinning of discrete interfaces,
 {\em Arch. Ration. Mech. Anal.}, {\bf 95} (2010), 469--498.

\bibitem{BMN}
 A. Braides, A. Malusa, and M. Novaga, Crystalline evolutions with rapidly oscillating forcing terms,
 {\em Ann. Sc. Norm. Super. Pisa Cl. Sci.}, (5) {\bf 20} (2020), no. 1, 143--175.
 
\bibitem{BSc} A. Braides and G. Scilla,
Motion of discrete interfaces in periodic media,
Interfaces Free Bound, {\bf 15} (2013), 451--476

\bibitem{BST} A. Braides, G. Scilla and A. Tribuzio, Nucleation and growth of lattice crystals,
J. Nonlinear Sci., {\bf 31}, (2021), Paper No. 97, 63.

\bibitem{BP} A. Braides and A. Piatnitski, Variational problems with percolation: dilute spin systems at zero temperature,
{\em J. Stat. Phys.}, {\bf 149} (2012), no.5, 846--864.

\bibitem{BSpenrose} A. Braides and M. Solci, Interfacial energies on Penrose lattices,
{\em Math. Models Methods Appl. Sci.}, {\bf 21} (2011), no. 5, 1193--1210.

\bibitem{BS} A. Braides and M. Solci,
Geometric flows on planar lattices, Birkh\"{a}user/Springer, Cham,
(2021), x+134.

\bibitem{CDL} L.A. Caffarelli and R. de la Llave,
Interfaces of Ground States in Ising Models with Periodic Coefficients,
{\em J. Stat. Phys.}, {\bf 118} (2005), 687--719.

\bibitem{CDGM} A. Chambolle, D. De Gennaro and M. Morini,
Discrete-to-continuous crystalline curvature flows, {\em preprint} (2024) available at https://cvgmt.sns.it/paper/6464/

\bibitem{Fusco2015} N. Fusco, The quantitative isoperimetric inequality and related topics,
{\em Bulletin of Mathematical Sciences}, {\bf 5} (2015) 517--607.

\bibitem{JMPS} V. Julin, M. Morini, M. Ponsiglione and E. Spadaro,
The asymptotics of the area-preserving mean curvature and the Mullins-Sekerka flow in two dimensions,
{\em Math. Ann.}, {\bf 387} (2023), no.3-4, 1969--1999.

%
\bibitem{LS} S. Luckhaus and T. Sturzenhecker,
Implicit time discretization for the mean curvature flow,
{\em Calc. Var.}, {\bf 3} (1995),  253--271.

\bibitem{Maggibook}F. Maggi, {\em Sets of finite perimeter and geometric variational problems: 
an introduction to Geometric Measure Theory} (No. 135). Cambridge University Press. (2012).

\bibitem{MN} A. Malusa and M. Novaga, Crystalline evolutions in chessboard-like microstructures
{\em  Netw. Heterog. Media}, {\bf 13} (2018), no. 3, 493--513.


\bibitem{MPS} M. Morini, M. Ponsiglione and E. Spadaro, 
Long time behavior of discrete volume preserving mean curvature flows,
{\em J. Reine Angew. Math.}, {\bf 784} (2022), 27--51.

\bibitem{MSS} L. Mugnai, C. Seis, and E. Spadaro,
Global solutions to the volume-preserving mean-curvature flow,
{\em Calc. Var. Partial Differential Equations}, {\bf 55} (2016), Art. 18, 23.



\bibitem{Taylor1} J.E. Taylor, J. Cahn, and C. Handwerker, Geometric Models
of Crystal Growth, {\em Acta Metall. Mater.}, {\bf 40} (1992), 1443--1474.

\bibitem{Taylor2} J.E. Taylor, J. Cahn, and C. Handwerker, Mean Curvature and 
Weighted Mean Curvature, {\em Acta Metall. Mater.}, {\bf 40} (1992), 1475--1485.

\end{thebibliography}
\end{document}